%% file: main.tex
\documentclass{article}
\usepackage{mathrsfs}
\usepackage[dvipsnames]{xcolor}
\usepackage[colorlinks=true, linkcolor=red, citecolor=blue, urlcolor=ForestGreen!70!black]{hyperref}
\usepackage[utf8]{inputenc}
\usepackage[letterpaper,margin=1.0in]{geometry}
\input{math.tex}
\usepackage{mathtools}
\usepackage{enumerate}
\usepackage{multirow}
\usepackage{algorithm}
\usepackage{algorithmic}
\usepackage{microtype}
\usepackage[numbers,sort]{natbib}
\usepackage{url}
\usepackage{graphicx}
\usepackage{pifont}
\usepackage{cases}
\usepackage{makecell}
\usepackage{cellspace}
\usepackage{cleveref}
\crefname{equation}{}{}

\usepackage{booktabs}

\newcommand{\hvg}{\hat{\vg}}
\newcommand{\hvs}{\hat{\vs}}
\newcommand{\hvy}{\hat{\vy}}
\newcommand{\hvd}{\hat{\vd}}

\newcommand{\hmB}{\hat{\mB}}

\newcommand{\hatm}{\hat{m}}
\newcommand{\hq}{\hat{q}}
\newcommand{\htheta}{\hat{\theta}}

\newcommand{\bmB}{\bar{\mB}}

\newcommand{\ttt}{\tilde{t}}
\newcommand{\tgamma}{\tilde{\gamma}}

\newcommand{\tM}{\tilde{M}}
\newcommand{\tmB}{\tilde{\mB}}

\newcommand{\thh}{\tilde{h}}
\newcommand{\tH}{\tilde{H}}

\def\SABFGS{{SA$^2$-BFGS}}

\newcommand{\diff}[2]{\frac{\mathrm{d}}{\mathrm{d} #2} #1}
\newcommand{\evaldiff}[3]{\left. \frac{\mathrm{d}}{\mathrm{d} #2} #1 \right|_{#2 = #3}}
\newcommand{\eval}[2]{\left. #1 \right|_{#2}}

\title{Explicit Global Convergence Rates of BFGS without Line Search}

\author{Jianjiang Yu\textsuperscript{1},\quad 
Weiguo Gao\textsuperscript{1,2,3},\quad 
and Luo Luo\textsuperscript{2} \\[2ex]
\textsuperscript{1}School of Mathematical Sciences, Fudan University, Shanghai, China \\
\textsuperscript{2}School of Data Science, Fudan University, Shanghai, China \\
\textsuperscript{3}Shanghai Key Laboratory of Contemporary Applied Mathematics, Shanghai, China \\[2ex]
\texttt{jjyu22@m.fudan.edu.cn, wggao@fudan.edu.cn, and luoluo@fudan.edu.cn}
} 

\date{\today \\}

\begin{document}

\maketitle

\begin{abstract}
This paper studies the convergence rates of the Broyden--Fletcher--Goldfarb--Shanno~(BFGS) method without line search.
We show that the BFGS method with an adaptive step size [Gao and Goldfarb, Optimization Methods and Software, 34(1):194-217, 2019] exhibits a two-phase non-asymptotic global convergence behavior when minimizing a strongly convex function, i.e., a linear convergence rate of $\mathcal{O}((1 - 1 / \varkappa)^{k})$ in the first phase and a superlinear convergence rate of $\mathcal{O}((\varkappa / k)^{k})$ in the second phase, where $k$ is the iteration counter and $\varkappa$ is the condition number.
In contrast, the existing analysis only establishes asymptotic results.
Furthermore, we propose a novel adaptive BFGS method without line search, which allows a larger step size by taking the gradient Lipschitz continuity into the algorithm design.
We prove that our method achieves faster convergence when the initial point is far away from the optimal solution. 
\end{abstract}

\section{Introduction}

This paper considers the unconstrained minimization problem 
\begin{equation}\label{prob:main}
    \min_{\vx \in \BR^{n}} f(\vx),
\end{equation}
where $f: \BR^{n} \to \BR$ is strongly convex and has both Lipschitz continuous gradient and Hessian. 
We focus on quasi-Newton methods which approximate the second-order information by a low-rank modification during iterations, avoiding the expensive computational cost of evaluating and inverting the exact Hessian. 
Popular quasi-Newton methods such as Broyden--Fletcher--Goldfarb--Shanno~(BFGS)~\cite{broydenQuasiNewtonMethodsTheir1967, fletcherNewApproachVariable1970, goldfarbFamilyVariablemetricMethods1970,shannoConditioningQuasiNewtonMethods1970}, Davidon--Fletcher--Powell~
(DFP)~\cite{davidonVariableMetricMethod1991, fletcherRapidlyConvergentDescent1963}, and Symmetric Rank-1~(SR1)~\cite{connConvergenceQuasiNewtonMatrices1991, khalfanTheoreticalExperimentalStudy1993}
have been extensively studied and widely applied in scientific computation~\cite{pfrommerRelaxationCrystalsQuasiNewton1997, schlegelExploringPotentialEnergy2003, virieuxOverviewFullwaveformInversion2009, liuQuasiNewtonMethodsRealtime2017, haberQuasiNewtonMethodsLargescale2004, goldfarbPracticalQuasiNewtonMethods2020, zdunekNonnegativeMatrixFactorization2006}. 

It is well-known that asymptotic superlinear local convergence rates of classical quasi-Newton methods can be achieved by the Dennis--Mor{\'e} characterization~\cite{dennisCharacterizationSuperlinearConvergence1974, griewankLocalConvergenceAnalysis1982, yuanModifiedBFGSAlgorithm1991},
whereas the explicit convergence rates were established only in recent years.
Specifically, \citet{jinNonasymptoticSuperlinearConvergence2023}, as well as \citet{rodomanovRatesSuperlinearConvergence2022, rodomanovNewResultsSuperlinear2021} were the first to provide non-asymptotic superlinear local convergence rates for classical BFGS and DFP methods.
Later, \citet{yeExplicitSuperlinearConvergence2023} established non-asymptotic results for the classical SR1 method.
In parallel, another line of research studied quasi-Newton methods with greedy and random directions, also showing provably explicit superlinear rates~\cite{rodomanovGreedyQuasiNewtonMethods2021, jinSharpenedQuasiNewtonMethods2022, linExplicitConvergenceRates2022, jiGreedyPSBMethods2023, liuSymmetricRank$k$Methods2024,liu2024incremental}. 
However, all the above results rely on the unit step size during iteration and only work when the initial point lies in a local neighborhood of the optimum.

Line search strategy is a popular way to establish global guarantees of quasi-Newton methods~\cite{powell1976some, byrdGlobalConvergenceClass1987,byrdToolAnalysisQuasiNewton1989}.
Recent analysis shows that incorporating line search steps into the quasi-Newton iteration can achieve explicit global convergence rates~\cite{krutikovConvergenceRateQuasiNewton2023, rodomanovGlobalComplexityAnalysis2024,  jinNonasymptoticGlobalConvergence2024b,jiangOnlineLearningGuided2023,jinNonasymptoticGlobalConvergence2025}.
In addition, the trust region~\cite{byrdAnalysisSymmetricRankone1996,jiGreedyPSBMethods2023} and cubic regularization methods~\cite{kamzolovCubicRegularizationKey2023,wangGlobalNonasymptoticSuperlinear2024} are also used to provide global convergence guarantees.
However, the iteration of the above strategies requires a subroutine to determine the step size or solve the sub-problem, leading to additional computational cost and complicated implementations.
In practice, we prefer simple and provable optimization algorithms.
The seminal work of \citet{gaoQuasiNewtonMethodsSuperlinear2019} proposed an adaptive step size rule for the BFGS method to achieve both a global convergence guarantee and local superlinear rates without line search, while their results are asymptotic.
Recently, \citet{wangGlobalNonasymptoticSuperlinear2024} introduced a correction strategy based on the gradient norm, achieving a modified SR1 method with global non-asymptotic convergence rates.

In this paper, we focus on the convergence of BFGS methods without line search.
We first revisit the adaptive BFGS proposed by \citet{gaoQuasiNewtonMethodsSuperlinear2019}. 
Compared with the original asymptotic analysis, we show that this method enjoys global non-asymptotic convergence, i.e., the linear convergence rate of $\fO((1 - 1 / (6 \varkappa))^{k})$ in the first phase and the superlinear convergence rate of $\fO((D_{1} / k)^{k})$ in the second phase, where $k$ is the iteration counter, $\varkappa$ is the condition number, and $D_{1}$ depends on properties of the objective and the initialization. 
Furthermore, we propose a smoothness aided adaptive BFGS method, which iterates with a larger step size by considering the gradient Lipschitz continuity.
Our theoretical analysis shows that such a strategy still preserves the linear and superlinear non-asymptotic convergence.
Additionally, the larger step size allows our method to enter the phase with the linear convergence rate of $\fO((1 - 1 / (2 \varkappa))^{k})$ earlier when the initial point is far away from the optimal solution.
We summarize our theoretical results in Table~\ref{tab:rates}.

\begin{table}[t]
    \centering
    \resizebox{\linewidth}{!}{
    \begin{tabular}{|Sc|Sc|Sc|Sc|}
    \hline
    Method & Phase & Convergence Rate & Starting Moment \\  \hline
    \multirow{2}[5]{*}{\makecell[c]{Algorithms~\ref{alg:ada_bfgs} and~\ref{alg:ada_bfgs_improved}}} & \makecell{Linear Phase \\ (Theorem~\ref{thm:linear_phase_linear_rates})} & $\fO\Biggl(\biggl(1 - \dfrac{1}{6 \varkappa}\biggr)^{k}\Biggr)$ & $ \tilde{\fO}\Bigl(\Psi(\bmB_{0}) + \min\bigl\{M^{2} \Delta, \varkappa M\sqrt{\Delta}\bigr\}\Bigr)$ \\ \cline{2-4}
    & \makecell{Superlinear Phase \\ (Theorem~\ref{thm:superlinear_phase_superlinear_rates})} &     $\fO\Biggl(\biggl(\dfrac{D_{1}}{k}\biggr)^{k}\Biggr)$ & $D_{1} \defeq \tilde{\fO}\Bigl(\Psi(\tmB_{0}) + \tM \sqrt{\Delta}\bigl(\Psi(\bmB_{0}) + \varkappa + \min\bigl\{M^{2} \Delta,  \varkappa M \sqrt{\Delta}\bigr\}\bigr)\Bigr)$ \\ \hline
    \multirow{3}[12]{*}{\makecell[c]{Algorithm~\ref{alg:ada_bfgs_improved}}} & \makecell{Linear Phase $\mathrm{I}$ \\ (Theorem~\ref{thm:linear_phase_linear_rates_1_improved})} & $\fO\Biggl(\biggl(1 - \dfrac{1}{2 \varkappa \min\{M^{2} \Delta, \varkappa\}} \biggr)^{k} \Biggr)$ & $  \fO\Bigl(\Psi(\bmB_{0})\Bigr)$ \\ \cline{2-4}
    & \makecell{Linear Phase $\mathrm{II}$ \\ (Theorem~\ref{thm:linear_phase_linear_rates_2_improved})} & $\fO\Biggl(\biggl(1 - \dfrac{1}{2 \varkappa}\biggr)^{k}\Biggr)$ & $\tilde{\fO}\Bigl(\Psi(\bmB_{0}) + \min\bigl\{\varkappa M^{2} \Delta, \varkappa^{2}\bigr\}\Bigr)$ \\ \cline{2-4}
    & \makecell{Superlinear Phase \\ (Theorem~\ref{thm:superlinear_phase_superlinear_rates_improved})} & $\fO\Biggl(\biggl(\dfrac{D_{2}}{k}\biggr)^{k}\Biggr)$ & $
        D_{2} \defeq \fO\Bigl(\Psi(\tmB_{0}) + 
        \tM \sqrt{\Delta} \bigl(\Psi(\bmB_{0}) + \min\bigl\{\varkappa M^{2} \Delta, \varkappa^{2}\bigr\}\bigr)\Bigr)$ \\ \hline
    \end{tabular}}
    \caption{\small We summarize our explicit convergence rates for adaptive BFGS method~(Algorithm~\ref{alg:ada_bfgs})~\cite{gaoQuasiNewtonMethodsSuperlinear2019} and for the variants proposed in this work~(Algorithm~\ref{alg:ada_bfgs_improved}). 
    Here, we denote 
    $\Psi(\mA) = \Tr(\mA) - \ln \det(\mA) - n$, 
    $\bmB_{0} = \mB_{0} / L$, $\tM = L_{2}/({2 \mu^{{3}/{2}}})$, $\tmB_{0} = \nabla^{2} f(\vx_{*})^{-{1}/{2}} \mB_{0} \nabla^{2} f(\vx_{*})^{-{1}/{2}}$, $\varkappa=L / \mu$, and $\Delta = f(\vx_{0}) - f(\vx_*)$, 
    where $\mB_{0}$ is the Hessian estimator at the initial point $\vx_0$, 
    $L$ is the gradient Lipschitz parameter, 
    $L_{2}$ is the Hessian Lipschitz parameter,
    and $M$ is the self-concordant parameter which may be smaller than $\tM$. 
    }\label{tab:rates}
\end{table}

\section{Preliminaries}\label{sec:pre}

In this section, we present the required assumptions, briefly review the BFGS update, and introduce the form of the adaptive step size, whose derivation will be provided in Section~\ref{sec:ada_stepsize}. We conclude the section by summarizing the complete algorithm.

\subsection{Notations and Assumptions}

We use $\Norm{\cdot}$ to denote the Euclidean norm for vectors and the spectral norm for matrices, respectively. We let $\BS_{++}^{n}$ be the set of $n \times n$ symmetric positive definite matrices and let $\mI$ be the $n \times n$ identity matrix. 
The trace and determinant are denoted by $\Tr(\cdot)$ and $\det (\cdot)$, respectively. 
For a twice differentiable and convex objective $f:\BR^{n} \to \BR$, we define $\lNorm{\vd}{\vx} \defeq (\vd^{\T} \nabla^{2} f(\vx) \vd)^{1/2}$ as the weighted norm of a vector $\vd \in \BR^{n}$ with respect to the Hessian at point $\vx \in \BR^{n}$. 
Additionally, the unique minimizer of the objective $f$ is denoted by $\vx_{*}$.

We first impose the assumptions of strong convexity and a Lipschitz continuous gradient as follows.

\begin{asm}\label{asm:strongly_convex}
We suppose the function $f:\BR^{n} \to \BR$ is $\mu$-strongly convex with  $\mu>0$, i.e., for all $\vx, \vy \in \BR^{n}$ and $\lambda\in[0,1]$, we have 
\begin{align*}
    f(\lambda \vx + (1 - \lambda) \vy) \leq \lambda f(\vx) + (1 - \lambda) f(\vy) - \frac{\lambda(1 - \lambda) \mu}{2}\Norm{\vx - \vy}^2.
\end{align*}
\end{asm}
\begin{asm}\label{asm:grad_Lip}

We suppose the function $f:\BR^{n} \to \BR$ has an $L$-Lipschitz continuous gradient with $L > 0$, i.e., for all $\vx, \vy \in \BR^{n}$, we have 
\begin{align*}
    \Norm{\nabla f(\vx)-\nabla f(\vy)}\leq L\Norm{\vx-\vy}.
\end{align*}
\end{asm}
Consequently, we define the condition number of $f$ as $\varkappa \defeq {L}/{\mu}$. 
We then consider the assumptions of self-concordance and a Lipschitz continuous Hessian, which are commonly used in the study of second-order optimization.

\begin{asm}\label{asm:self_con}
We suppose the convex function $f:\BR^{n} \to \BR$ is $M$-self-concordant with $M > 0$, i.e., for all $\vx, \vd \in \BR^{n}$, we have
\begin{equation*}
    \abs*{\nabla^{3} f(\vx) [\vd, \vd, \vd]} \leq 2 M \lNorm{\vd}{\vx}^{3}.
\end{equation*}
In particular, we say $f$ is standard self-concordant if $M = 1$. 
\end{asm}

\begin{asm}\label{asm:Hessian_Lip}
We suppose the function $f:\BR^{n} \to \BR$ has an $L_{2}$-Lipschitz continuous Hessian with $L_{2} > 0$, i.e., for all $\vx, \vy \in \BR^{n}$, we have
\begin{equation*}
    \Norm*{\nabla^{2} f(\vx) - \nabla^{2} f(\vy)} \leq L_{2} \Norm{\vx - \vy}.
\end{equation*}
\end{asm}

The above assumptions have the following relationships. 

\begin{lem}[{\cite[Example 4.1]{rodomanovGreedyQuasiNewtonMethods2021}}]\label{lem:Hessian_Lip2self_con}
Under Assumptions~\ref{asm:strongly_convex} and~\ref{asm:Hessian_Lip}, 
the function $f$ is self-concordant with parameter $L_{2}/(2\mu^{3/2})$.
\end{lem}

\begin{lem}[{\cite[Theorem 6.7]{gaoQuasiNewtonMethodsSuperlinear2019}}]\label{lem:self_con2Hessian_Lip}
Under Assumptions~\ref{asm:grad_Lip} and~\ref{asm:self_con}, 
the function $f$ has a $2 M L^{3/2}$-Lipschitz continuous Hessian.
\end{lem}

We define $\tM \defeq {L_2}/(2 \mu^{3/2})$.
Note that the self-concordant parameter obtained by Lemma~\ref{lem:Hessian_Lip2self_con} may not be tight.
In other words, under Assumptions~\ref{asm:strongly_convex} and~\ref{asm:Hessian_Lip}, the function $f$ may be $M$-self-concordant with $M\leq \tM$, e.g., the logistic loss~\cite{zhangCommunicationefficientDistributedOptimization2018} and the regularized log-sum-exp function~\cite{rodomanovGreedyQuasiNewtonMethods2021}. 

\subsection{BFGS Methods}

Quasi-Newton methods solve the minimization problem (\ref{prob:main}) by the iteration scheme
\begin{equation}\label{eq:quasi_update}
\begin{cases}
\vd_{k} = - \mB_{k}^{-1} \vg_{k}, \\
\vx_{k + 1} = \vx_{k} + t_{k} \vd_{k}, 
\end{cases}   
\end{equation}
where $\vg_{k} = \nabla f(\vx_{k})$, $\mB_{k}\in\BR^{n\times n}$ is the estimator of $\nabla^2 f(\vx)$, and $t_{k}$ is the step size.
We focus on BFGS methods~\cite{broydenQuasiNewtonMethodsTheir1967, fletcherNewApproachVariable1970, goldfarbFamilyVariablemetricMethods1970,shannoConditioningQuasiNewtonMethods1970} which update the Hessian estimator by
\begin{equation}\label{eq:bfgs_update}
    \mB_{k + 1} = \mB_{k} - \frac{\mB_{k} \vs_{k} \vs_{k}^{\T} \mB_{k}}{\vs_{k}^{\T} \mB_{k} \vs_{k}} + \frac{\vy_{k} \vy_{k}^{\T}}{\vy_{k}^{\T} \vs_{k}},
\end{equation}
where $\vs_{k} = \vx_{k + 1} - \vx_{k}$ and $\vy_{k} = \vg_{k + 1} - \vg_{k}$.
Note that we can apply the Sherman--Morrison--Woodbury formula~\cite{hendersonDerivingInverseSum1981} to access $\mB_{k+1}^{-1}$ with a computational cost of $\fO(n^2)$ given $\mB_k^{-1}$, since BFGS performs the rank-two update \eqref{eq:bfgs_update}. It is well known that for strongly convex objective functions, BFGS methods ensure that the matrix $\mB_{k}$ is positive definite for all $k \geq 1$ if the initial estimator $\mB_{0}$ is positive definite.
It is known that the BFGS iteration \eqref{eq:quasi_update}--\eqref{eq:bfgs_update} with $t_k\equiv 1$ enjoys non-asymptotic local convergence rates~\cite{jinNonasymptoticSuperlinearConvergence2023,rodomanovRatesSuperlinearConvergence2022, rodomanovNewResultsSuperlinear2021}, 
while global convergence cannot be guaranteed by taking the unit step size.

The convergence analysis of BFGS is typically based on the following affine transformations. 
For any given weight matrix $\mP\in\BS_{++}^{n}$, we define
\begin{align}
    \hvg_{k} = \mP^{-\frac{1}{2}} \vg_{k}, \quad \hvs_{k} = \mP^{\frac{1}{2}} \vs_{k}, \quad \hvy_{k} = \mP^{-\frac{1}{2}} \vy_{k}, \quad \hvd_{k} = \mP^{\frac{1}{2}} \vd_{k},  \quad \text{and} \quad
    \hmB_{k} = \mP^{-\frac{1}{2}} \mB_{k} \mP^{-\frac{1}{2}}. \label{def:wgt_vec}
\end{align}

It is well known that BFGS methods are invariant under affine transformations, i.e., 
\begin{align}\label{eq:affine-inv}
    \hvg_{k}^{\T} \hvs_{k} = \vg_{k}^{\T} \vs_{k}, \quad \hvy_{k}^{\T} \hvs_{k} = \vy_{k}^{\T} \vs_{k}, \quad \text{and} \quad
    \hmB_{k + 1} = \hmB_{k} - \frac{\hmB_{k} \hvs_{k} \hvs_{k}^{\T} \hmB_{k}}{\hvs_{k}^{\T} \hmB_{k} \hvs_{k}} + \frac{\hvy_{k} \hvy_{k}^{\T}}{\hvy_{k}^{\T} \hvs_{k}}.
\end{align}

\begin{algorithm}[ht]
    \caption{Adaptive BFGS~(c.f.\cite{gaoQuasiNewtonMethodsSuperlinear2019})}
    \label{alg:ada_bfgs}
    \begin{algorithmic}[1]
        \STATE $\vx_{0} \in \BR^{n}$, $M > 0$, $\mB_{0} \succ \vzero$ \\[0.15cm]
        \STATE \textbf{for} $k = 0, 1, 2, \dots$ $\textbf{do}$ \\[0.15cm]
        \STATE \quad $\vg_{k}=\nabla f(\vx_{k})$, \ $\vd_{k} = - \mB_{k}^{-1} \vg_k$ \\[0.15cm]
        \STATE \quad \label{line:tk} $t_{k} = -\dfrac{\vg_{k}^{\T}\vd_{k}}{\lNorm{\vd_{k}}{\vx_{k}} (\lNorm{\vd_{k}}{\vx_{k}} - M \vg_{k}^{\T} \vd_{k})}$ \\[0.15cm]
        \STATE \quad $\vx_{k + 1} = \vx_{k} + t_{k} \vd_{k}$ \\[0.15cm]
        \STATE \quad $\vs_{k} = t_{k} \vd_{k}$, \ $\vy_{k} = \vg_{k+1} - \vg_k$ \\[0.15cm]
        \STATE \quad $\mB_{k + 1} = \mB_{k} - \dfrac{\mB_{k} \vs_{k} \vs_{k}^{\T} \mB_{k}}{\vs_{k}^{\T} \mB_{k} \vs_{k}} + \dfrac{\vy_{k} \vy_{k}^{\T}}{\vy_{k}^{\T} \vs_{k}}$ \\[0.15cm]
        \STATE \textbf{end for}
    \end{algorithmic}
\end{algorithm}
We also define the quantities $\htheta_{k}$ and $\hatm_{k}$ by
\begin{equation}\label{def:theta_m}
\cos(\htheta_{k}) \defeq \frac{-\hvg_{k}^{\T} \hvs_{k}}{\Norm{\hvg_{k}} \Norm{\hvs_{k}}}
 \quad \text{and} \quad
\hatm_{k} \defeq \frac{\hvy_{k}^{\T} \hvs_{k}}{\Norm{\hvs_{k}}^{2}}.
\end{equation}
For a given $\mA \in \BS_{++}^{n}$, we introduce the potential function:
\begin{equation}\label{def:potential_func}
    \Psi(\mA) \defeq \Tr(\mA) - \ln \det(\mA) - n, 
\end{equation}
which can be viewed as the Bregman distance between $\mA$ and the identity matrix $\mI$ with respect to the function $\Phi(\mA) = - \ln \det(\mA)$.
The following result has been widely employed in the convergence rate analysis of BFGS~\cite{nocedalNumericalOptimization2006, jinNonasymptoticGlobalConvergence2025, jinNonasymptoticGlobalConvergence2024b, jinAffineinvariantGlobalNonasymptotic2025}. 

\begin{prop}[{\cite[Proposition~2]{jinNonasymptoticGlobalConvergence2025}}]\label{prop:potential_func_upper}
Let $\{\mB_{k}\}_{k \geq 0}$ be the Hessian approximation matrices generated by the BFGS update in \eqref{eq:bfgs_update},
then we have
\begin{equation*}
    \Psi(\hmB_{k + 1}) \leq \Psi(\hmB_{k}) + \frac{\Norm{\hvy_{k}}^{2}}{\hvs_{k}^{\T} \hvy_{k}} - 1 + \ln \frac{\cos^{2} \htheta_{k}}{\hatm_{k}}
\end{equation*}
for all $k\geq 0$, where $\hat\mB_k$, $\htheta_{k}$ and $\hatm_{k}$ follow the definitions in \cref{def:wgt_vec,eq:affine-inv,def:theta_m} with any $\mP \in \BS_{++}^{n}$. Furthermore, for all $k \geq 1$, we have
\begin{equation}\label{eq:acc_potential_func}
    \sum_{i = 0}^{k - 1} \ln \frac{\cos^{2}(\htheta_{i})}{\hatm_{i}} \geq - \Psi(\hmB_{0}) + \sum_{i = 0}^{k - 1} \left( 1 - \frac{\Norm{\hvy_{i}}^{2}}{\hvs_{i}^{\T} \hvy_{i}} \right).
\end{equation}
\end{prop}

Recent advances in BFGS have shown that line search strategies can attain explicit global convergence rates.
Specifically, \citet{krutikovConvergenceRateQuasiNewton2023} showed that BFGS with exact line search achieves global linear convergence when the objective is strongly convex and has a Lipschitz continuous gradient.
Later, \citet{jinNonasymptoticGlobalConvergence2025} proved that an additional assumption of Hessian Lipschitz continuity can ensure non-asymptotic local superlinear rates while maintaining global linear rates.
Furthermore, \citet{rodomanovGlobalComplexityAnalysis2024} established both non-asymptotic global and local guarantees for BFGS with more practical inexact line search strategies and \citet{jinNonasymptoticGlobalConvergence2024b,jinAffineinvariantGlobalNonasymptotic2025} considered a specific Armijo--Wolfe condition to provide sharper convergence rates and extended the results to non-strongly convex objectives.

\subsection{Global Convergence with Adaptive Step Size}\label{sec:ada_stepsize}

This subsection briefly reviews the adaptive BFGS for minimizing the self-concordant function
\cite{gaoQuasiNewtonMethodsSuperlinear2019} \footnote{
This subsection presents adaptive BFGS and its intuitions for minimizing the general self-concordant function, which can be easily extended from the results of \citet{gaoQuasiNewtonMethodsSuperlinear2019} that focus on the standard self-concordant case. For completeness, we provide proofs for the results under the general self-concordant assumption in Appendix~\ref{sec:pre}.}.
We present details of this method in Algorithm~\ref{alg:ada_bfgs}.
The highlight of adaptive BFGS is that the algorithm determines the step size $t_k$ by a closed form expression, which is easy to implement and avoids the additional computational cost in line search strategies.

We begin to introduce the intuition of BFGS with the following lemma, which applies the properties of self-concordant functions~\cite[Theorems 5.1.8 and~5.1.9]{nesterovLecturesConvexOptimization2018}
to characterize the divergence of the function value with respect to its linear approximation.

\begin{lem}\label{lem:self_con_bound}
    Under Assumptions~\ref{asm:self_con}, 
    we have 
    \begin{align}
        & \nabla f(\vx + t \vd)^\T\vd \geq \nabla f(\vx)^\T\vd + \frac{t \lNorm{\vd}{\vx}^{2}}{1 + M t\lNorm{\vd}{\vx}}, \label{eq:self_con_grad_lower} \\
        & f(\vx + t \vd) \geq f(\vx) + t\nabla f(\vx)^\T\vd + \frac{\omega(Mt\lNorm{\vd}{\vx})}{M^{2}} \label{eq:self_con_f_lower}
    \end{align}
    for all $\vx,\vd\in\BR^n$ and $t \geq 0$,
    where $\omega(z) \defeq z - \ln(1 + z)$ for $z \in [0, +\infty)$.
    Additionally, we have
    \begin{align}
        & \nabla f(\vx + t \vd)^\T\vd \leq \nabla f(\vx)^\T\vd + \frac{t \lNorm{\vd}{\vx}^{2}}{1 - M t \lNorm{\vd}{\vx}}, \label{eq:self_con_grad_upper} \\
        & f(\vx + t \vd) \leq f(\vx) + t\nabla f(\vx)^\T\vd + \frac{ \omega_{*}(Mt\lNorm{\vd}{\vx})}{M^{2}} \label{eq:self_con_f_upper}
    \end{align}
    for all $t \in \BR$ such that $Mt\lNorm{\vd}{\vx}\in[0,1)$, where $\omega_{*}(z) \defeq - z - \ln(1 - z)$ for $z \in [0, 1)$. 
\end{lem}

For a given descent direction $\vd_k\in\BR^n$ such that $\vg_{k}^{\T} \vd_{k} < 0$, minimizing the right-hand side of \eqref{eq:self_con_f_upper} with respect to $t$ by taking $\vx=\vx_k$ and $\vd=\vd_k$ leads to the step size in adaptive BFGS~(Algorithm~\ref{alg:ada_bfgs}), i.e.,
\begin{equation}\label{eq:adaptive_stepsize}
    t_{k} = -\frac{\vg_{k}^{\T} \vd_{k}}{\lNorm{\vd_{k}}{\vx_{k}} (\lNorm{\vd_{k}}{\vx_{k}} - M \vg_{k}^{\T} \vd_{k})}.
\end{equation}

The following lemma extends Theorems 4.1, 6.2, and 6.3 of \citet{gaoQuasiNewtonMethodsSuperlinear2019} to show the decrease of the function value, the curvature condition for decent directions, and the step size defined in \eqref{eq:adaptive_stepsize}.

\begin{lem}\label{lem:ada_bfgs_properties}
    Under Assumptions~\ref{asm:strongly_convex} and~\ref{asm:self_con}, we let $\vx_{k + 1} = \vx_{k} + t_{k} \vd_{k}$ such that $\vd_{k}\in\BR^n$ is a descent direction such that $\vg_{k}^{\T} \vd_{k} < 0$ and $t_{k}$ follows \eqref{eq:adaptive_stepsize}, then we have
    \begin{align}
    & f(\vx_{k + 1}) \leq f(\vx_{k}) - \frac{\omega(M \eta_{k})}{M^{2}} , \label{eq:self_con_f_dec} \\
    & f(\vx_{k + 1}) - f(\vx_{k}) \leq \frac{1}{2} t_{k} \vg_{k}^{\T} \vd_{k}, \label{eq:armijo_condition} \\
    & \frac{2 M \eta_{k}}{1 + 2 M \eta_{k}} \vg_{k}^{\T} \vd_{k} \leq  \vg_{k+1}^{\T} \vd_{k} \leq 0, \label{eq:curvature_condition}
    \end{align}
    where $\omega(z) \defeq z - \ln(1 + z)$ and
    \begin{equation}\label{def:eta}
        \eta_{k} = -\frac{\vg_{k}^{\T} \vd_{k}}{\lNorm{\vd_{k}}{\vx_{k}}}.
    \end{equation}
\end{lem}

\begin{proof}
    Please refer to Appendix~\ref{proof:ada_bfgs_properties}.
\end{proof}

It is worth noting that \cref{eq:armijo_condition,eq:curvature_condition} can guarantee that the step size defined in $t_k$ satisfies the Armijo--Wolfe condition.
Therefore, adaptive BFGS~(Algorithm~\ref{alg:ada_bfgs}) has the global asymptotic convergence. 

\begin{prop}[{\cite[Theorem~6.6]{gaoQuasiNewtonMethodsSuperlinear2019}}]\label{prop:gao_global_linear}
    Under Assumptions~\ref{asm:strongly_convex} and~\ref{asm:self_con}, Algorithm~\ref{alg:ada_bfgs} holds that 
    \begin{align*}
        \sum_{k = 0}^{+\infty} \Norm{\vx_{k} - \vx_{*}} < +\infty.
    \end{align*}
\end{prop}

Proposition~\ref{prop:gao_global_linear} and the Lipschitz continuity of the Hessian~(Assumption~\ref{asm:Hessian_Lip}) directly yield the Dennis--Mor{\'e} characterization~\cite[Proposition~4]{griewankLocalConvergenceAnalysis1982}. 
In addition, \citet[Theorem~6.8]{gaoQuasiNewtonMethodsSuperlinear2019} shows that the step size $t_{k}$ converges to $1$, leading to the asymptotic superlinear convergence of adaptive BFGS~(Algorithm~\ref{alg:ada_bfgs}).

\begin{prop}[{\cite[Theorem~6.9]{gaoQuasiNewtonMethodsSuperlinear2019}}]\label{prop:gao_superlinear}
    Under Assumptions~\ref{asm:strongly_convex} and~\ref{asm:self_con}, Algorithm~\ref{alg:ada_bfgs} holds that 
    \begin{align*}
        \lim_{k \to +\infty} \frac{\Norm{\vx_{k + 1} - \vx_{*}}}{ \Norm{\vx_{k} - \vx_{*}}} = 0.
    \end{align*}
\end{prop}

\begin{remark}
    The proofs of Propositions~\ref{prop:gao_global_linear} and~\ref{prop:gao_superlinear}~\cite{gaoQuasiNewtonMethodsSuperlinear2019} require the Lipschitz continuity of the gradient and Hessian on the sub-level set $\{\vx \in \BR^{n}: f(\vx) \leq f(\vx_0)\}$, which naturally holds under Assumptions~\ref{asm:strongly_convex} and~\ref{asm:self_con}.
\end{remark}

\section{The Phase of Global Linear Convergence}\label{sec:linear_phase}

This section provides the explicit global linear convergence rates for adaptive BFGS~(Algorithm~\ref{alg:ada_bfgs}).
We analyze the behavior of the algorithm based on its two cases, which are partitioned according to the value of $\eta_{k} = -\vg_{k}^{\T} \vd_{k}/\lNorm{\vd_{k}}{\vx_{k}}$.
Specifically, we define the index set for all $k \geq 0$ based on the value of $\eta_{k}$ as follows:
\begin{align}
    & \fI_{\infty} \defeq \left\{ k \in \BN : 0 \leq M \eta_{k} < 1 \right\}. \label{def:ind_set} 
\end{align}    
Then the first case corresponds to the iteration counters with  $k \in \fI_{\infty}$ (i.e., the small $\eta_k$) that yields the adaptive step sizes satisfying the curvature condition \eqref{eq:curvature_condition}, and the second case corresponds to the iteration counters with $k \notin \fI_{\infty}$ (i.e., the large $\eta_k$)  that leads to the decrease on the function value controlled by \eqref{eq:self_con_f_dec}.
The remainder of this section establishes the global linear convergence rate of $\fO(( 1 - 1 / \varkappa)^k)$ by analyzing the cases of small $\eta_k$ and large $\eta_k$, respectively. 

\subsection{The Case of Small \texorpdfstring{$\eta_{k}$}{eta\_k}}

We first show the decease of the function value gap $f(\vx_{k}) - f(\vx_{*})$ for adaptive BFGS~(Algorithm~\ref{alg:ada_bfgs}).

\begin{prop}\label{prop:analy_framework}
    Under Assumptions~\ref{asm:strongly_convex} and~\ref{asm:self_con}, Algorithm~\ref{alg:ada_bfgs} holds
    \begin{equation}\label{eq:analy_framework}
        \frac{f(\vx_{k + 1}) - f(\vx_{*})}{f(\vx_{k}) - f(\vx_{*})}
         \leq 1 - \frac{\hq_{k}\cos^{2}(\htheta_{k})}{2 (1 + 2 M \eta_{k})\hatm_{k}}
    \end{equation}
    for all $k \geq 0$, where $\htheta_{k}$ and $\hatm_{k}$ follow the definitions in \cref{def:wgt_vec,eq:affine-inv,def:theta_m} with any $\mP\in\BS^n_{++}$ and $\hq_{k}$ is defined as
    \begin{equation}\label{def:q}
        \hq_{k} \defeq \frac{\Norm{\hvg_{k}}^{2}}{f(\vx_{k}) - f(\vx_{*})}.
    \end{equation}
\end{prop}

\begin{proof}
    Recall that \eqref{eq:affine-inv} states $\hvg_{k}^{\T} \hvs_{k} = \vg_{k}^{\T} \vs_{k}$ and $\hvy_{k}^{\T} \hvs_{k} = \vy_{k}^{\T} \vs_{k}$. 
    Then applying \cref{eq:armijo_condition,eq:curvature_condition} in results of Lemma~\ref{lem:ada_bfgs_properties}, we have
    \begin{equation}
        f(\vx_{k}) - f(\vx_{k + 1}) \geq - \frac{1}{2} \hvg_{k}^{\T} \hvs_{k} = -\frac{\hvg_{k}^{\T} \hvs_{k}}{2\Norm{\hvg_{k}}^{2}} \Norm{\hvg_{k}}^{2}, \label{eq:_analy_framework} 
    \end{equation}
    and
    \begin{equation}
        \frac{\hvy_{k}^{\T} \hvs_{k}}{- \hvg_{k}^{\T} \hvs_{k}} = \frac{(\hvg_{k + 1} - \hvg_{k})^{\T} \hvd_{k}}{- \hvg_{k}^{\T} \hvd_{k}} \geq \frac{1}{1 + 2 M \eta_{k}}, \label{eq:analy_framework_curvature}
    \end{equation}
    respectively.
    Combining \eqref{eq:analy_framework_curvature} with the definition of $\htheta_{k}$ and $\hatm_{k}$, we obtain
    \begin{equation*}
        \frac{- \hvg_{k}^{\T} \hvs_{k}}{\Norm{\hvg_{k}}^{2}} = \frac{(\hvg_{k}^{\T} \hvs_{k})^{2}}{\Norm{\hvg_{k}}^{2} \Norm{\hvs_{k}}^{2}} \frac{\Norm{\hvs_{k}}^{2}}{- \hvg_{k}^{\T} \hvs_{k}} = \frac{(\hvg_{k}^{\T} \hvs_{k})^{2}}{\Norm{\hvg_{k}}^{2} \Norm{\hvs_{k}}^{2}} \frac{\Norm{\hvs_{k}}^{2}}{\hvy_{k}^{\T} \hvs_{k}} \frac{\hvy_{k}^{\T} \hvs_{k}}{- \hvg_{k}^{\T} \hvs_{k}} \geq \frac{\cos^{2}(\htheta_{k})}{(1 + 2 M \eta_{k})\hatm_{k}}.
    \end{equation*}
    Substituting above inequality into \eqref{eq:_analy_framework} and using the definition of $\hq_{k}$, we can derive
    \begin{equation}\label{eq:f_dec_analy_framework}
        f(\vx_{k}) - f(\vx_{k + 1}) \geq \frac{\hq_{k}\cos^{2}(\htheta_{k})}{2 (1 + 2 M \eta_{k})\hatm_{k}} \left(f(\vx_{k}) - f(\vx_{*}) \right).
    \end{equation}
    By rearranging the terms in the above inequality, we obtain the desired result in \eqref{eq:analy_framework}.    
\end{proof}

\begin{remark}
    Note that the right-hand side of \eqref{eq:analy_framework} holds
    \begin{equation*}
        \frac{\hq_{k}\cos^{2}(\htheta_{k})}{2 (1 + 2 M \eta_{k}) \hatm_{k}}\in (0, 1).
    \end{equation*}
    Specifically, \eqref{eq:self_con_f_dec} of Lemma~\ref{lem:ada_bfgs_properties} implies that $f(\vx_{k}) > f(\vx_{k + 1}) > f(\vx_{*})$. 
    Combining with \eqref{eq:f_dec_analy_framework}, we obtain the upper bound
    \begin{equation*}
        \frac{\hq_{k}}{2 (1 + 2 M \eta_{k})} \frac{\cos^{2}(\htheta_{k})}{\hatm_{k}} \leq \frac{f(\vx_{k}) - f(\vx_{k + 1})}{f(\vx_{k}) - f(\vx_{*})} < 1.
    \end{equation*}
    The lower bound is guaranteed by the fact that 
    \begin{equation*}
        \cos(\htheta_{k}) = -\frac{\hvg_{k}^{\T} \hvs_{k}}{\Norm{\hvg_{k}} \Norm{\hvs_{k}}} = -\frac{\vg_{k}^{\T} \vd_{k}}{\Norm{\hvg_{k}} \Norm{\hvd_{k}}} > 0,
    \end{equation*}
    which ensures that the right-hand side of \eqref{eq:analy_framework} lies in $(0,1)$.
\end{remark}
\begin{remark}
The proof of Proposition~\ref{prop:analy_framework} establishes the decrease in the function value for the adaptive step size based on the curvature condition parameterized by $1/(1+2M\eta_{k})$ and the Armijo condition with parameter $1/2$.
In a recent work, \citet[Proposition~1]{jinNonasymptoticGlobalConvergence2025} also characterized the decrease in the function value in a form similar to our formulation \eqref{eq:analy_framework}. However, their results rely on exact line search and a curvature condition parameterized by 1.
\end{remark}

We then define the index set
\begin{equation}\label{def:ind_set_small_case}
    \fI_{k} \defeq \left\{ i \in \fI_{\infty} : i < k \right\}, 
\end{equation}
which consists of the iteration counters associated with the small $\eta_{i}$ in the first $k$ iterations. 
Multiplying \eqref{eq:analy_framework} over the indices in $\fI_k$ yields the following proposition for the small $\eta_i$.

\begin{prop}\label{prop:linear_phase_small_case}
Under Assumptions~\ref{asm:strongly_convex},~\ref{asm:grad_Lip} and~\ref{asm:self_con}, Algorithm~\ref{alg:ada_bfgs} holds
\begin{equation}\label{eq:f_dec_linear_phase_small_case}
    \frac{f(\vx_{k + 1}) - f(\vx_{*})}{f(\vx_{k}) - f(\vx_{*})} \leq 1 - \frac{\hq_{k} \cos^{2}(\htheta_{k})}{6 \hatm_{k}}
\end{equation}
for all $k \in \fI_{\infty}$, where $\htheta_{k}$ and $\hatm_{k}$ follow definitions in \cref{def:wgt_vec,eq:affine-inv,def:theta_m} with any $\mP\in\BS^n_{++}$.
Furthermore, for all $k \geq 1$ with $|\fI_{k}| \geq 1$, we have
\begin{equation}\label{eq:linear_phase_small_case}
    \frac{f(\vx_{k}) - f(\vx_{*})}{f(\vx_{0}) - f(\vx_{*})} \leq \left( 1 -  \left( \prod_{i \in \fI_{k}} \frac{\hq_{i} \cos^{2}(\htheta_{i})}{6 \hatm_{i}} \right)^{\frac{1}{|\fI_{k}|}} \right)^{|\fI_{k}|}.
\end{equation}
\end{prop}

The proof of Proposition~\ref{prop:linear_phase_small_case} requires the following lemma to bound the term of product.

\begin{lem}\label{lem:ineq_arith_geo}
    Suppose that the sequence $\{u_{i}\}_{i = 0}^{k - 1}$ and the positive number $a \geq 0$ 
    satisfy $u_{i} \geq 0$ and $1 - a u_{i} \geq 0$
    for all $i=0,\dots,k - 1$, then we have
    \begin{equation*}
        \prod_{i = 0}^{k - 1} (1 - a u_{i}) \leq \left( 1 - a \left( \prod_{i = 0}^{k - 1} u_{i} \right)^{\frac{1}{k}} \right)^{k}.
    \end{equation*}
\end{lem}

\begin{proof}
    Please refer to Appendix~\ref{proof:ineq_arith_geo}.
\end{proof}

We now prove Proposition~\ref{prop:linear_phase_small_case} as follows.

\begin{proof}[Proof of Proposition~\ref{prop:linear_phase_small_case}]
    Since $k \in \fI_{\infty}$, the definition \eqref{def:ind_set} implies
    \begin{equation*}
        \frac{1}{1 + 2 M \eta_{k}} \geq \frac{1}{3}.
    \end{equation*}
    Substituting above inequality into \eqref{eq:analy_framework}, we obtain \eqref{eq:f_dec_linear_phase_small_case}.

    Note that for all $k \geq 1$, we have
    \begin{equation*}
        \frac{f(\vx_{k}) - f(\vx_{*})}{f(\vx_{0}) - f(\vx_{*})} = \prod_{i = 0}^{k - 1} \frac{f(\vx_{i + 1}) - f(\vx_{*})}{f(\vx_{i}) - f(\vx_{*})} \leq \prod_{i \in \fI_{k}} \frac{f(\vx_{i + 1}) - f(\vx_{*})}{f(\vx_{i}) - f(\vx_{*})} \leq \prod_{i \in \fI_{k}} \left( 1 - \frac{\hq_{i} \cos^{2}(\htheta_{i})}{6 \hatm_{i}} \right),
    \end{equation*}
    where the first inequality is due to the strict monotonic decrease of $\{f(\vx_{k})\}_{k \geq 0}$ given by \eqref{eq:self_con_f_dec}, and the second follows from \eqref{eq:f_dec_linear_phase_small_case}. 
    Thus, Lemma~\ref{lem:ineq_arith_geo} implies that
    \begin{equation*}
        \frac{f(\vx_{k}) - f(\vx_{*})}{f(\vx_{0}) - f(\vx_{*})} \leq \prod_{i \in \fI_{k}} \left( 1 - \frac{\hq_{i} \cos^{2}(\htheta_{i})}{6 \hatm_{i}} \right) \leq \left( 1 - \left( \prod_{i \in \fI_{k}} \frac{\hq_{i} \cos^{2}(\htheta_{i})}{6 \hatm_{i}} \right)^{\frac{1}{|\fI_{k}|}} \right)^{|\fI_{k}|},
    \end{equation*}
    which completes the proof.
\end{proof}

\subsection{The Case of Large \texorpdfstring{$\eta_{k}$}{eta\_k}}

We first show that for all $k \notin \fI_{\infty}$, 
the adaptive BFGS~(Algorithm~\ref{alg:ada_bfgs}) can decrease the objective function value by at least some fixed constant, which means the number of such iteration counters is finite. 
To formalize this result,  we define the index set
\begin{equation}\label{def:ind_set_large_case}
    \bar{\fI}_{k} \defeq \left\{ 0, 1, 2, \dots, k - 1 \right\} \backslash \fI_{k}. 
\end{equation}
and provide the following properties for the function $\omega(z) = z - \ln (1 + z)$, which extends Lemma 5.1.5 of \citet{nesterovLecturesConvexOptimization2018} and Lemma~A.1 of \citet{rodomanovGlobalComplexityAnalysis2024}.
\begin{lem}\label{lem:omega_bounds}
    We let $\omega(z) = z - \ln (1 + z)$ for $z \in [0, +\infty)$. Then the function $\omega(z)$ is strictly increasing on its domain and has the lower bound 
    \begin{equation}\label{eq:omega_lower}
        \omega(z) \geq \begin{cases}
            \dfrac{z^{2}}{4}, & 0 \leq z < 1, \\[0.3cm]
            \dfrac{z}{4}, & z \geq 1.
        \end{cases}
    \end{equation}
    Furthermore, its inverse function $\omega^{-1}(y)$ holds
    \begin{equation}\label{eq:omega_inv_upper}
        \omega^{-1}(y) \leq y + \sqrt{2 y}
    \end{equation}
    for all $y \geq 0$.
\end{lem}

\begin{proof}
    Please refer to Appendix~\ref{proof:omega_bounds}.
\end{proof}

We present the upper bound of $|\bar{\fI}_{k}|$ as follows. 

\begin{prop}\label{prop:linear_phase_large_case_bounds1}
Under Assumptions~\ref{asm:strongly_convex} and~\ref{asm:self_con}. For all $k \geq 1$, the index set $\bar{\fI}_{k}$ defined in \eqref{def:ind_set_large_case} holds
\begin{equation}\label{eq:linear_phase_large_case_bounds1}
    |\bar{\fI}_{k}| \leq 4 M^{2} \Delta, 
\end{equation}
where $\Delta \defeq f(\vx_{0}) - f(\vx_{*})$.
\end{prop}

\begin{proof}
According to \eqref{eq:self_con_f_dec}, we know that  $\{f(\vx_{k})\}_{k \geq 0}$ is decreasing and it also implies
\begin{equation*}
\begin{aligned}
    f(\vx_{k}) - f(\vx_{*}) & = f(\vx_{0}) - f(\vx_{*}) + \sum_{i = 0}^{k - 1} \bigl( f(\vx_{i + 1}) - f(\vx_{i}) \bigr) \\
    & \leq f(\vx_{0}) - f(\vx_{*}) + \sum_{i \in \bar{\fI}_{k}} \bigl( f(\vx_{i + 1}) - f(\vx_{i}) \bigr) \\
    & \leq f(\vx_{0}) - f(\vx_{*}) - \frac{1}{M^{2}} \sum_{i \in \bar{\fI}_{k}} \omega(M \eta_{k}) \\ 
    & \leq f(\vx_{0}) - f(\vx_{*}) - \frac{\omega(1)|\bar{\fI}_{k}|}{M^{2}},
\end{aligned}
\end{equation*}
where the last inequality holds since $\omega(z)$ is monotonically increasing for $z \geq 0$ and the fact that $M \eta_{i} \geq 1$ for all $i \in \bar{\fI}_{k}$.
Combining above result with facts $f(\vx_{k}) - f(\vx_{*}) \geq 0$ and $\Delta = f(\vx_{0}) - f(\vx_{*})$, we obtain
\begin{equation*}
    \frac{\omega(1)|\bar{\fI}_{k}|}{M^{2}}  \leq \Delta.
\end{equation*}
Finally, \eqref{eq:omega_lower} with $z=1$ implies $\omega(1) \geq {1}/{4}$, which immediately achieves \eqref{eq:linear_phase_large_case_bounds1}.
\end{proof}

Recall that $k \notin \fI_{\infty}$ implies $M \eta_{k} \geq 1$, we can then apply \eqref{eq:omega_lower} to obtain a lower bound for $\omega(M \eta_{k})$. 
Together with \eqref{eq:self_con_f_dec}, this allows us to characterize the decrease for the square root of the function value gap by $\cos(\htheta_{k})$.

\begin{prop}\label{prop:linear_phase_large_case}
    Under Assumptions~\ref{asm:strongly_convex},~\ref{asm:grad_Lip} and~\ref{asm:self_con}, 
    Algorithm~\ref{alg:ada_bfgs} holds
    \begin{equation}\label{eq:linear_phase_large_case}
        \sqrt{f(\vx_{k + 1}) - f(\vx_{*})} - \sqrt{f(\vx_{k}) - f(\vx_{*})} \leq - \frac{\cos(\htheta_{k})}{4 \sqrt{2 \varkappa} M}
    \end{equation}
    for all $k \notin \fI_{\infty}$, where $\htheta_k$ follows the definitions in \cref{def:wgt_vec,eq:affine-inv,def:theta_m} with $\mP = L \mI$.
    Furthermore, for all $k \geq 1$, we have
    \begin{equation}\label{eq:linear_phase_large_case_bounds2}
        \sum_{i \in \bar{\fI}_{k}} \cos(\htheta_{i}) \leq 4 \sqrt{2 \varkappa} M \sqrt{\Delta}.
    \end{equation}
\end{prop}

The proof of Proposition~\ref{prop:linear_phase_large_case} is based on the following lemma.

\begin{lem}\label{lem:dec_sqrt}
    Let $\{y_{k}\}_{k \geq 0}$, $\{\gamma_{k}\}_{k \geq 0}$ be two positive sequences satisfying
    \begin{equation*}
        y_{k + 1} - y_{k} \leq - \gamma_{k} \sqrt{y_{k}},
    \end{equation*}
    then we have
    \begin{equation*}
        \sqrt{y_{k + 1}} - \sqrt{y_{k}} \leq - \frac{\gamma_{k}}{2}.
    \end{equation*}
\end{lem}

\begin{proof}
    Please refer to Appendix~\ref{proof:dec_sqrt}.
\end{proof}

We now prove Proposition~\ref{prop:linear_phase_large_case} as follows.

\begin{proof}[Proof of Proposition~\ref{prop:linear_phase_large_case}]
    For $k \notin \fI_{\infty}$, we have 
    \begin{equation*}
    \begin{aligned}
        f(\vx_{k + 1}) - f(\vx_{k}) & \overset{\eqref{eq:self_con_f_dec}}{\leq} - \frac{\omega(M \eta_{k})}{M^{2}}  \overset{\eqref{eq:omega_lower}}{\leq }- \frac{\eta_{k}}{4 M}  = - \frac{1}{4 M} \frac{- \vg_{k}^{\T} \vd_{k}}{\lNorm{\vd_{k}}{\vx_{k}}} \\ 
        & \leq - \frac{1}{4 \sqrt{L} M } \frac{-\vg_{k}^{\T} \vd_{k}}{\Norm{\vd_{k}} \Norm{\vg_{k}}} \Norm{\vg_{k}} \leq - \frac{\sqrt{2} \cos(\htheta_{k})}{4 \sqrt{\varkappa} M} \sqrt{f(\vx_{k}) - f(\vx_{*})}, 
    \end{aligned}
    \end{equation*}
    where the second inequality utilizes \eqref{eq:omega_lower} and $M\eta_{k} \geq 1$, the third inequality employs Assumption~\ref{asm:grad_Lip}, and the final inequality relies on Assumption~\ref{asm:strongly_convex} which leads  to $\Norm{\vg_{k}}^{2} \geq 2 \mu (f(\vx_{k}) - f(\vx_{*}))$.
    
    Then applying Lemma~\ref{lem:dec_sqrt} with $y_k=f(\vx_{k}) - f(\vx_{*})$ and $\gamma_k={\sqrt{2}\cos(\htheta_{k})}/(4 \sqrt{\varkappa} M)$ achieves \eqref{eq:linear_phase_large_case}.

    We now ready to prove \eqref{eq:linear_phase_large_case_bounds2}. 
    Note that for all $k \geq 1$, we have
    \begin{equation*}
    \begin{aligned}
        \sqrt{f(\vx_{k}) - f(\vx_{*})} & = \sqrt{f(\vx_{0}) - f(\vx_{*})} + \sum_{i = 0}^{k - 1} \left( \sqrt{f(\vx_{i + 1}) - f(\vx_{*})} - \sqrt{f(\vx_{i}) - f(\vx_{*})} \right) \\
        & \leq \sqrt{f(\vx_{0}) - f(\vx_{*})} + \sum_{i \in \bar{\fI}_{k}} \left( \sqrt{f(\vx_{i + 1}) - f(\vx_{*})} - \sqrt{f(\vx_{i}) - f(\vx_{*})} \right) \\
        & \leq \sqrt{f(\vx_{0}) - f(\vx_{*})} - \frac{1}{4 \sqrt{2 \varkappa} M } \sum_{i \in \bar{\fI}_{k}} \cos(\htheta_{i}),
    \end{aligned}
    \end{equation*}
    where the first inequality follows from the strict monotonic decrease of $\{f(\vx_{k})\}_{k \geq 0}$ in \eqref{eq:self_con_f_dec}, and the second from \eqref{eq:linear_phase_large_case}. 
    Based on facts $f(\vx_{k}) - f(\vx_{*}) \geq 0$ and $\Delta = f(\vx_{0}) - f(\vx_{*})$, we have
    \begin{equation*}
        \frac{1}{4 \sqrt{2 \varkappa} M} \sum_{i \in \bar{\fI}_{k}} \cos(\htheta_{i}) \leq \sqrt{\Delta},
    \end{equation*}
    which leads to \eqref{eq:linear_phase_large_case_bounds2}.
\end{proof}

\subsection{The Global Linear Convergence Rates}\label{sec:global-linear-rates}

We are ready to apply Propositions~\ref{prop:linear_phase_small_case} and~\ref{prop:linear_phase_large_case} to show adaptive BFGS~(Algorithm~\ref{alg:ada_bfgs}) has the global linear convergence rate of $\fO(\left( 1 - 1 / \varkappa \right)^{k})$. 
For ease of presentation, we use $K_{1}$ and $K_{2}$ to denote the right-hand sides of \eqref{eq:linear_phase_large_case_bounds1} and \eqref{eq:linear_phase_large_case_bounds2}, respectively, that is,
\begin{equation}\label{def:K_1_2}
    K_{1} \defeq 4 M^{2} \Delta 
    \quad \text{and} \quad K_{2} \defeq 4 \sqrt{2 \varkappa} M \sqrt{\Delta}.
\end{equation}
Additionally, we let $\mP=L\mI$ for \cref{def:wgt_vec,eq:affine-inv,def:theta_m} until the end of this section. 
We also define the weighted matrix $\bmB_{k}$ as
\begin{equation}\label{def:bar_B}
    \bmB_{k} = \frac{1}{L} \mB_{k}.
\end{equation}

The following lemma provides a lower bound for $\hq_{k}$ defined in \eqref{def:q} and an upper bound for $\Norm{\hvy_{k}}^{2} / (\hvs_{k}^{\T} \hvy_{k})$.

\begin{lem}[{\cite[Lemmas~4 and~5]{jinNonasymptoticGlobalConvergence2025}}]\label{lem:P_L_q_J_bounds}
    Under Assumptions~\ref{asm:strongly_convex}, ~\ref{asm:grad_Lip} and~\ref{asm:self_con}, Algorithm~\ref{alg:ada_bfgs} holds
    \begin{equation}\label{eq:P_L_q_lower}
        \hq_{k} \geq \frac{2}{\varkappa},
    \end{equation}
    and 
    \begin{equation}\label{eq:P_L_J_upper}
        \frac{\Norm{\hvy_{k}}^{2}}{\hvs_{k}^{\T} \hvy_{k}} \leq 1
    \end{equation}
    for all $k \geq 0$, where $\hq_{k}$ is defined in \eqref{def:q} and $\hvy_{k}$ and $\hvs_{k}$ are defined in \eqref{def:wgt_vec} by taking $\mP=L\mI$.
\end{lem}

We now present the global linear convergence rate in the regime of $k \geq 2 \Psi(\bmB_{0}) + (2 \ln \varkappa + 1) K_{1}$.

\begin{prop}\label{prop:linear_phase_linear_rates1}
    Under Assumptions~\ref{asm:strongly_convex},~\ref{asm:grad_Lip} and~\ref{asm:self_con}, Algorithm~\ref{alg:ada_bfgs} holds
    \begin{equation}\label{eq:_linear_phase_linear_rates1}
        \frac{f(\vx_{k}) - f(\vx_{*})}{f(\vx_{0}) - f(\vx_{*})} \leq \biggl( 1 - \frac{1}{3 \varkappa} \exp \biggl( - \frac{\Psi(\bmB_{0}) + K_{1} \ln \varkappa}{k - K_{1}} \biggr) \biggr)^{k - K_{1}}
    \end{equation}
    for all $k > K_{1}$, where $\Psi(\cdot)$, $\bmB_{0}$, and $K_{1}$ follow the definitions in \cref{def:potential_func,def:bar_B,def:K_1_2} respectively. 
    Furthermore, for all $k \geq 2 \Psi(\bmB_{0}) + (2 \ln \varkappa + 1) K_{1}$, we have
    \begin{equation}\label{eq:linear_phase_first_rates1}
        \frac{f(\vx_{k}) - f(\vx_{*})}{f(\vx_{0}) - f(\vx_{*})} \leq \biggl( 1 - \frac{1}{6 \varkappa} \biggr)^{k - K_{1}}.
    \end{equation}
\end{prop}

The proof of Proposition~\ref{prop:linear_phase_linear_rates1} relies on the following lemma.

\begin{lem}\label{lem:dec_aux_func}
    Define $\xi$: $(0, +\infty) \to \BR$ as
    \begin{equation*}
        \xi(u) = \left( 1 - a b^{\frac{1}{u}} \right)^{u},
    \end{equation*}
    where $0 < a < 1$ and $0 < b \leq 1$. Then the function $\xi(u)$ is strictly decreasing.
\end{lem}

\begin{proof}
    Please refer to Appendix~\ref{proof:dec_aux_func}.
\end{proof}

We now prove Proposition~\ref{prop:linear_phase_linear_rates1} as follows.

\begin{proof}[Proof of Proposition~\ref{prop:linear_phase_linear_rates1}]
    Since we take $\mP = L \mI$, Lemma~\ref{lem:P_L_q_J_bounds} implies $\hq_{k} \geq 2 / \varkappa$. 
    Substituting it into \eqref{eq:linear_phase_small_case} of Proposition~\ref{prop:linear_phase_small_case}, we obtain
    \begin{equation}\label{eq:linear_phase_terms}
        \frac{f(\vx_{k}) - f(\vx_{*})}{f(\vx_{0}) - f(\vx_{*})} \leq \left( 1 - \frac{1}{3 \varkappa} \left( \prod_{i \in \fI_{k}} \frac{\cos^{2} (\htheta_{i})}{\hatm_{i}} \right)^{\frac{1}{|\fI_{k}|}} \right)^{|\fI_{k}|}.
    \end{equation}

    In view of Lemma~\ref{lem:dec_aux_func}, we need to provide the lower bounds for $|\fI_{k}|$ and $\sum_{i \in \fI_{k}} \ln \frac{\cos^{2}(\htheta_{i})}{\hatm_{i}}$ to bound the right-hand side of \eqref{eq:linear_phase_terms}.  
    According to \eqref{eq:linear_phase_large_case_bounds1}, we obtain the lower bound of $|\fI_{k}|$ as 
    \begin{equation}\label{eq:linear_phase_rates1_bounds1}
        |\fI_{k}| = k - |\bar{\fI}_{k}| \geq k - K_{1}.
    \end{equation}
    As a direct consequence, we also obtain $|\fI_{k}| \geq 1$ since $k > K_{1}$.
    According to \eqref{eq:acc_potential_func}, we have
    \begin{equation*}
        \sum_{i = 0}^{k - 1} \ln \frac{\cos^{2}(\htheta_{i})}{\hatm_{i}}  \geq - \Psi(\bmB_{0}) + \sum_{i = 0}^{k - 1} \left( 1 - \frac{\Norm{\hvy_{i}}^{2}}{\hvs_{i}^{\T} \hvy_{i}} \right).
    \end{equation*}
    Recall that Lemma~\ref{lem:P_L_q_J_bounds} states $1 - \Norm{\hvy_{i}^{2}} / (\hvs_{i}^{\T} \hvy_{i}) \geq 0$, then we have
    \begin{equation*}
        \sum_{i = 0}^{k - 1} \ln \frac{\cos^{2}(\htheta_{i})}{\hatm_{i}} \geq - \Psi(\bmB_{0}).
    \end{equation*}
    On the other hand, the facts $\cos(\htheta_{i}) \in(0, 1)$ and $\hatm_{i} = \vy_{i}^{\T} \vs_{i} / (L \Norm{\vs_{i}}^{2}) \geq 1 / \varkappa$ implies
    \begin{equation*}
        \sum_{i \in \bar{\fI}_{k}} \ln \cos^{2}(\htheta_{i}) \leq 0 \quad \text{and} \quad \sum_{i \in \bar{\fI}_{k}} \ln \hatm_{i} \geq - |\bar{\fI}_{k}| \ln \varkappa.
    \end{equation*}
    Combining all above results, we have 
    \begin{equation}\label{eq:linear_phase_rates1_bounds2}
    \begin{aligned}
        \sum_{i \in \fI_{k}} \ln \frac{\cos^{2}(\htheta_{i})}{\hatm_{i}} & = \sum_{i = 0}^{k - 1} \ln \frac{\cos^{2}(\htheta_{i})}{\hatm_{i}} - \sum_{i \in \bar{\fI}_{k}} \ln \cos^{2}(\htheta_{i}) + \sum_{i \in \bar{\fI}_{k}} \ln \hatm_{i} \\
        & \geq - \Psi(\bmB_{0}) -  |\bar{\fI}_{k}| \ln \varkappa.
    \end{aligned}
    \end{equation}
    Consequently, we have
    \begin{equation*}
    \begin{aligned}
        \frac{f(\vx_{k}) - f(\vx_{*})}{f(\vx_{0}) - f(\vx_{*})} & ~\overset{\mathclap{\eqref{eq:linear_phase_terms}}}{\leq}~ \bigggl( 1 - \frac{1}{3 \varkappa} \Biggl( \prod_{i \in \fI_{k}} \frac{\cos^{2}(\htheta_{i})}{\hatm_{i}} \Biggr)^{\frac{1}{|\fI_{k}|}} \bigggr)^{|\fI_{k}|} \\
        & ~\overset{\mathclap{\eqref{eq:linear_phase_rates1_bounds2}}}{\leq}~ \biggl( 1 - \frac{1}{3 \varkappa} \exp \biggl( - \frac{\Psi(\bmB_{0}) + |\bar{\fI}_{k}| \ln \varkappa}{|\fI_{k}|} \biggr) \biggr)^{|\fI_{k}|} \\
        & ~\overset{\mathclap{\eqref{eq:linear_phase_large_case_bounds1}}}{\leq}~ \biggl( 1 - \frac{1}{3 \varkappa} \exp \biggl( - \frac{\Psi(\bmB_{0}) + K_{1} \ln \varkappa}{|\fI_{k}|} \biggr) \biggr)^{|\fI_{k}|} \\
        & ~\leq~ \biggl( 1 - \frac{1}{3 \varkappa} \exp \biggl( - \frac{\Psi(\bmB_{0}) + K_{1} \ln \varkappa}{k - K_{1}} \biggr) \biggr)^{k - K_{1}},
    \end{aligned}
    \end{equation*}
    where the last inequality is based on \eqref{eq:linear_phase_rates1_bounds1} and Lemma~\ref{lem:dec_aux_func}. 
    Thus, we obtain \eqref{eq:_linear_phase_linear_rates1}. 
    
    In the case of $k \geq 2 \Psi(\bmB_{0}) + (2 \ln \varkappa + 1) K_{1}$, we have
    \begin{equation*}
        \frac{f(\vx_{k}) - f(\vx_{*})}{f(\vx_{0}) - f(\vx_{*})} \leq \biggl( 1 - \frac{1}{3 \varkappa} \exp \biggl( - \frac{1}{2} \biggr) \biggr)^{k - K_{1}} \leq \biggl( 1 - \frac{1}{6 \varkappa} \biggr)^{k - K_{1}},
    \end{equation*}
    which completes the proof.
\end{proof}

We then present the global linear convergence rate in the regime of $k \geq 2 \Psi(\bmB_{0}) + 2 \sqrt{\varkappa} K_{2}$.

\begin{prop}\label{prop:linear_phase_linear_rates2}
    Under Assumptions~\ref{asm:strongly_convex},~\ref{asm:grad_Lip} and~\ref{asm:self_con}, Algorithm~\ref{alg:ada_bfgs} holds
    \begin{equation}\label{eq:linear_phase_linear_rates2}
        \frac{f(\vx_{k}) - f(\vx_{*})}{f(\vx_{0}) - f(\vx_{*})} \leq \left( 1 - \frac{1}{6 \varkappa} \right)^{k - \sqrt{\varkappa} K_{2}}
    \end{equation}
    for all $k \geq 2 \Psi(\bmB_{0}) + 2 \sqrt{\varkappa} K_{2}$, where $\Psi(\cdot)$, $\bmB_{0}$, and $K_{2}$ follow the definitions in \cref{def:potential_func,def:bar_B,def:K_1_2}, respectively.
\end{prop}

The proof of Proposition~\ref{prop:linear_phase_linear_rates2} requires the following lemma.

\begin{lem}\label{lem:aux_func_2}
    Define $\zeta(u) : [0, \bar{u}) \to \BR$ as 
    \begin{equation}\label{eq:aux_func_2}
        \zeta(u) \defeq \bigggl( 1 - a \Biggl( b \Bigl( \max\Bigl\{\frac{u}{s}, 1 \Bigr\} \Bigr)^{2 u} \biggl( \frac{1}{\varkappa} \biggr)^{u} \Biggr)^{\frac{1}{k - u}} \bigggr)^{k - u},
    \end{equation}
    where $a\in(0,1)$, $b\in(0,1]$, $s > 0$, $\varkappa > 1$ and 
    \begin{equation}\label{def:bar_u}
        \bar{u} \defeq \sup \bigggl\{ u : 0 < u < k, \ a \Biggl( b \Bigl( \max\Bigl\{\frac{u}{s}, 1 \Bigr\} \Bigr)^{2 u} \biggl( \frac{1}{\varkappa} \biggr)^{u} \Biggr)^{\frac{1}{k - u}} < 1 \bigggr\}.
    \end{equation}
    Then if $k \geq 2 \ln (1 / b) + 2 \sqrt{\varkappa} s$, we have the upper bound on the function $\zeta(u)$ as
    \begin{equation}\label{eq:aux_func_2_upper_1}
        \sup_{u \in [0, \bar{u})} \zeta(u) \leq \exp \biggl(- \frac{1}{2} a \bigl(k - \sqrt{\varkappa} s\bigr)\biggr).
    \end{equation}
    If we furthermore suppose $a \leq 1 / 3$, the upper bound can be enhanced to
    \begin{equation}\label{eq:aux_func_2_upper_2}
        \sup_{u \in [0, \bar{u})} \zeta(u) \leq \left( 1 - \frac{a}{2}\right)^{k - \sqrt{\varkappa} s}.
    \end{equation}
\end{lem}

\begin{proof}
    Please refer to Appendix~\ref{proof:aux_func_2}.
\end{proof}

We now prove Proposition~\ref{prop:linear_phase_linear_rates2} as follows.

\begin{proof}[Proof of Proposition~\ref{prop:linear_phase_linear_rates2}]
    Recall that the proof of Proposition~\ref{prop:linear_phase_linear_rates1} has shown that
    \begin{numcases}{}
        \sum_{i = 0}^{k - 1} \ln \dfrac{\cos^{2}(\htheta_{i})}{\hatm_{i}} \geq - \Psi(\bmB_{0}), \label{eq:linear_phase_acc_potential_func} \\
        \sum_{i \in \bar{\fI}_{k}} \ln \cos^{2}(\htheta_{i}) \leq 0, \label{eq:linear_phase_acc_potential_func_rem1} \\
        \sum_{i \in \bar{\fI}_{k}} \ln \hatm_{i} \geq - |\bar{\fI}_{k}| \ln \varkappa. \label{eq:linear_phase_acc_potential_func_rem2}
    \end{numcases}
    Compared with \eqref{eq:linear_phase_acc_potential_func_rem1}, we can also provide another upper bound for $\sum_{i \in \bar{\fI}_{k}} \ln \cos^{2}(\htheta_{i})$. 
    We first consider the case of~$|\bar{\fI}_{k}| \geq 1$. 
    We apply AM-GM inequality and \eqref{eq:linear_phase_large_case_bounds2} to obtain
    \begin{equation*}
        \prod_{i \in \bar{\fI}_{k}} \cos(\htheta_{i}) \leq \left( \frac{\sum_{i \in \bar{\fI}_{k}} \cos(\htheta_{i})}{|\bar{\fI}_{k}|} \right)^{|\bar{\fI}_{k}|} \leq \left( \frac{K_{2}}{|\bar{\fI}_{k}|} \right)^{|\bar{\fI}_{k}|}.
    \end{equation*}
    Therefore, using above upper bound on $\prod_{i \in \bar{\fI}_{k}} \cos(\htheta_{i})$ and following the derivation of \eqref{eq:linear_phase_rates1_bounds2}, we obtain the improved lower bound for $\sum_{i \in \fI_{k}} \ln \frac{\cos^{2}(\htheta_{i})}{\hatm_{i}}$ as follows:
    \begin{equation}\label{eq:bound_3_first_phase_first_case}
        \begin{aligned}
            \sum_{i \in \fI_{k}} \ln \frac{\cos^{2}(\htheta_{i})}{\hatm_{i}} & = \sum_{i = 0}^{k - 1} \ln \frac{\cos^{2}(\htheta_{i})}{\hatm_{i}} - \sum_{i \in \bar{\fI}_{k}} \ln \cos^{2}(\htheta_{i}) + \sum_{i \in \bar{\fI}_{k}} \ln \hatm_{i} \\
            & \geq - \Psi(\bmB_{0}) + \max\biggl\{ 2 |\bar{\fI}_{k}| \ln \biggl( \frac{|\bar{\fI}_{k}|}{K_{2}} \biggr), 0 \biggr\} - |\bar{\fI}_{k}| \ln \varkappa \\
            & = - \Psi(\bmB_{0}) + 2 |\bar{\fI}_{k}| \ln \max\biggl\{\frac{|\bar{\fI}_{k}|}{K_{2}}, 1 \biggr\} - |\bar{\fI}_{k}| \ln \varkappa.
        \end{aligned}
    \end{equation}
    In the case of $|\bar{\fI}_{k}| = 0$, the above lower bound for $\sum_{i \in \fI_{k}} \ln \frac{\cos^{2}(\htheta_{i})}{\hatm_{i}}$ still holds due to \eqref{eq:linear_phase_acc_potential_func}.

    We now consider the case of $|\fI_{k}| \geq 1$, which means $|\bar{\fI}_{k}| < k$. 
    Substituting \eqref{eq:bound_3_first_phase_first_case} into \eqref{eq:linear_phase_terms}, we obtain
    \begin{align}
            \frac{f(\vx_{k}) - f(\vx_{*})}{f(\vx_{0}) - f(\vx_{*})} & \leq \left( 1- \frac{1}{3 \varkappa} \left( \prod_{i \in \fI_{k}} \frac{\cos^{2}(\htheta_{i})}{\hatm_{i}} \right)^{\frac{1}{|\fI_{k}|}} \right)^{|\fI_{k}|} \nonumber\\
            & \leq \left( 1 - \frac{1}{3 \varkappa} \left( \exp \left( - \Psi(\bmB_{0}) \right) \left( \max\left\{ \frac{|\bar{\fI}_{k}|}{K_{2}}, 1 \right\} \right)^{2 |\bar{\fI}_{k}|} \left( \frac{1}{\varkappa} \right)^{|\bar{\fI}_{k}|} \right)^{\frac{1}{|\fI_{k}|}} \right)^{|\fI_{k}|} \nonumber\\
            & = \left( 1 - \frac{1}{3 \varkappa} \left( \exp \left( - \Psi(\bmB_{0}) \right) \left( \max\left\{ \frac{|\bar{\fI}_{k}|}{K_{2}}, 1 \right\} \right)^{2 |\bar{\fI}_{k}|} \left( \frac{1}{\varkappa} \right)^{|\bar{\fI}_{k}|} \right)^{\frac{1}{k - |\bar{\fI}_{k}|}} \right)^{k - |\bar{\fI}_{k}|}. \label{eq:linear-some-line}
    \end{align}
    Applying Lemma~\ref{lem:aux_func_2} with $a = 1 / (3 \varkappa)$, $b = \exp \left(-\Psi(\bmB_{0}) \right)$ and $s = K_{2}$ to \eqref{eq:linear-some-line} and combining with conditions $k \geq 2 \sqrt{\varkappa} K_{2} + 2 \Psi(\bmB_{0})$ and $a \leq 1 / 3$, we obtain
    \begin{equation*}
        \frac{f(\vx_{k}) - f(\vx_{*})}{f(\vx_{0}) - f(\vx_{*})} \leq \left(1 - \frac{1}{6 \varkappa} \right)^{k - \sqrt{\varkappa} K_{2}}.
    \end{equation*}
    This is exactly \eqref{eq:linear_phase_linear_rates2}.

    It remains to rule out the possibility that $|\fI_{k}| = 0$. 
    We proceed by contradiction. 
    Suppose that $|\fI_{k}| = 0$, leading to $|\bar{\fI}_{k}| = k$. 
    Combining \eqref{eq:linear_phase_acc_potential_func} and \eqref{eq:linear_phase_acc_potential_func_rem2}, we have
    \begin{equation}\label{eq:linear-proof-lncos2}
        \sum_{i = 0}^{k - 1} \ln \cos^{2}(\htheta_{i})  \geq - \Psi(\bmB_{0}) + \sum_{i \in \bar{\fI}_{k}} \ln \hatm_{i} 
         \geq - \Psi(\bmB_{0}) - k \ln \varkappa.
    \end{equation}
    Using AM-GM inequality, we have
    \begin{equation*}
        K_{2} \overset{\eqref{eq:linear_phase_large_case_bounds2}}{\geq} \sum_{i = 0}^{k - 1} \cos(\htheta_{i}) \geq k \left( \prod_{i = 0}^{k - 1} \cos(\htheta_{i}) \right)^{\frac{1}{k}} \overset{\eqref{eq:linear-proof-lncos2}}{\geq} k \exp \left( - \frac{\Psi(\bmB_{0}) + k \ln \varkappa}{2 k} \right),
    \end{equation*}
    which leads to
    \begin{equation*}
        \frac{\sqrt{\varkappa} K_{2}}{k} \geq \exp \left( - \frac{\Psi(\bmB_{0})}{2 k} \right).
    \end{equation*}
    However, our proposition requires $k \geq 2 \Psi(\bmB_{0}) + 2 \sqrt{\varkappa} K_{2}$, which means
    \begin{equation*}
        \frac{\sqrt{\varkappa} K_{2}}{k} < \frac{1}{2} 
        \quad \text{and} \quad
        \exp \left( - \frac{\Psi(\bmB_{0})}{2 k} \right) > \exp \left( -\frac{1}{4} \right).
    \end{equation*}
   This implies $1 / 2 > \exp \left( - 1 / 4 \right)$, yielding the contradiction. 
   Therefore, the case $|\fI_{k}| = 0$ cannot occur, so that we complete the proof.
\end{proof}

\begin{remark}    
    Note that both Propositions~\ref{prop:linear_phase_linear_rates1} and~\ref{prop:linear_phase_linear_rates2} enjoy the linear convergence rates of $\fO((1 - 1 / \varkappa)^{k})$, while their iteration thresholds are different.
    Specifically, Proposition~\ref{prop:linear_phase_linear_rates1} has the starting moment of
    \begin{align*}
    2 \Psi(\bmB_{0}) + (2 \ln \varkappa + 1) K_{1} = \fO( \Psi(\bmB_{0}) + M^{2} \Delta \ln \varkappa),
    \end{align*}
    and Proposition~\ref{prop:linear_phase_linear_rates2} has the starting moment of
    \begin{align*}
        2 \Psi(\bmB_{0}) + 2 \sqrt{\varkappa} K_{2} = \fO (\Psi(\bmB_{0}) + \varkappa M \sqrt{\Delta}).
    \end{align*}
    Hence, the results of Proposition~\ref{prop:linear_phase_linear_rates1} are preferable in the case of $M \sqrt{\Delta} \leq \fO({\varkappa}/{\ln \varkappa})$, while Proposition~\ref{prop:linear_phase_linear_rates2} is preferable in the case of $M \sqrt{\Delta} \geq \Omega({\varkappa}/{\ln \varkappa}$).
\end{remark}

We combine Propositions~\ref{prop:linear_phase_linear_rates1} and~\ref{prop:linear_phase_linear_rates2} to obtain the following theorem.

\begin{thm}\label{thm:linear_phase_linear_rates}
    Under Assumptions~\ref{asm:strongly_convex},~\ref{asm:grad_Lip} and~\ref{asm:self_con}, Algorithm~\ref{alg:ada_bfgs} holds
    \begin{equation}\label{eq:linear_phase_linear_rates}
        \frac{f(\vx_{k}) - f(\vx_{*})}{f(\vx_{0}) - f(\vx_{*})} \leq \left( 1 - \frac{1}{6 \varkappa} \right)^{k - K_{3}}
    \end{equation}
    for all $k \geq 2 \Psi(\bmB_{0}) + \min\left\{(2 \ln \varkappa + 1) K_{1} , 2 \sqrt{\varkappa} K_{2}\right\}$, where $\Psi(\cdot)$, $\bmB_{0}$, $K_{1}$ and $K_{2}$ follow the definitions in \cref{def:potential_func,def:bar_B,def:K_1_2}, respectively, and $K_{3}$ is defined as
    \begin{equation}\label{def:K_3}
        K_{3} \defeq \frac{\min\left\{(2 \ln \varkappa + 1) K_{1}, 2 \sqrt{\varkappa} K_{2} \right\}}{\min \left\{ 2 \ln \varkappa + 1, 2 \right\}}.
    \end{equation}
\end{thm}

\begin{proof}
    Here we only provide details for the case of $(2 \ln \varkappa + 1) K_{1} \leq 2 \sqrt{\varkappa} K_{2}$, since the other case can be addressed similarly.
    In this case, the linear convergence shown in \eqref{eq:linear_phase_first_rates1} from Proposition~\ref{prop:linear_phase_linear_rates1} begins when the iteration counter satisfies
    \begin{equation*}
        k \geq 2 \Psi(\bmB_{0}) + (2 \ln \varkappa + 1) K_{1} = 2 \Psi(\bmB_{0}) + \min\left\{(2 \ln \varkappa + 1) K_{1} , 2 \sqrt{\varkappa} K_{2}\right\}.
    \end{equation*}
    Moreover, we follow \eqref{eq:linear_phase_first_rates1} to obtain
    \begin{equation*}
        \frac{f(\vx_{k}) - f(\vx_{*})}{f(\vx_{0}) - f(\vx_{*})} \leq \left( 1 - \frac{1}{6 \varkappa} \right)^{k - K_{1}}  = \left( 1 - \frac{1}{6 \varkappa} \right)^{k - \frac{\min\left\{(2 \ln \varkappa + 1) K_{1}, 2 \sqrt{\varkappa} K_{2}\right\}}{2 \ln \varkappa + 1}} \leq \left( 1 - \frac{1}{6 \varkappa} \right)^{k - K_{3}},
    \end{equation*}
    which establishes \eqref{eq:linear_phase_linear_rates}.
\end{proof}

For ease of presentation, we denote the starting moment of linear convergence in Theorem~\ref{thm:linear_phase_linear_rates} by $K_{4}$, i.e.,
\begin{equation}\label{def:K_4}
\begin{aligned}
    K_{4} & \defeq 2 \Psi(\bmB_{0}) + \min\bigl\{ (2 \ln \varkappa + 1) K_{1}, 2 \sqrt{\varkappa} K_{2}  \bigr\} \\
    & = 2 \Psi(\bmB_{0}) + 4 \min\bigl\{ (2 \ln \varkappa + 1) M^{2} \Delta, 2 \sqrt{2} \varkappa M \sqrt{\Delta} \bigr\},
\end{aligned}
\end{equation}
which has the order of
\begin{equation*}
\begin{aligned}
    K_{4} & = \fO\bigl( \Psi(\bmB_{0}) + \min\bigl\{ M^{2} \Delta \ln \varkappa, \varkappa M \sqrt{\Delta}\bigr\} \bigr) \\
    & = \tilde{\fO} \bigl( \Psi(\bmB_{0}) + \min\bigl\{ M^{2} \Delta, \varkappa M \sqrt{\Delta}\bigr\} \bigr).
\end{aligned}
\end{equation*}

\section{The Phase of Local Superlinear Convergence}\label{sec:superlinear_phase}

This section provides the explicit local superlinear convergence rates for adaptive BFGS~(Algorithm~\ref{alg:ada_bfgs}).
Unlike the global analysis in Section~\ref{sec:global-linear-rates}, we follow Proposition~\ref{prop:analy_framework} by taking $\mP = \nabla^{2} f(\vx_{*})$ for the local convergence.
We first present the following lemma.

\begin{lem}[{\cite[Lemma~4 and~5]{jinNonasymptoticGlobalConvergence2025}}]\label{lem:P_star_q_J_bounds}
    Under Assumptions~\ref{asm:strongly_convex} and~\ref{asm:grad_Lip}, Algorithm~\ref{alg:ada_bfgs} holds
    \begin{equation}\label{eq:P_star_q_lower}
        \hq_{k} \geq \frac{2}{(1 + C_{k})^{2}},
    \end{equation}
    and 
    \begin{equation}\label{eq:P_star_J_upper}
        \frac{\Norm{\hvy_{k}}^{2}}{\hvs_{k}^{\T} \hvy_{k}} \leq 1 + C_{k}
    \end{equation}
    for all $k \geq 0$, where 
    \begin{equation}\label{def:C_k}
    C_{k} \defeq 2 \sqrt{2} \tM \sqrt{f(\vx_{k}) - f(\vx_{*})} 
    \end{equation}
    with $\tM = L_{2} / (2 \mu^{3/2})$, 
    $\hq_{k}$ follows the definition in \eqref{def:q} and $\hvy_{k}$ and $\hvs_{k}$ follow the definitions in \eqref{def:wgt_vec} by taking $\mP = \nabla^{2} f(\vx_{*})$.
\end{lem}

We then establish the superlinear convergence rate for adaptive BFGS as follows.

\begin{thm}\label{thm:superlinear_phase_superlinear_rates}
    Under Assumption~\ref{asm:strongly_convex},~\ref{asm:grad_Lip},~\ref{asm:self_con} and~\ref{asm:Hessian_Lip}, Algorithm~\ref{alg:ada_bfgs} holds
    \begin{equation}\label{eq:superlinear_phase_superlinear_rates}
        \frac{f(\vx_{k}) - f(\vx_{*})}{f(\vx_{0}) - f(\vx_{*})} \leq \Biggl( \frac{\Psi(\tmB_{0}) + 8 \sqrt{2} \tM \sqrt{\Delta} \bigl( 2 \Psi(\bmB_{0}) + 4 \min\bigl\{ (2 \ln \varkappa + 1) M^{2} \Delta, 2 \sqrt{2} \varkappa M \sqrt{\Delta}\bigr\} + 13 \varkappa \bigr)}{k} \Biggr)^{k}
    \end{equation}
    for all $k \geq 1$, where $\Psi(\cdot)$ and $\bmB_{0}$ follow the definitions in \cref{def:potential_func,def:bar_B}, respectively, and $\tmB_{0}$ is defined as
    \begin{equation}\label{def:tilde_B}
        \tmB_{0} \defeq \nabla^{2} f(\vx_{*})^{- \frac{1}{2}} \mB_{0} \nabla^{2} f(\vx_{*})^{- \frac{1}{2}}.
    \end{equation}
\end{thm}

\begin{proof}
    Let $\mP = \nabla^{2} f(\vx_{*})$, then Lemma~\ref{lem:P_star_q_J_bounds} says $\hq_{k} \geq 2/ (1 + C_{k})^{2}$. 
    Combing with Proposition~\ref{prop:analy_framework}, we have
    \begin{equation*}
        \frac{f(\vx_{k + 1}) - f(\vx_{*})}{f(\vx_{k}) - f(\vx_{*})} \leq 1 - \frac{\cos^{2}(\htheta_{k})}{(1 + C_{k})^{2} (1 + 2 M \eta_{k})\hatm_{k}}
    \end{equation*}
    for all $k \geq 0$.
    By taking the product over the above inequality and combining it with lemma~\ref{lem:ineq_arith_geo}, we have
    \begin{equation}\label{eq:superlinear_phase_product}
    \begin{aligned}
        \frac{f(\vx_{k}) - f(\vx_{*})}{f(\vx_{0}) - f(\vx_{*})} & 
       = \prod_{i = 0}^{k - 1} \frac{f(\vx_{i+1})-f(\vx_*)}{f(\vx_i)-f(\vx_*)} \\
        & \leq \prod_{i = 0}^{k - 1} \left( 1 - \frac{1}{(1 + C_{i})^{2} (1 + 2 M \eta_{i})} \frac{\cos^{2}(\htheta_{i})}{\hatm_{i}} \right) \\
        & \leq \left( 1 - \left( \prod_{i = 1}^{k - 1} \frac{1}{(1 + C_{i})^{2} (1 + 2 M \eta_{i})} \frac{\cos^{2}(\htheta_{i})}{\hatm_{i}} \right)^{\frac{1}{k}} \right)^{k} \\
        & = \left( 1 - \exp \left( - \frac{2 \sum_{i = 0}^{k - 1} \ln (1 + C_{i}) + \sum_{i = 0}^{k - 1} \ln (1 + 2 M \eta_{i})}{k} \right) \left( \prod_{i = 0}^{k - 1} \frac{\cos^{2}(\htheta_{i})}{\hatm_{i}} \right)^{\frac{1}{k}} \right)^{k}.
    \end{aligned}
    \end{equation}
    According to \eqref{eq:acc_potential_func} of Proposition~\ref{prop:potential_func_upper} and \eqref{eq:P_star_J_upper} of Lemma~\ref{lem:P_star_q_J_bounds}, we conclude that
    \begin{equation}\label{eq:superlinear_phase_first_term}
        \sum_{i = 0}^{k - 1} \ln\frac{\cos^{2}(\htheta_{i})}{\hatm_{i}} \geq - \Psi(\tmB_{0}) - \sum_{i = 0}^{k - 1} C_{i}.
    \end{equation}
    Moreover, combining \eqref{eq:self_con_f_dec} with \eqref{eq:omega_inv_upper} of Lemma~\ref{lem:omega_bounds}, we obtain
    \begin{equation*}
        M \eta_{k} \leq \omega^{-1} \left( M^{2} (f(\vx_{k}) - f(\vx_{*})) \right) \leq M^{2} (f(\vx_{k}) - f(\vx_{*})) + \sqrt{2} M \sqrt{f(\vx_{k}) - f(\vx_{*})},
    \end{equation*}
    which implies
    \begin{equation}\label{eq:superlinear_phase_second_term}
    \begin{aligned}
        1 + 2 M \eta_{k} \leq & \ 1 + 2 \sqrt{2} M \sqrt{f(\vx_{k}) - f(\vx_{*})} + 2 M^{2} (f(\vx_{k}) - f(\vx_{*})) \\
        = & \left( 1 + \sqrt{2} M \sqrt{f(\vx_{k}) - f(\vx_{*})} \right)^{2}.
    \end{aligned}
    \end{equation}
    Substituting \cref{eq:superlinear_phase_first_term,eq:superlinear_phase_second_term} into \eqref{eq:superlinear_phase_product} leads to
    \begin{equation}\label{eq:superlinear_phase_inter}
    \begin{aligned}
        \frac{f(\vx_{k}) - f(\vx_{*})}{f(\vx_{0}) - f(\vx_{*})} & ~\overset{\mathclap{\eqref{eq:superlinear_phase_first_term}}}{\leq}~ \Biggl( 1 - \exp \Biggl( - \frac{\Psi(\tmB_{0}) + \sum_{i = 0}^{k - 1} C_{i} + 2 \sum_{i = 0}^{k - 1} \ln (1 + C_{i}) + \sum_{i = 0}^{k - 1} \ln (1 + 2 M \eta_{i}) }{k}\Biggr) \Biggr)^{k} \\
        & ~\leq~ \left( \frac{\Psi(\tmB_{0}) + \sum_{i = 0}^{k - 1} C_{i} + 2 \sum_{i = 0}^{k - 1} \ln (1 + C_{i}) + \sum_{i = 0}^{k - 1} \ln (1 + 2 M \eta_{i})}{k} \right)^{k} \\
        & ~\overset{\mathclap{\eqref{eq:superlinear_phase_second_term}}}{\leq}~ \left( \frac{\Psi(\tmB_{0}) + \sum_{i = 0}^{k - 1} C_{i} + 2 \sum_{i = 0}^{k - 1} \ln (1 + C_{i}) + 2 \sum_{i = 0}^{k - 1} \ln \bigl(1 + \sqrt{2} M \sqrt{f(\vx_{i}) - f(\vx_{*})}\bigr) }{k} \right)^{k} \\
        & ~\leq~ \left( \frac{\Psi(\tmB_{0}) + 3 \sum_{i = 0}^{k - 1} C_{i} + 2 \sqrt{2} M \sum_{i = 0}^{k - 1} \sqrt{f(\vx_{i}) - f(\vx_{*})}}{k} \right)^{k} \\
        & ~\overset{\mathclap{\eqref{def:C_k}}}{\leq}~ \left( \frac{\Psi(\tmB_{0}) + 8 \sqrt{2} \tM \sum_{i = 0}^{k - 1} \sqrt{f(\vx_{k}) - f(\vx_{*})}}{k} \right)^{k},
    \end{aligned}
    \end{equation}
    where the second inequality uses the fact that $1 - \exp (- u) \leq u$ for all $u \geq 0$, the fourth inequality employs $\ln (1 + u) \leq u$ for all $u \geq 0$, and the final inequality is based on the definition of $C_{k}$ from \eqref{def:C_k} and the fact $M \leq \tM$.
    
    We now estimates the term $\sum_{i = 0}^{k - 1} \sqrt{f(\vx_{i}) - f(\vx_{*})}$ by splitting it as follows:
    \begin{equation*}
        \sum_{i = 0}^{k - 1} \sqrt{f(\vx_{i}) - f(\vx_{*})} = \sum_{i = 0}^{\ceil{K_{4}} - 1} \sqrt{f(\vx_{i}) - f(\vx_{*})} + \sum_{i = \ceil{K_{4}}}^{k - 1} \sqrt{f(\vx_{i}) - f(\vx_{*})}.
    \end{equation*}
    For the first part, the decease of $f(\vx_{k})$ achieved by \eqref{eq:self_con_f_dec} implies
    \begin{equation*}
        \sum_{i = 0}^{\ceil{K_{4}} - 1} \sqrt{f(\vx_{i}) - f(\vx_{*})} \leq \ceil{K_{4}}\sqrt{f(\vx_{0}) - f(\vx_{*})} \leq (K_{4} + 1) \sqrt{\Delta}.
    \end{equation*}
    For the second part, the linear convergence rate from Theorem~\ref{thm:linear_phase_linear_rates} implies
    \begin{equation*}
        \sum_{i = \ceil{K_{4}}}^{k - 1} \sqrt{f(\vx_{i}) - f(\vx_{*})} \leq \Biggl( \sum_{i = \ceil{K_{4}}}^{k - 1} \gamma^{\frac{i - K_{3}}{2}} \Biggr) \sqrt{\Delta}\leq \Biggl( \sum_{i = 0}^{+\infty} \gamma^{\frac{i}{2}} \Biggr) \sqrt{\Delta} = \frac{1}{1 - \sqrt{\gamma}} \sqrt{\Delta},
    \end{equation*}
    where $\gamma = 1 - 1 / (6 \varkappa)$.
    Combining above results yields
    \begin{equation}\label{eq:superlinear_phase_sum_sqrt_f}
    \begin{aligned}
        \sum_{i = 0}^{k - 1} \sqrt{f(\vx_{i}) - f(\vx_{*})} \leq & \left( K_{4} + 1 + \frac{1}{1 - \sqrt{\gamma}} \right) \sqrt{\Delta} = \left( K_{4} + 1 + \frac{1}{1 - \sqrt{1 - 1 / (6 \varkappa)}} \right) \sqrt{\Delta} \\
        \leq & \left( K_{4} + 1 + \frac{1}{1 - \left(1 - 1 / (12 \varkappa)\right)}\right) \sqrt{\Delta} < (K_{4} + 13 \varkappa) \sqrt{\Delta},
      \end{aligned}
    \end{equation}
    where the second inequality uses the fact that $\sqrt{1 - u} \leq 1 - u / 2$ for all $u \in [0, 1]$.
    
    By substituting \cref{eq:superlinear_phase_sum_sqrt_f,def:K_4} into \eqref{eq:superlinear_phase_inter}, we obtain the desired result shown in \eqref{eq:superlinear_phase_superlinear_rates}.
\end{proof}

The result of \eqref{eq:superlinear_phase_superlinear_rates} indicates that the starting moment of superlinear convergence is
\begin{equation*}
    \begin{aligned}
       & \Psi(\tmB_{0}) + 8 \sqrt{2} \tM \sqrt{\Delta} \bigl( 2 \Psi(\bmB_{0}) + 4 \min\bigl\{ (2 \ln \varkappa + 1) M^{2} \Delta, 2 \sqrt{2} \varkappa M \sqrt{\Delta}\bigr\} + 13 \varkappa\bigr)  \\
       = & \tilde{\fO} \Bigl( \Psi(\tmB_{0}) + \tM \sqrt{\Delta} \bigl( \Psi(\bmB_{0}) + \min\bigl\{ M^{2} \Delta, \varkappa M \sqrt{\Delta}\bigr\} + \varkappa \bigr) \Bigr).
    \end{aligned}
\end{equation*}

\section{The Smoothness Aided Adaptive Step Size}\label{sec:improved}

In this section, we introduce a new adaptive strategy by considering the gradient Lipschitz continuity, which ensures that the step size is never too small to achieve the better convergence behavior. 

\subsection{The Derivation of the Smoothness Aided Step Size}\label{sec:dsass}

Based on the definition of $\eta_{k}$ in \eqref{def:eta}, the step size $t_{k}$ in Algorithm~\ref{alg:ada_bfgs}~(line~\ref{line:tk}) can be written as 
\begin{equation}\label{eq:adaptive_stepsize1}
    t_{k} = \frac{\eta_{k}}{(1 + M \eta_{k})\lNorm{\vd_{k}}{\vx_{k}}},
\end{equation}
which follows that $t_{k} < 1 / (M \lNorm{\vd_{k}}{\vx_{k}})$. 
Therefore, \eqref{eq:adaptive_stepsize1} will be very small for large $\lNorm{\vd_{k}}{\vx_{k}}$. 
On the other hand, the smoothness~(Assumption~\ref{asm:grad_Lip}) of the objective $f$ implies
\begin{equation*}
    f(\vx_{k} + t \vd_{k}) \leq f(\vx_{k}) + t \vg_{k}^{\T} \vd_{k} + \frac{1}{2} L t^{2} \Norm{\vd_{k}}^{2}.
\end{equation*}
Hence, minimizing the right-hand side results in the step size $t_k= - \vg_{k}^{\T} \vd_{k}/(L \Norm{\vd_{k}}^{2})$, which may be larger than the upper bound $ 1 /(M \lNorm{\vd_{k}}{\vx_{k}})$ for \eqref{eq:adaptive_stepsize1}. 
In the following lemma, we demonstrate that the upper (lower) bounds in Lemma~\ref{lem:self_con_bound} can be sharpened by additional assuming the objective is smooth and strongly convex, which is crucial to achieve our new step size.

\begin{lem}\label{lem:self_con_bound_improved}
    Under Assumptions~\ref{asm:strongly_convex},~\ref{asm:grad_Lip} and~\ref{asm:self_con}, we have
    \begin{align}
        & \nabla f(\vx + t \vd)^{\T} \vd \geq \begin{cases}
            \nabla f(\vx)^{\T} \vd + \dfrac{t \lNorm{\vd}{\vx}^{2}}{1 + M t \lNorm{\vd}{\vx}}, & 0 \leq t \leq t^{l}, \\[0.3cm]
            \nabla f(\vx)^{\T} \vd + \dfrac{t^{l} \lNorm{\vd}{\vx}^{2} }{1 + M t^{l} \lNorm{\vd}{\vx}} + \mu (t - t^{l}) \Norm{\vd}^{2}, & t > t^{l},
        \end{cases} \label{eq:self_con_grad_lower_improved} \\
        & \nabla f(\vx + t \vd)^{\T} \vd \leq \begin{cases}
            \nabla f(\vx)^{\T} \vd + \dfrac{t \lNorm{\vd}{\vx}^{2}}{1 - M t \lNorm{\vd}{\vx}}, & 0 \leq t \leq t^{u}, \\[0.3cm]
            \nabla f(\vx)^{\T} \vd + \dfrac{t^{u} \lNorm{\vd}{\vx}^{2}}{1 - M t^{u} \lNorm{\vd}{\vx}} + L (t - t^{u}) \Norm{\vd}^{2}, & t > t^{u},
        \end{cases} \label{eq:self_con_grad_upper_improved} \\
        & f(\vx + t \vd) \leq \begin{cases}
            f(\vx) + t \nabla f(\vx)^{\T} \vd + \dfrac{\omega_{*}(M t \lNorm{\vd}{\vx})}{M^{2}} , & 0 \leq t \leq t^{u}, \\[0.3cm]
            f(\vx) + t \nabla f(\vx)^{\T} \vd + \dfrac{\omega_{*}(M t^{u} \lNorm{\vd}{\vx})}{M^{2}}  + \dfrac{t^{u} (t - t^{u}) \lNorm{\vd}{\vx}^{2} }{1 - M t^{u} \lNorm{\vd}{\vx}} + \dfrac{1}{2}L(t - t^{u})^{2} \Norm{\vd}^{2}, & t > t^{u}
        \end{cases} \label{eq:self_con_f_upper_improved}
    \end{align}
    for all $\vx \in \BR^{n}$, $\vd \in \BR^{n}$ such that $\vd\neq\vzero$, and $t \geq 0$, 
    where $\omega_{*}(z) = - z - \ln(1 - z)$ for $z \in [0, 1)$,
    \begin{equation}\label{eq:def_t_l_u}
        t^{l} \defeq \frac{1}{M \lNorm{\vd}{\vx}} \left( \frac{\lNorm{\vd}{\vx}}{\sqrt{\mu} \Norm{\vd}} - 1 \right), \quad \text{and} \quad
        t^{u} \defeq \frac{1}{M \lNorm{\vd}{\vx}} \left( 1 - \frac{\lNorm{\vd}{\vx}}{\sqrt{L}\Norm{\vd}} \right).
    \end{equation}
\end{lem}

\begin{proof}
    Please refer to Appendix~\ref{proof:self_con_bound_improved}.
\end{proof}

For a given descent direction $\vd_k\in\BR^n$ such that $\vg_{k}^{\T} \vd_{k} < 0$, minimizing the right-hand side of \eqref{eq:self_con_f_upper_improved} with respect to $t$ by taking $\vx=\vx_k$ and $\vd=\vd_k$ leads to the smooth aided adaptive step size
\begin{equation}\label{eq:adaptive_stepsize_improved}
    \ttt_{k} = \begin{cases}
        \dfrac{\eta_{k}}{(1 + M \eta_{k}) \lNorm{\vd_{k}}{\vx_{k}}}, & (1 + M \eta_{k}) \alpha_{k} \leq 1, \\[0.3cm]
        \dfrac{M \eta_{k} \alpha_{k}^{2} + (1 - \alpha_{k})^{2}}{M \lNorm{\vd_{k}}{\vx_{k}}}, & (1 + M \eta_{k}) \alpha_{k} > 1,
    \end{cases}
\end{equation}
where $\eta_{k}$ is defined in \eqref{def:eta}, and $\alpha_{k}$ is defined as
\begin{equation}\label{def:alpha}
    \alpha_{k} \defeq \frac{\lNorm{\vd_{k}}{\vx_{k}}}{\sqrt{L}\Norm{\vd_{k}}}.
\end{equation}
We formally prove that our $\ttt_k$ minimizes the right-hand side of \eqref{eq:self_con_f_upper_improved} in Appendix~\ref{proof:ada_stepsize_improved}.
Recall that BFGS methods \eqref{eq:quasi_update} update the variable along with the direction $\vd_{k} = -\mB_{k}^{-1} \vg_{k}$. 
Combining with the step size $\ttt_k$, we achieve the smoothness aided adaptive BFGS ({\SABFGS}) in Algorithm~\ref{alg:ada_bfgs_improved}.

\begin{algorithm}[t]
    \caption{Smoothness Aided Adaptive BFGS ({\SABFGS})}
    \label{alg:ada_bfgs_improved}
    \begin{algorithmic}[1]
        \STATE $\vx_{0} \in \BR^{n}$, $M > 0$, $L > 0$, $\mB_{0} \succ \vzero$ \\[0.15cm]
        \STATE \textbf{for} {$k = 0, 1, 2, \dots$} \textbf{do} \\[0.15cm]
        \STATE \quad $\vg_{k} = \nabla f(\vx_{k})$, \ $\vd_{k} = - \mB_{k}^{-1} \nabla f(\vx_{k})$ \\[0.15cm]
        \STATE \quad $\eta_{k} = -\dfrac{\vg_{k}^{\T} \vd_{k}}{\lNorm{\vd_{k}}{\vx_{k}}}$, \ $\alpha_{k} = \dfrac{\lNorm{\vd_{k}}{\vx_{k}}}{\sqrt{L} \Norm{\vd_{k}}}$ \\[0.15cm]
        \STATE \quad \textbf{if} {$(M \eta_{k} + 1) \alpha_{k} \leq 1$} \textbf{then} \\[0.15cm]
        \STATE \quad \quad $\ttt_{k} = \dfrac{\eta_{k}}{(1 + M \eta_{k}) \lNorm{\vd_{k}}{\vx_{k}}}$ \\[0.15cm]
        \STATE \quad \textbf{else} \\[0.15cm]
        \STATE \quad \quad $\ttt_{k} = \dfrac{M \eta_{k} \alpha_{k}^{2} + (1 - \alpha_{k})^{2}}{M \lNorm{\vd_{k}}{\vx_{k}}}$ \\[0.15cm]
        \STATE \quad \textbf{end if} \\[0.15cm]
        \STATE \quad $\vx_{k + 1} = \vx_{k} + \ttt_{k} \vd_{k}$ \\[0.15cm]
        \STATE \quad $\vs_{k} = \vx_{k+1}-\vx_k$, \ $\vy_{k} = \nabla f(\vx_{k + 1}) - \nabla f(\vx_{k})$ \\[0.15cm]
        \STATE \quad $\mB_{k + 1} = \mB_{k} - \dfrac{\mB_{k} \vs_{k} \vs_{k}^{\T} \mB_{k}}{\vs_{k}^{\T} \mB_{k} \vs_{k}} + \dfrac{\vy_{k} \vy_{k}^{\T}}{\vy_{k}^{\T} \vs_{k}}$ \\[0.15cm]
        \STATE \textbf{end for}
    \end{algorithmic}
\end{algorithm}

\subsection{Convergence Analysis}

We now present the convergence analysis of our {\SABFGS}~(Algorithm~\ref{alg:ada_bfgs_improved}).
The framework of the proof follows the counterpart in Sections~\ref{sec:linear_phase} and~\ref{sec:superlinear_phase}.
We first show the decrease in the function value and the curvature condition for the decent directions and the step size for {\SABFGS}~(Algorithm~\ref{alg:ada_bfgs_improved}) as follows.

\begin{lem}\label{lem:ada_bfgs_improved_properties}
    Under Assumptions~\ref{asm:strongly_convex},~\ref{asm:grad_Lip} and~\ref{asm:self_con}, we let $\vx_{k + 1} = \vx_{k} + \ttt_{k} \vd_{k}$, where $\vd_{k} \in \BR^{n}$ is a descent direction such that $\vg_{k}^{\T} \vd_{k} < 0$ and $\ttt_{k}$ follows \eqref{eq:adaptive_stepsize_improved}, then we have
    \begin{align}
        & f(\vx_{k + 1}) \leq f(\vx_{k}) - \frac{\omega(M \eta_{k})}{M^{2}}, \label{eq:self_con_f_dec_improved} \\
        & f(\vx_{k + 1}) - f(\vx_{k}) \leq \frac{1}{2} \vg_{k}^{\T} (\vx_{k + 1} - \vx_{k}), \label{eq:armijo_condition_improved} \\
        & \min\left\{ \frac{2 M \eta_{k}}{1 + 2 M \eta_{k}}, 1 - \frac{1}{\varkappa} \right\} \vg_{k}^{\T} \vd_{k} \leq \vg_{k + 1}^{\T} \vd_{k} \leq 0. \label{eq:curvature_condition_improved}
    \end{align}
\end{lem}

\begin{proof}
    Please refer to Appendix~\ref{proof:ada_bfgs_improved_properties}.
\end{proof}

\begin{remark}
    Compared with Lemma~\ref{lem:ada_bfgs_properties} for adaptive BFGS~(Algorithm~\ref{alg:ada_bfgs}), results of Lemma~\ref{lem:ada_bfgs_improved_properties} indicate a tighter lower bound on $\vg_k^\T\vd_k$ in \eqref{eq:curvature_condition_improved} for {\SABFGS}~(Algorithm~\ref{alg:ada_bfgs_improved}).
\end{remark}

\begin{remark}   
Recall that our analysis in Sections~\ref{sec:linear_phase} and~\ref{sec:superlinear_phase} implies the global and local convergence rates of adaptive BFGS~(Algorithm~\ref{alg:ada_bfgs}), as shown in Theorems~\ref{thm:linear_phase_linear_rates} and~\ref{thm:superlinear_phase_superlinear_rates}, only require that we obtain \cref{eq:self_con_f_dec,eq:armijo_condition,eq:curvature_condition}.
Since the result of Lemma~\ref{lem:ada_bfgs_improved_properties} directly means \cref{eq:self_con_f_dec,eq:armijo_condition,eq:curvature_condition}, the global and local convergence rates claimed in Theorems~\ref{thm:linear_phase_linear_rates} and~\ref{thm:superlinear_phase_superlinear_rates} also hold for our {\SABFGS}~(Algorithm~\ref{alg:ada_bfgs_improved}).
Furthermore, because the bounds shown in Lemma~\ref{lem:ada_bfgs_improved_properties} are tighter than the results of \cref{eq:self_con_f_dec,eq:armijo_condition,eq:curvature_condition}, our subsequent analysis will show that {\SABFGS} is possible to enjoy a better convergence behavior than adaptive BFGS.  
\end{remark}

We then show that {\SABFGS}~(Algorithm~\ref{alg:ada_bfgs_improved}) can achieve a tighter upper bound than the counterpart in Proposition~\ref{prop:analy_framework}. 

\begin{prop}\label{prop:analy_framework_improved}
    Under Assumptions~\ref{asm:strongly_convex},~\ref{asm:grad_Lip}, and~\ref{asm:self_con}, Algorithm~\ref{alg:ada_bfgs_improved} holds
    \begin{equation}\label{eq:_analy_framework_improved}
        \frac{f(\vx_{k + 1}) - f(\vx_{*})}{f(\vx_{k}) - f(\vx_{*})}
         \leq 1 - \frac{\hq_{k} \cos^{2}(\htheta_{k})}{2 \min\{ \bigl(1 + \sqrt{2} M \sqrt{f(\vx_{k}) - f(\vx_{*})} \bigr)^{2}, \varkappa\} \hatm_{k}}
    \end{equation}
    for all $k \geq 0$, where $\hq_{k}$ follows the definitions in \eqref{def:q} and $\htheta_{k}$ and $\hatm_{k}$ follow the definitions in \cref{def:wgt_vec,eq:affine-inv,def:theta_m} with any $\mP \in \BS_{++}^{n}$. Furthermore, for all $k \geq 1$, we have
    \begin{equation}\label{eq:analy_framework_improved}
        \frac{f(\vx_{k}) - f(\vx_{*})}{f(\vx_{0}) - f(\vx_{*})} \leq \left( 1- \frac{1}{2} \left( \prod_{i = 0}^{k - 1} \frac{\hq_{i} \cos^{2}(\htheta_{i})}{\min\bigl\{\bigl(1 + \sqrt{2} M \sqrt{f(\vx_{i}) - f(\vx_{*})}\bigr)^{2}, \varkappa\bigr\} \hatm_{i}} \right)^{\frac{1}{k}} \right)^{k}.
    \end{equation}
\end{prop}

\begin{proof}
    The proof of \eqref{eq:_analy_framework_improved} is very similar to that of \eqref{eq:analy_framework}. 
    The only difference is the improved lower bound of $-\hvy_{k}^{\T} \hvs_{k}/\hvg_{k}^{\T} \hvs_{k}$.
    We first apply \eqref{eq:curvature_condition_improved} to obtain
    \begin{equation}\label{eq:_analy_framework_curvature_improved}
        \frac{\hvy_{k}^{\T} \hvs_{k}}{- \hvg_{k}^{\T} \hvs_{k}} 
        = \frac{(\hvg_{k + 1} - \hvg_{k})^{\T} \hvd_{k}}{- \hvg_{k}^{\T} \hvd_{k}} \geq \frac{1}{\min\{1 + 2 M \eta_{k}, \varkappa\}}.
    \end{equation}
    We then use \cref{eq:self_con_f_dec_improved,eq:omega_inv_upper} to bound $M \eta_{k}$ as
    \begin{equation*}
        M \eta_{k} \overset{\eqref{eq:self_con_f_dec_improved}}{\leq} \omega^{-1}(M^{2}(f(\vx_{k}) - f(\vx_{*}))) \overset{\eqref{eq:omega_inv_upper}}{\leq} M^{2} (f(\vx_{k}) - f(\vx_{*})) + \sqrt{2} M \sqrt{f(\vx_{k}) - f(\vx_{*})},
    \end{equation*}
    which implies
    \begin{equation*}
        1 + 2 M \eta_{k} \leq \bigl(1 + \sqrt{2} M \sqrt{f(\vx_{k}) - f(\vx_{*})} \bigr)^{2}.
    \end{equation*}
    Combining above result with \eqref{eq:_analy_framework_curvature_improved}, we have
    \begin{equation}\label{eq:analy_framework_curvature_improved}
        -\frac{\hvy_{k}^{\T} \hvs_{k}}{\hvg_{k}^{\T} \hvs_{k}} \geq \frac{1}{\min\bigl\{\bigl(1 + \sqrt{2} M \sqrt{f(\vx_{k}) - f(\vx_{*})}\bigr)^{2}, \varkappa\bigr\}}.
    \end{equation}
    Now we follow the proof of Proposition~\ref{prop:analy_framework} by replacing \eqref{eq:analy_framework_curvature} with \eqref{eq:analy_framework_curvature_improved}, which yields the desired result of \eqref{eq:_analy_framework_improved}.
    Finally, we apply Lemma~\ref{lem:ineq_arith_geo} to \eqref{eq:_analy_framework_improved}, which yields \eqref{eq:analy_framework_improved}. 
\end{proof}

Our subsequent analysis requires the lower bound for $\hq_{k}$ and the upper bound for $\Norm{\hvy_{k}}^{2} / (\hvs_{k}^{\T} \hvy_{k})$. 
Note that the proofs of Lemmas~\ref{lem:P_L_q_J_bounds} and~\ref{lem:P_star_q_J_bounds}~\cite[Lemmas~4 and~5]{jinNonasymptoticGlobalConvergence2025} only require the step size that ensures the function value decreases, and thus they still hold for Algorithm~\ref{alg:ada_bfgs_improved}. 
Therefore, we can combine Lemmas~\ref{lem:P_L_q_J_bounds} and~\ref{lem:P_star_q_J_bounds} with Proposition~\ref{prop:analy_framework_improved} to analyze the global and local convergence of {\SABFGS}~(Algorithm~\ref{alg:ada_bfgs}).

We show the first linear convergence phase for {\SABFGS}~(Algorithm~\ref{alg:ada_bfgs_improved}) as follows.

\begin{thm}\label{thm:linear_phase_linear_rates_1_improved}
    Under Assumptions~\ref{asm:strongly_convex},~\ref{asm:grad_Lip}, and~\ref{asm:self_con}, Algorithm~\ref{alg:ada_bfgs_improved} holds
    \begin{equation}\label{eq:linear_phase_linear_rates_1_improved}
        \frac{f(\vx_{k}) - f(\vx_{*})}{f(\vx_{0}) - f(\vx_{*})} \leq \Biggl( 1 - \frac{1}{\varkappa \min\{\Gamma, \varkappa\}} \exp \biggl( - \frac{\Psi(\bmB_{0})}{k} \biggr) \Biggr)^{k}
    \end{equation}
    for all $k \geq 1$, where $\Psi(\cdot)$ and $\bmB_{0}$ follow the definitions in \cref{def:potential_func,def:bar_B}, respectively, and $\Gamma$ is defined~as
    \begin{equation}\label{def:Gamma}
        \Gamma \defeq \bigl(1 + \sqrt{2} M \sqrt{\Delta}\bigr)^{2}.
    \end{equation}
    Furthermore, for all $k \geq 2 \Psi(\bmB_{0})$, we have
    \begin{equation}\label{eq:_linear_phase_linear_rates_1_improved}
        \frac{f(\vx_{k}) - f(\vx_{*})}{f(\vx_{0}) - f(\vx_{*})} \leq \left( 1 - \frac{1}{2 \varkappa \min\{\Gamma, \varkappa\}} \right)^{k}.
    \end{equation}
\end{thm}

\begin{proof}
    We follow the definitions of $\hvy_{k}$, $\hvs_{k}$, and $\hq_{k}$ in \cref{def:wgt_vec,def:q} with $\mP = L \mI$, then Lemma~\ref{lem:P_L_q_J_bounds} implies $\hq_{k} \geq 2 / \varkappa$ and $1 - \Norm{\hvy_{k}}^{2} / ({\hvs_{k}^{\T}\hvy_{k}}) \geq 0$. 
    Therefore, \eqref{eq:acc_potential_func} in Proposition~\ref{prop:potential_func_upper} implies
    \begin{equation*}
        \sum_{i = 0}^{k - 1} \ln \frac{\cos^{2} (\htheta_{i})}{\hatm_{i}} \geq - \Psi(\bmB_{0}) + \sum_{i = 0}^{k - 1}\left( 1 - \frac{\Norm{\hvy_{k}}^{2}}{\hvs_{k}^{\T}\hvy_{k}} \right) \geq - \Psi(\bmB_{0}).
    \end{equation*}
    Note that Lemma~\ref{lem:ada_bfgs_improved_properties} implies $\{f(\vx_{k})\}_{k \geq 0}$ is strictly decreasing, then we have
    \begin{equation*}
        \bigl(1 + \sqrt{2} M \sqrt{f(\vx_{k}) - f(\vx_{*})}\bigr)^{2} 
        \leq \bigl(1 + \sqrt{2} M \sqrt{f(\vx_{0}) - f(\vx_{*})}\bigr)^{2} 
        = (1 + \sqrt{2} M \sqrt{\Delta})^{2} = \Gamma.
    \end{equation*}
    Substituting above results into \eqref{eq:analy_framework_improved} of Proposition~\ref{prop:analy_framework_improved} yields
    \begin{equation*}
    \begin{aligned}
        \frac{f(\vx_{k}) - f(\vx_{*})}{f(\vx_{0}) - f(\vx_{*})} & \leq \left( 1 - \frac{1}{\varkappa \min\{ \Gamma, \varkappa \}} \left( \prod_{i = 0}^{k - 1} \frac{\cos^{2}(\htheta_{i})}{\hatm_{i}} \right)^{\frac{1}{k}} \right)^{k} \\
        & \leq \biggl( 1 - \frac{1}{\varkappa \min\{\Gamma, \varkappa\}} \exp \biggl( - \frac{\Psi(\bmB_{0})}{k} \biggr) \biggr)^{k}.
    \end{aligned}
    \end{equation*}
    Thus, we have proved \eqref{eq:linear_phase_linear_rates_1_improved}. 
    Furthermore, in the case of $k \geq 2 \Psi(\bmB_{0})$, we have
    \begin{align*}
    \exp \left(-\frac{\Psi(\bmB_{0})}{k}\right) \geq \exp \left(-\frac{1}{2}\right) \geq \frac{1}{2}.
    \end{align*}
    Hence, \eqref{eq:_linear_phase_linear_rates_1_improved} follows immediately.
\end{proof}

We then show the second linear convergence phase for {\SABFGS}~(Algorithm~\ref{alg:ada_bfgs_improved}), where we achieve a convergence rate of $\fO((1-1/(2\varkappa))^{k})$, which is faster than the rate of $\fO((1-1/(2\kappa\min\{\Gamma,\varkappa\}))^{k})$ in the first phase (Theorem~\ref{thm:linear_phase_linear_rates_1_improved}).

\begin{thm}\label{thm:linear_phase_linear_rates_2_improved}
    Under Assumptions~\ref{asm:strongly_convex},~\ref{asm:grad_Lip}, and~\ref{asm:self_con}, Algorithm~\ref{alg:ada_bfgs_improved} holds
    \begin{equation}\label{eq:linear_phase_linear_rates_2_improved}
        \frac{f(\vx_{k}) - f(\vx_{*})}{f(\vx_{0}) - f(\vx_{*})} \leq \Biggl( 1 - \frac{1}{\varkappa} \exp \Biggl( - \frac{(\Psi(\bmB_{0})  + 1)(2 \ln \min\{ \Gamma, \varkappa\} + 1) + \varkappa \min \{\Gamma, \varkappa\} \bigl( (\ln \Gamma)^{2} + 24 \bigr) }{k} \Biggr) \Biggr)^{k}
    \end{equation}
    for all $k \geq 1$, where $\Psi(\cdot)$, $\bmB_{0}$, and $\Gamma$ follow the definitions in \cref{def:potential_func,def:bar_B,def:Gamma}, respectively. Furthermore, for all $k \geq 2 (\Psi(\bmB_{0}) + 1) (2 \ln \min\{\Gamma, \varkappa\} + 1) + 2 \varkappa \min\{\Gamma, \varkappa\} ((\ln \Gamma)^{2} + 24)$, we have
    \begin{equation}\label{eq:linear_phase_linear_rates_2_improved_2}
        \frac{f(\vx_{k}) - f(\vx_{*})}{f(\vx_{0}) - f(\vx_{*})} \leq \left( 1 - \frac{1}{2 \varkappa} \right)^{k}.
    \end{equation}
\end{thm}

\begin{proof}
    Following the proof of Theorem~\ref{thm:linear_phase_linear_rates_1_improved}, we have
    \begin{align*}
     \hq_{k} \geq \frac{2}{\varkappa} \quad\text{and}\quad
     \sum_{i = 0}^{k - 1} \ln \frac{\cos^{2}(\htheta_{i})}{\hatm_{i}} \geq - \Psi(\bmB_{0}).
    \end{align*}
    From \eqref{eq:analy_framework_improved}, it follows that
    \begin{equation}\label{eq:_linear_phase_2_improved}
    \begin{aligned}
        \frac{f(\vx_{k}) - f(\vx_{*})}{f(\vx_{0}) - f(\vx_{*})} & \leq \left( 1 - \frac{1}{\varkappa} \left( \prod_{i = 0}^{k - 1} \frac{1}{\min\bigl\{\bigl(1 + \sqrt{2} M \sqrt{f(\vx_{k}) - f(\vx_{*})}\bigr)^{2}, \varkappa\bigr\}} \right)^{\frac{1}{k}} \exp \left( - \frac{\Psi(\bmB_{0})}{k}\right) \right)^{k} \\
        & = \left( 1 - \frac{1}{\varkappa} \exp \left( - \frac{\Psi(\bmB_{0}) + \sum_{i = 0}^{k - 1} \ln \min\bigl\{\bigl(1 + \sqrt{2} M \sqrt{f(\vx_{k}) - f(\vx_{*})}\bigr)^{2}, \varkappa\bigr\}}{k} \right) \right)^{k}.
    \end{aligned}
    \end{equation}
    Thus, it remains to provide an upper bound for 
    \begin{align}\label{eq:_linear_phase_2_improved_terms}
    \begin{split}
    & \sum_{i = 0}^{k - 1} \ln \min\big\{\bigl(1 + \sqrt{2} M \sqrt{f(\vx_{i}) - f(\vx_{*})}\bigr)^{2}, \varkappa\big\} \\
    = & \underbrace{\sum_{i = 0}^{\ceil*{2 \Psi(\bmB_{0})}} \ln \min\big\{\bigl(1 + \sqrt{2} M \sqrt{f(\vx_{i}) - f(\vx_{*})}\bigr)^{2}, \varkappa\big\}}_{S_1} \\
    & + \underbrace{\sum_{i = \ceil*{2 \Psi(\bmB_{0})} + 1}^{k - 1} \ln \min\big\{\bigl(1 + \sqrt{2} M \sqrt{f(\vx_{i}) - f(\vx_{*})}\bigr)^{2}, \varkappa\big\}}_{S_2}.
    \end{split}
    \end{align}
    For the term $S_{1}$, the strict monotonic decrease of $\{f(\vx_{k})\}_{k \geq 0}$ given by Lemma~\ref{lem:ada_bfgs_improved_properties} yields
    \begin{equation}\label{eq:_linear_phase_2_improved_terms_1}
    \begin{aligned}
        S_{1} \leq \left( \ceil*{2 \Psi(\bmB_{0})} + 1 \right) \ln \min\{(1 + \sqrt{2} M \sqrt{f(\vx_0) - f(\vx_{*})})^{2}, \varkappa\} \leq 2\left( \Psi(\bmB_{0}) + 1 \right) \ln \min\{\Gamma, \varkappa\}.
    \end{aligned}
    \end{equation}
    For the term $S_2$, we use the global linear convergence rate in Theorem~\ref{thm:linear_phase_linear_rates_1_improved} to achieve
    \begin{equation}\label{eq:_linear_phase_2_improved_terms_2_origin}
    \begin{aligned}
        S_{2} & ~=~ \sum_{i = \ceil*{2 \Psi(\bmB_{0})} + 1}^{k - 1} \ln \min\bigl\{\bigl(1 + \sqrt{2} M \sqrt{f(\vx_{i}) - f(\vx_{*})}\bigr)^{2},\varkappa\bigr\} \\
        & ~\leq~ 2 \sum_{i = \ceil*{2 \Psi(\bmB_{0})} + 1}^{k - 1} \ln \bigl(1 + \sqrt{2} M \sqrt{f(\vx_{i}) - f(\vx_{*})}\bigr) \\
        & ~\overset{\mathclap{\eqref{eq:_linear_phase_linear_rates_1_improved}}}{\leq}~ 2 \sum_{i = 1}^{+\infty} \ln \bigl(1 + \sqrt{2} M \sqrt{\Delta} \tgamma^{\frac{i}{2}}\bigr) \\
        & ~\leq~ 2 \int_{0}^{+\infty} \ln \bigl(1 + \sqrt{2} M \sqrt{\Delta} \tgamma^{\frac{u}{2}}\bigr) \, \mathrm{d} u,
    \end{aligned}
    \end{equation}
    where $\tgamma \defeq 1 - 1 / (2 \varkappa \min\{\Gamma, \varkappa\})$. 
    Here the last inequality in above derivation is due to the function $\ln (1 + \sqrt{2} M \sqrt{\Delta} \tgamma^{\frac{u}{2}})$ is decreasing with respect to $u$.
    
    The integral term in above result can be bounded as 
    \begin{equation}\label{eq:_linear_phase_2_improved_terms_2_integral}
    \begin{aligned}
            & \int_{0}^{+\infty} \ln \bigl(1 + \sqrt{2} M \sqrt{\Delta} \tgamma^{\frac{u}{2}}\bigr) \, \mathrm{d} u \\
        = & \frac{2}{\ln \tgamma} \int_{0}^{+\infty} \frac{\ln \bigl(1 + \sqrt{2} M \sqrt{\Delta} \tgamma^{\frac{u}{2}}\bigr)}{\sqrt{2} M \sqrt{\Delta} \tgamma^{\frac{u}{2}}} \, \mathrm{d} \bigl(\sqrt{2} M \sqrt{\Delta} \tgamma^{\frac{u}{2}}\bigr) \\
        = & - \frac{2}{\ln \tgamma} \int_{0}^{\sqrt{2} M \sqrt{\Delta}} \frac{\ln (1 + y)}{y} \, \mathrm{d} y \\
        = & - \frac{2}{\ln \tgamma} \int_{0}^{\sqrt{2} M \sqrt{\Delta}} \frac{\ln (1 + y)}{1 + y} \, \mathrm{d} y - \frac{2}{\ln \tgamma} \int_{0}^{\sqrt{2} M \sqrt{\Delta}} \frac{\ln (1 + y)}{y(1 + y)} \, \mathrm{d} y \\
        = & - \frac{1}{\ln \tgamma} \left( \ln (1 + \sqrt{2} M \sqrt{\Delta}) \right)^{2} - \frac{2}{\ln \tgamma} \int_{0}^{\sqrt{2} M \sqrt{\Delta}} \frac{\ln (1 + y)}{y(1 + y)} \, \mathrm{d} y \\
        \leq & - \frac{1}{4 \ln \tgamma} (\ln \Gamma)^{2} - \frac{2}{\ln \tgamma} \int_{0}^{+\infty} \frac{\ln (1 + y)}{y(1 + y)} \, \mathrm{d} y.
    \end{aligned}
    \end{equation}
    Since it holds
    $\frac{\ln(1 + y)}{y(1 + y)} \to 1$ as $y \to 0$ and $\frac{ y^{\alpha}\ln (1 + y)}{y(1 + y)} \to 0$ as $y \to +\infty$ for any $\alpha\in(1,2)$, we know that $\frac{\ln (1 + y)}{y(1 + y)}$ is integrable over $[0, +\infty]$. We then denote
    \begin{equation}\label{def:E}
        E = \int_{0}^{+\infty} \frac{\ln (1 + y)}{y(1 + y)} \,\mathrm{d} y,
    \end{equation}
    which is smaller than $3$~(see Appendix~\ref{proof:constant_E}). Moreover, since
    \begin{equation*}
        - \frac{1}{\ln \tgamma} = - \frac{1}{\ln \left(1 - 1 / (2 \varkappa \min\{\Gamma, \varkappa\}) \right)} \leq 2 \varkappa \min\{\Gamma, \varkappa \},
    \end{equation*}
    where we use the fact that $\ln (1 - u) \geq - u$ for all $u \in [0, 1)$. 
    We further combine \cref{eq:_linear_phase_2_improved_terms_2_origin,eq:_linear_phase_2_improved_terms_2_integral} with the facts $E < 3$ to obtain an upper bound for $S_{2}$ as 
    \begin{equation}\label{eq:_linear_phase_2_improved_terms_2}
        S_{2} \leq \varkappa \min\{\Gamma, \varkappa\} \bigl((\ln \Gamma)^{2} + 24\bigr).
    \end{equation}
    Combining \cref{eq:_linear_phase_2_improved_terms,eq:_linear_phase_2_improved_terms_1,eq:_linear_phase_2_improved_terms_2}, we obtain
    \begin{equation*}
    \begin{aligned}
             \sum_{i = 0}^{k - 1} \ln \min\bigl\{\bigl(1 + \sqrt{2} M \sqrt{f(\vx_{k}) - f(\vx_{*})}\bigr)^{2}, \varkappa\bigr\} 
        \leq& 2 (\Psi(\bmB_{0}) + 1) \ln \min\{\Gamma, \varkappa\} + \varkappa \min\{\Gamma, \varkappa\} \bigl( (\ln \Gamma)^{2} + 24\bigr).
    \end{aligned}
    \end{equation*}
    Substituting above result into \eqref{eq:_linear_phase_2_improved}, 
    we establish \eqref{eq:linear_phase_linear_rates_2_improved}.

    When $k \geq 2 (\Psi(\bmB_{0}) + 1)(2 \ln \min\{ \Gamma, \varkappa\} + 1) + 2 \varkappa \min \{ \Gamma, \varkappa\}((\ln \Gamma)^{2} + 24)$, we have
    \begin{equation*}
        \frac{f(\vx_{k}) - f(\vx_{*})}{f(\vx_{0}) - f(\vx_{*})} \overset{\eqref{eq:linear_phase_linear_rates_2_improved}}\leq \left( 1 - \frac{1}{\varkappa} \exp \left( -\frac{1}{2} \right) \right)^{k} \leq \left( 1 - \frac{1}{2 \varkappa} \right)^{k},
    \end{equation*}
    which completes the proof.
\end{proof}

\begin{remark}
    Recall that we have define
    $\Gamma = (1 + \sqrt{2} M \sqrt{\Delta})^{2} = \fO(M^{2} \Delta)$,
    then the starting moment of convergence rate $\fO((1 - 1 / \varkappa)^{k})$ achieved by Theorem~\ref{thm:linear_phase_linear_rates_2_improved} is 
    \begin{equation*}
    \begin{aligned}
        & 2 (\Psi(\bmB_{0}) + 1)(2 \ln \min\{\Gamma, \varkappa\} + 1) + 2 \varkappa \min\{\Gamma, \varkappa\} \bigl((\ln \Gamma)^{2} + 24\bigr) \\
        = & \fO\Bigl(\Psi(\bmB_{0}) \ln \min\{ M^{2} \Delta, \varkappa\} + \varkappa \min\{M^{2} \Delta, \varkappa\} \bigl(\ln M \sqrt{\Delta}\bigr)^{2} \Bigr) \\
        = & \tilde{\fO}\Bigl(\Psi(\bmB_{0}) + \min\bigl\{\varkappa M^{2} \Delta, \varkappa^{2}\bigr\} \Bigr).
    \end{aligned}
    \end{equation*}
    We point out that when $M \sqrt{\Delta} \geq \tilde{\Omega}(\max\{\varkappa, \Psi(\bmB_{0}) / \varkappa\})$, Theorem~\ref{thm:linear_phase_linear_rates_2_improved} implies our {\SABFGS}~(Algorithm~\ref{alg:ada_bfgs_improved}) requires fewer iterations than adaptive BFGS~(Algorithm~\ref{alg:ada_bfgs} with Theorem~\ref{thm:linear_phase_linear_rates}) to achieve the phase of global linear convergence rate $\fO((1 - 1 / \varkappa)^{k})$.
\end{remark}

Finally, we present the specific superlinear convergence rate for Algorithm~\ref{alg:ada_bfgs_improved}. 

\begin{thm}\label{thm:superlinear_phase_superlinear_rates_improved}
    Under Assumptions~\ref{asm:strongly_convex},~\ref{asm:grad_Lip},~\ref{asm:self_con}, and~\ref{asm:Hessian_Lip}, Algorithm~\ref{alg:ada_bfgs_improved} holds
    \begin{equation}\label{eq:superlinear_phase_improved_rate}
        \frac{f(\vx_{k}) - f(\vx_{*})}{f(\vx_{0}) - f(\vx_{*})} \leq \Biggl( \frac{\Psi(\tmB_{0}) + 8 \sqrt{2} \tM \sqrt{\Delta} (2 \Psi(\bmB_{0}) + 1 + 4 \varkappa \min\{\Gamma, \varkappa\})}{k} \Biggr)^{k}
    \end{equation}
    for all $k \geq 1$, where $\Psi(\cdot)$, $\bmB_{0}$, $\tmB_{0}$, and $\Gamma$ follow the definitions in \cref{def:potential_func,def:bar_B,def:tilde_B,def:Gamma}, respectively.
\end{thm}

\begin{proof}
    We follow definitions of $\hvy_{k}$, $\hvs_{k}$, $\hq_{k}$ and $C_{k}$ in \cref{def:wgt_vec,def:q,def:C_k} with $\mP = \nabla^{2} f(\vx_{*})$, then Lemma~\ref{lem:P_L_q_J_bounds} implies that 
    \begin{align}\label{eq:_suplinear_phase_improved_q_lower}
        \hq_{k} \geq 2 / (1 + C_{k})^{2} \quad\text{and}\quad 1 - \Norm{\hvy_{k}}^{2} / ({\hvs_{k}^{\T}\hvy_{k}}) \geq - C_{k}.
    \end{align}
    Therefore, \eqref{eq:acc_potential_func} in Proposition~\ref{prop:potential_func_upper} implies
    \begin{equation}\label{eq:_superlinear_acc_potential}
        \sum_{i = 0}^{k - 1} \ln \frac{\cos^{2}(\htheta_{i})}{\hatm_i} \geq - \Psi(\tmB_{0}) - \sum_{i = 0}^{k - 1} C_i.
    \end{equation}
    Substituting the above results into \eqref{eq:analy_framework_improved}, we obtain
    \begin{equation}\label{eq:superlinear_phase_improved_terms}
    \begin{aligned}
        & \frac{f(\vx_{k}) - f(\vx_{*})}{f(\vx_{0}) - f(\vx_{*})} \\
        \leq & \left( 1- \frac{1}{2} \left( \prod_{i = 0}^{k - 1} \frac{\hq_{i}\cos^{2}(\htheta_{i})}{\min\{1 + 2 M \eta_{i}, \varkappa\} \hatm_{i}}\right)^{\frac{1}{k}} \right)^{k} \\
        \leq & \left(1 - \exp \left( -\frac{\Psi(\tmB_{0}) + \sum_{i = 0}^{k - 1} C_{i} + 2 \sum_{i = 0}^{k - 1} \ln (1 + C_{i}) + 2 \sum_{i = 0}^{k - 1} \ln \bigl(1 + \sqrt{2} M \sqrt{f(\vx_{k}) - f(\vx_{*})}\bigr)}{k} \right)  \right)^{k} \\
        \leq & \left( \frac{\Psi(\tmB_{0}) + 3 \sum_{i = 0}^{k - 1} C_{i} + 2 \sqrt{2} M \sum_{i = 0}^{k - 1} \sqrt{f(\vx_{k}) - f(\vx_{*})}}{k} \right)^{k} \\
        \leq & \left( \frac{\Psi(\tmB_{0}) + 8 \sqrt{2} \tM \sum_{i = 0}^{k - 1} \sqrt{f(\vx_{k}) - f(\vx_{*})}}{k} \right)^{k},
    \end{aligned}
    \end{equation}
    where the second inequality is based on \cref{eq:_suplinear_phase_improved_q_lower,eq:_superlinear_acc_potential}, the third inequality employs the facts $1 - \exp (- u) \leq u$ and $\ln (1 + u) \leq u$ for all $u \geq 0$, and the last inequality is due to the definition of $C_{k}$ from \eqref{def:C_k} and the fact $M \leq \tM$.

    We provide the upper bound for $\sum_{i = 0}^{k - 1}\sqrt{f(\vx_{i}) - f(\vx_{*})}$ as follows:
    \begin{equation}\label{eq:_superlinear_phase_sum_sqrt_f_improved}
    \begin{aligned}
        \sum_{i = 0}^{k - 1} \sqrt{f(\vx_{k}) - f(\vx_{*})} = & \sum_{i = 0}^{\ceil{2 \Psi(\bmB_{0})} - 1} \sqrt{f(\vx_{k}) - f(\vx_{*})} + \sum_{i = \ceil{2 \Psi(\bmB_{0})}}^{k - 1} \sqrt{f(\vx_{k}) - f(\vx_{*})} \\
        \leq & \Biggl( \ceil{2 \Psi(\bmB_{0})} + \sum_{i = \ceil{2 \Psi(\bmB_{0})}}^{k - 1} \tgamma^{\frac{i}{2}} \Biggr) \sqrt{\Delta} \\
        \leq & \Biggl( 2 \Psi(\bmB_{0}) + 1 + \sum_{i = 0}^{+\infty} \tgamma^{\frac{i}{2}} \Biggr) \sqrt{\Delta} \\
        = & \biggl( 2 \Psi(\bmB_{0}) + 1 + \frac{1}{1 - \sqrt{\tgamma}} \biggr) \sqrt{\Delta},
    \end{aligned}
    \end{equation}
    where $\tgamma = 1 - 1 / (2 \varkappa \min\{\Gamma, \varkappa\})$. 
    Here the first inequality is due to the monotonic decrease of $\{f(\vx_{k})\}_{k \geq 0}$ derived from \eqref{eq:self_con_f_dec} and the global linear convergence rate stated in Theorem~\ref{thm:linear_phase_linear_rates_1_improved}. 
    
    For the last line of \eqref{eq:_superlinear_phase_sum_sqrt_f_improved}, we have
    \begin{equation*}
        \frac{1}{1 - \sqrt{\tgamma}} = \frac{1}{1 - \sqrt{1 - 1/(2 \varkappa \min\{\Gamma, \varkappa\})}} \leq 4 \varkappa \min\{\Gamma, \varkappa\},
    \end{equation*}
    which implies
    \begin{equation}\label{eq:superlinear_phase_sum_sqrt_f_improved}
        \sum_{i = 0}^{k - 1} \sqrt{f(\vx_{k}) - f(\vx_{*})} \leq (2 \Psi(\bmB_{0}) + 1 + 4 \varkappa \min\{ \Gamma, \varkappa \}) \sqrt{\Delta}.
    \end{equation}
    Substituting \eqref{eq:superlinear_phase_sum_sqrt_f_improved} into \eqref{eq:superlinear_phase_improved_terms} yields the desired result in  \eqref{eq:superlinear_phase_improved_rate}.
\end{proof}

\begin{remark}
    The starting moment of superlinear convergence in Theorem~\ref{thm:superlinear_phase_superlinear_rates_improved} is 
    \begin{equation*}
    \begin{aligned}
        & \Psi(\tmB_{0}) + 8 \sqrt{2} \tM \sqrt{\Delta} (2 \Psi(\bmB_{0}) + 1 + 4 \varkappa \min\{\Gamma, \varkappa\}) \\
        = & \fO \Bigl(\Psi(\tmB_{0}) + \tM \sqrt{\Delta} \bigl( \Psi(\bmB_{0}) + \min\bigl\{\varkappa M^{2} \Delta, \varkappa^{2}\bigr\}\bigr)\Bigr).
    \end{aligned}
    \end{equation*}
    Compared with the result of Theorem~\ref{thm:superlinear_phase_superlinear_rates} for adaptive BFGS~(Algorithm~\ref{alg:ada_bfgs}), the starting moment shown in Theorem~\ref{thm:superlinear_phase_superlinear_rates_improved} implies that our {\SABFGS}~(Algorithm~\ref{alg:ada_bfgs_improved}) requires fewer iterations to achieve the superlinear phase with the convergence rate of $\fO((1/k)^{k})$ when $M \sqrt{\Delta} \geq  \tilde{\Omega}(\varkappa)$.
\end{remark}

\section{Numerical Experiments}\label{sec:experiments}

In this section, we compare the proposed smoothness aided adaptive BFGS ({\SABFGS}, Algorithm~\ref{alg:ada_bfgs_improved}) with adaptive BFGS (A-BFGS, Algorithm~\ref{alg:ada_bfgs}) \cite{gaoQuasiNewtonMethodsSuperlinear2019}, and the BFGS method with Armijo--Wolfe line search (LS-BFGS) \cite{jinNonasymptoticGlobalConvergence2024b}. We focus on the $\ell_2$-regularized logistic regression, which can be formulated by 
\begin{equation}\label{eq:logistic-regression}
    \min_{\vx \in \BR^{n} } f(\vx) := \frac{1}{m} \sum_{i = 1}^{m} \ln (1 + \exp (-b_{i} \va_{i}^{\top} \vx_{i})) + \frac{1}{2 m} \Norm{\vx}^{2},
\end{equation}
where $\va_i\in\BR^n$ is the feature of the $i$-th sample, $b_i\in\{+1,-1\}$ is the corresponding label. We can verify that the self-concordant and smoothness parameters of the objective in \eqref{eq:logistic-regression} have upper bounds
\begin{equation}\label{eq:logistic_M_L}
    M_{\mathrm{lr}} = \frac{A \sqrt{m}}{2} 
    \quad \text{and} \quad 
    L_{\mathrm{lr}} = \frac{\lambda_{\max}(\mA \mA^{\top})}{4 m}  + \frac{1}{m},
\end{equation}
respectively \cite{zhangCommunicationefficientDistributedOptimization2018,linExplicitConvergenceRates2022}.
Here, we denote $A := \max_{i \in [m]} \Norm{\va_{i}}$ and $\mA := [\va_{1}, \va_{2}, \dots, \va_{n}] \in \BR^{n \times m}$. 
In addition, the strongly convex parameter of the objective can be lower bounded by $\mu_{\mathrm{lr}}= 1 / m$. 

We conduct our experiments on three datasets ``mushrooms'' ($n = 112$, $m = 8124$), ``w8a''($n = 300$, $m = 49,749$), and ``real-sim''($n = 20,958$, $m = 72,309$) from LIBSVM repository~\cite{changLIBSVMLibrarySupport2011}. 
We set the initial point as $\vx_{0} = \vone \in \BR^{n}$, and let the initial estimator of the inverse Hessian be either $\mB_{0} = L_{\mathrm{lr}} \mI$ or $\mB_{0} = \mu_{\mathrm{lr}} \mI$.
We tune the parameters $M$ and $L$ for both {\SABFGS} and A-BFGS from $\{M_{\mathrm{lr}}, M_{\mathrm{lr}} / 10, M_{\mathrm{lr}} / 10^{2}, M_{\mathrm{lr}} / 10^{3}\}$ and $\{L_{\mathrm{lr}}, L_{\mathrm{lr}} / 5, L_{\mathrm{lr}} / 5^{2}, L_{\mathrm{lr}} / 5^{3}\}$, respectively, where $M_{\mathrm{lr}}$ and $L_{\mathrm{lr}}$ follow the definitions in \eqref{eq:logistic_M_L}.
For LS-BFGS, we follow the setup of \citet[Algorithm~1]{jinNonasymptoticGlobalConvergence2024b} by using the logarithmic bisection to determine the step size with the line search parameters $\alpha = 0.1$ and $\beta = 0.9$.

We present the results for the function value gap $f(\vx) - f(\vx_{*})$ against the number of iterations in Figure~\ref{fig:logistic_iter}.
We observe that the proposed {\SABFGS} enter the superlinear convergence phase with fewer iterations in most cases compared to the baselines.
On the other hand, each iteration of A-BFGS and {\SABFGS} requires one additional Hessian–vector product (HVP) to determine the step size, while each iteration of LS-BFGS may involve multiple function value and gradient evaluations due to the line search.
Therefore, we also present the results of the function value gap against the number of function value, gradient, and HVP evaluations in Figure~\ref{fig:logistic_eval}.
We observe that the proposed {\SABFGS} performs better than the baseline methods in terms of the number of evaluations. 
In addition, initialization with $\mB_{0} = \mu_{\mathrm{lr}} \mI$ typically yields better convergence behaviors than initialization with $\mB_{0} = L_{\mathrm{lr}} \mI$, which is consistent with the empirical study of \citet{jinNonasymptoticGlobalConvergence2024b}. 

\begin{figure}[t]
    \centering
    \begin{tabular}{ccc}
        \includegraphics[scale=0.38]{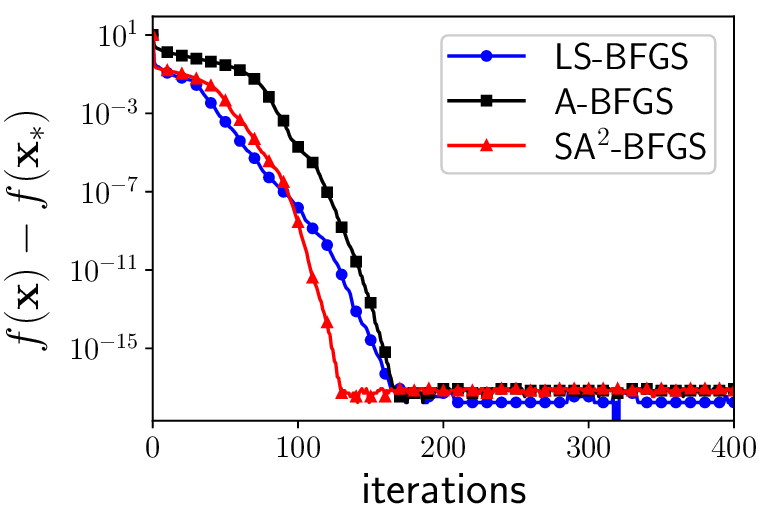}  & 
        \includegraphics[scale=0.38]{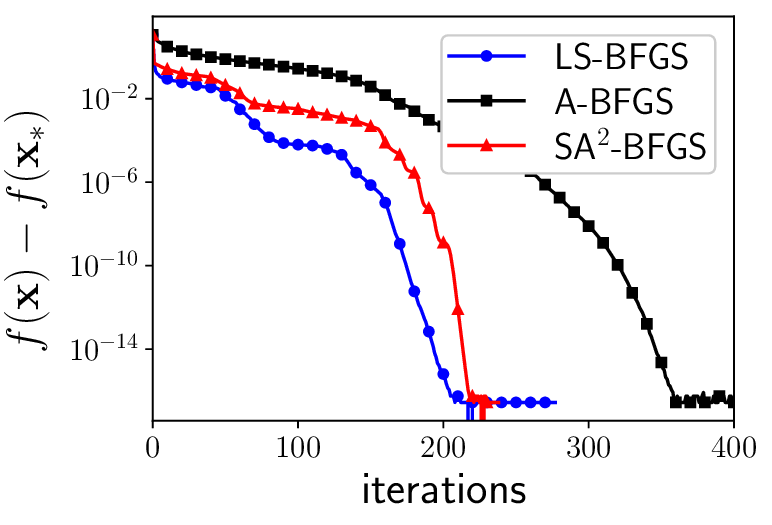} & 
        \includegraphics[scale=0.38]{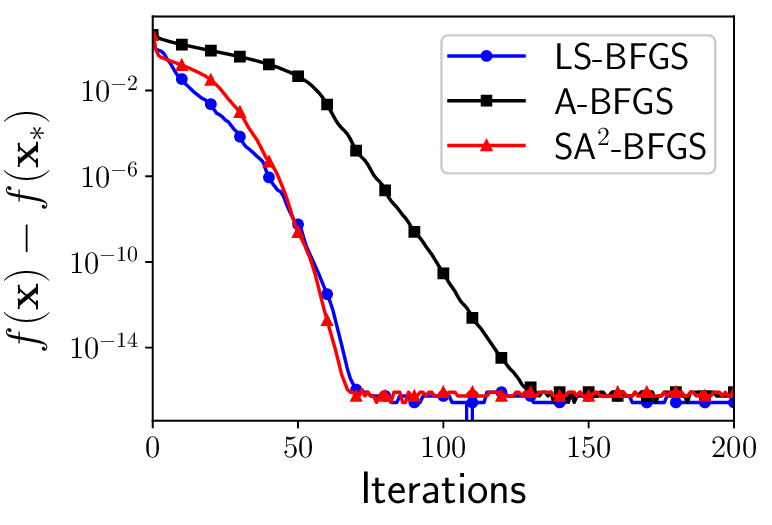} \\
        \small (a) mushrooms with $\mB_0=\mu_{\rm lr}\mI$  & 
        \small (b) w8a with $\mB_0=\mu_{\rm lr}\mI$ & 
        \small (c) real-sim with $\mB_0=\mu_{\rm lr}\mI$ \\[0.25cm]
        \includegraphics[scale=0.38]{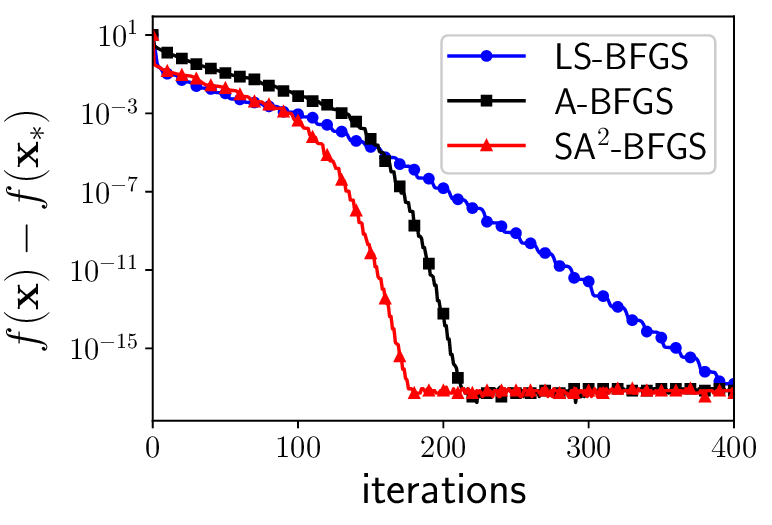}  & 
        \includegraphics[scale=0.38]{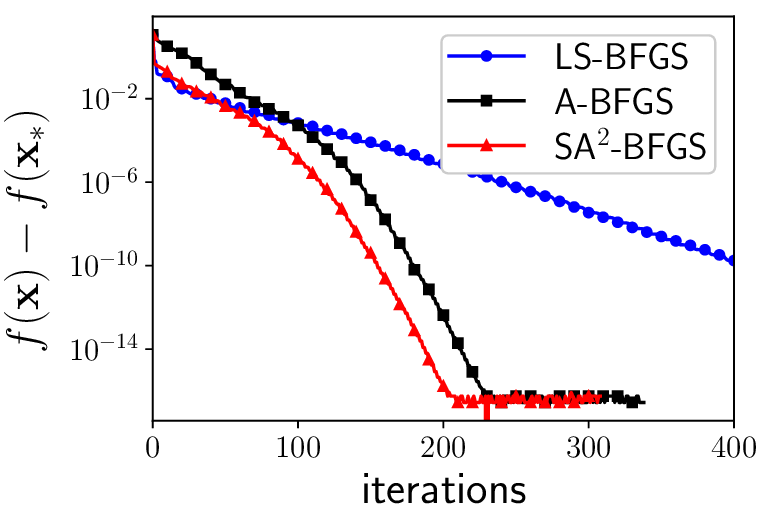}  & 
        \includegraphics[scale=0.38]{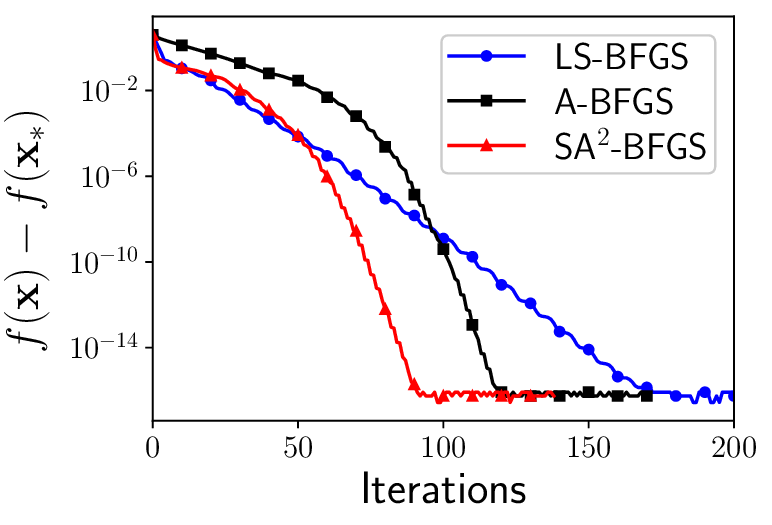} \\
        \small (d) mushrooms with $\mB_0=L_{\rm lr}\mI$  & 
        \small (e) w8a with $\mB_0=L_{\rm lr}\mI$ & 
        \small (f) real-sim with $\mB_0=L_{\rm lr}\mI$ \\
        \end{tabular}
        \caption{We compare our {\SABFGS} with LS-BFGS and A-BFGS in terms of function value gap $f(\vx) - f(\vx_{*})$ versus the number of iterations with the initializations $\mB_{0} = \mu_{\mathrm{lr}} \mI$ and $\mB_{0} = L_{\mathrm{lr}} \mI$.}
        \label{fig:logistic_iter}
\end{figure}

\begin{figure}[t]
    \centering
    \begin{tabular}{ccc}
        \includegraphics[scale=0.38]{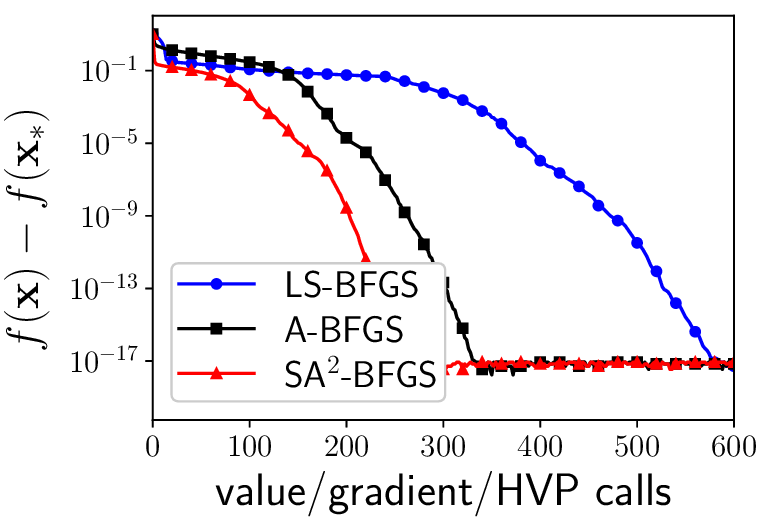}  & 
        \includegraphics[scale=0.38]{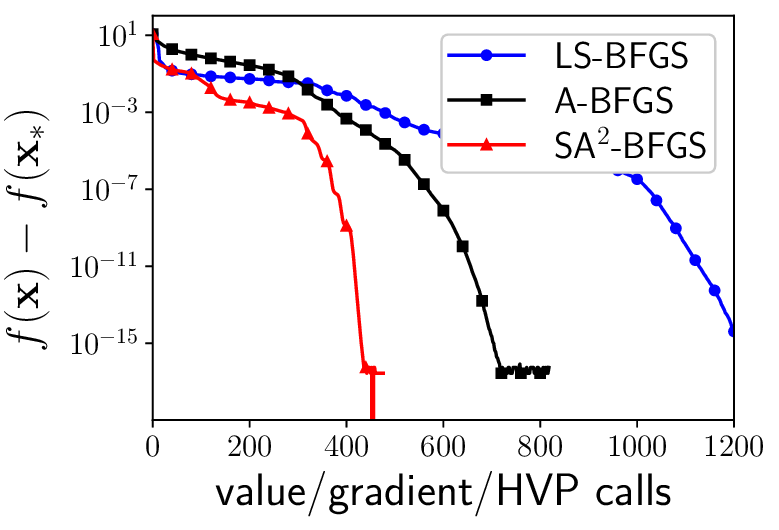}  & 
        \includegraphics[scale=0.38]{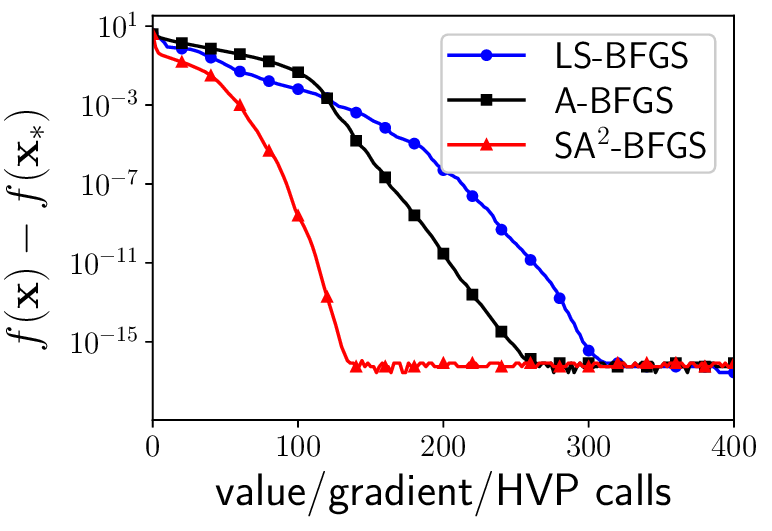} \\
        \small (a) mushrooms with $\mB_0=\mu_{\rm lr}\mI$  & 
        \small (b) w8a with $\mB_0=\mu_{\rm lr}\mI$ & 
        \small (c) real-sim with $\mB_0=\mu_{\rm lr}\mI$ \\[0.25cm]
        \includegraphics[scale=0.38]{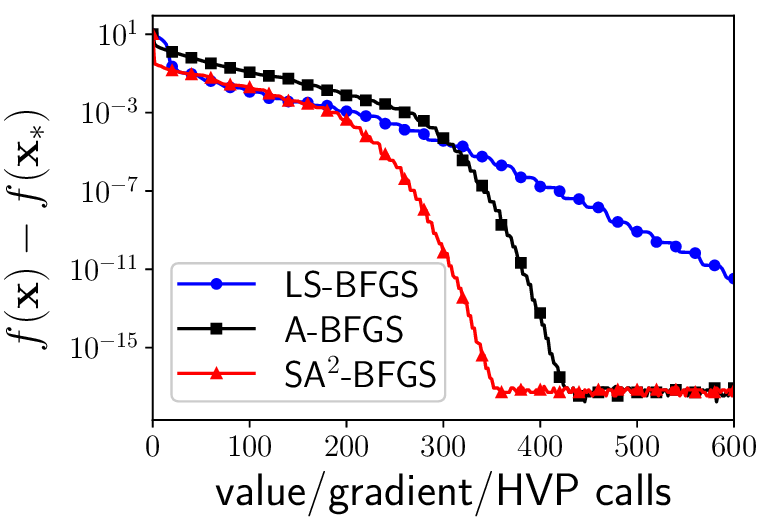}  & 
        \includegraphics[scale=0.38]{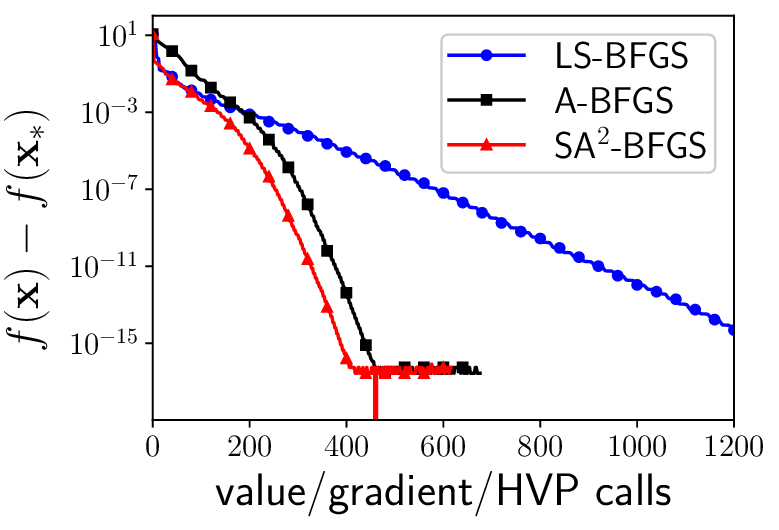}  & 
        \includegraphics[scale=0.38]{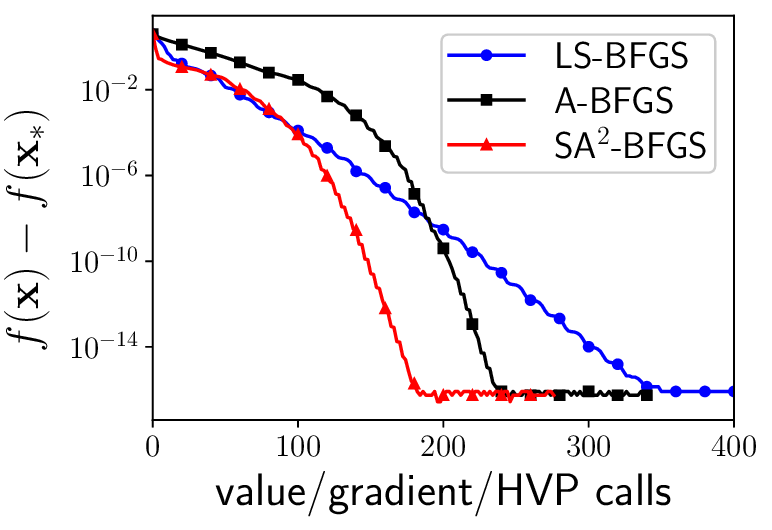} \\
        \small (d) mushrooms with $\mB_0=L_{\rm lr}\mI$  & 
        \small (e) w8a with $\mB_0=L_{\rm lr}\mI$ & 
        \small (f) real-sim with $\mB_0=L_{\rm lr}\mI$ \\
        \end{tabular}
        \caption{We compare our {\SABFGS} with LS-BFGS and A-BFGS in terms of function value gap $f(\vx) - f(\vx_{*})$ versus the
        the number of evaluations on function value, gradient, and Hessian–vector product (HVP) with the initializations $\mB_{0} = \mu_{\mathrm{lr}} \mI$ and $\mB_{0} = L_{\mathrm{lr}} \mI$.}
        \label{fig:logistic_eval}
\end{figure}

\section{Conclusion}

In this work, we revisit the adaptive BFGS method~\cite{gaoQuasiNewtonMethodsSuperlinear2019} by establishing its explicit global linear convergence and superlinear local convergence.
Our result helps to better understand the behavior of BFGS methods without line search, as the previous analysis only provides asymptotic convergence.
We further propose a smoothness aided adaptive BFGS method, enjoying the better convergence behavior.
In future work, we are interested in extending our idea to address the objective that only satisfies the strictly convex assumption~\cite{jinAffineinvariantGlobalNonasymptotic2025} and studying the other quasi-Newton methods such as DFP and SR1.

\appendix

\section{Proofs in Section~\ref{sec:pre}}\label{proof:ada_bfgs_properties}

We present the proof of Lemma~\ref{lem:ada_bfgs_properties}, which extends Theorems 4.1, 6.2 and 6.3 of \citet{gaoQuasiNewtonMethodsSuperlinear2019} to the general self-concordant parameter $M$. 

\begin{proof}
    According to the definition of $\eta_{k}$ in \eqref{def:eta}, the adaptive step size $t_{k}$ can be written as
    \begin{equation}\label{eq:appendix-t-k}
        t_{k} = \frac{\eta_{k}}{\lNorm{\vd_{k}}{\vx_{k}}(1 + M \eta_{k})}.
    \end{equation}
    We substitute \eqref{eq:appendix-t-k} into \eqref{eq:self_con_f_upper} to prove \eqref{eq:self_con_f_dec} as follows:
    \begin{equation*}
    \begin{aligned}
        f(\vx_{k + 1}) & \leq f(\vx_{k}) + \frac{\vg_{k}^{\T} \vd_{k}}{M \lNorm{\vd_{k}}{\vx_{k}}} \frac{M \eta_{k}}{1 + M \eta_{k}} + \frac{1}{M^{2}} \left(- \frac{M \eta_{k}}{1 + M \eta_{k}} - \ln \left( 1 - \frac{M \eta_{k}}{1 + M \eta_{k}} \right) \right) \\
        & = f(\vx_{k}) - \frac{1}{M^{2}} \frac{(M \eta_{k})^{2}}{1 + M \eta_{k}} - \frac{1}{M^{2}} \frac{M \eta_{k}}{1 + M \eta_{k}} + \frac{1}{M^{2}} \ln (1 + M \eta_{k}) \\
        & = f(\vx_{k}) - \frac{1}{M^{2}} \left( 1 + M \eta_{k} - \ln (1 + M \eta_{k}) \right) \\
        & = f(\vx_{k}) - \frac{\omega(M \eta_{k})}{M^{2}}.
    \end{aligned}
    \end{equation*}
    According to \eqref{eq:self_con_f_dec}, proving \eqref{eq:armijo_condition} suffices to show that
    \begin{equation}\label{eq:appendix-omega}
        \omega(M \eta_{k}) \geq \frac{1}{2} M^{2} t_{k} (- \vg_{k}^{\T} \vd_{k}).
    \end{equation}
    The right-hand side of \eqref{eq:appendix-omega} holds that
    \begin{equation*}
        M^{2} t_{k} (- \vg_{k}^{\T} \vd_{k}) = M^{2} \frac{1}{\lNorm{\vd_{k}}{\vx_{k}}} \frac{\eta_{k}}{1 + M \eta_{k}} (- \vg_{k}^{\T} \vd_{k}) = \frac{(M \eta_{k})^{2}}{1 + M \eta_{k}},
    \end{equation*}
    where we again use the definitions of $t_k$ and $\eta_{k}$ in \cref{eq:appendix-t-k,def:eta}, respectively. 
    Thus, it remains to verify that 
    \begin{equation*}
        \omega(M \eta_{k}) \geq \frac{(M \eta_{k})^{2}}{2(1 + M \eta_{k})}.
    \end{equation*}
    From Lemma~5.1.5 in~\cite{nesterovLecturesConvexOptimization2018}, we know that
    \begin{equation*}
        \omega(z) \geq \frac{z^{2}}{2 (1 + z)}
    \end{equation*}
    for all $z \in [0, +\infty)$. Substituting $z = M \eta_{k} > 0$ leads the desired result.
    
    Applying \eqref{eq:self_con_grad_upper} with $\vx=\vx_k$, $\vd=\vd_k$, and $t=t_k$, we have
    \begin{equation}\label{eq:min-t-k}
        \vg_{k + 1}^{\T} \vd_{k} \leq \vg_{k}^{\T} \vd_{k} + \frac{\lNorm{\vd_{k}}{\vx_{k}} \eta_{k} / (1 + M \eta_{k})} {1 - M \eta_{k}/ (1 + M \eta_{k})} = \vg_{k}^{\T} \vd_{k} + \eta_{k} \lNorm{\vd_{k}}{\vx_{k}} = 0.
    \end{equation}
    On the other hand, applying \eqref{eq:self_con_grad_lower} with $\vx=\vx_k$, $\vd=\vd_k$, and $t=t_k$, we have
    \begin{equation*}
        \vg_{k + 1}^\T \vd_{k} \geq \vg_{k}^{\T} \vd_{k} + \frac{\lNorm{\vd_{k}}{\vx_{k}} \eta_{k}/(1 + M \eta_{k})}{1 + M \eta_{k}/(1 + M \eta_{k})} = \vg_{k}^{\T} \vd_{k} - \frac{1}{1 + 2 M \eta_{k}}  \vg_{k}^{\T} \vd_{k} = \frac{2 M \eta_{k}}{1 + 2 M \eta_{k}} \vg_{k}^{\T} \vd_{k}.
    \end{equation*}
    Together with above two inequalities, we establish \eqref{eq:curvature_condition}.
\end{proof}

\begin{remark}
    Recall that the right-hand side of \eqref{eq:self_con_grad_upper} is the derivative of the right-hand side of \eqref{eq:self_con_f_upper} with respect to $t$, and we have shown that it is zero by taking $t=t_k$ in \eqref{eq:min-t-k}. 
    Therefore, taking $t=t_{k}$ minimizes the right-hand side of \eqref{eq:self_con_f_upper} with respect to $t$.
\end{remark}

\section{Proofs in Section~\ref{sec:linear_phase}}

This section provides the missing proofs for the global linear convergence of adaptive BFGS~(Algorithm~\ref{alg:ada_bfgs}).

\subsection{Proof of Lemma~\ref{lem:ineq_arith_geo}}\label{proof:ineq_arith_geo}

\begin{proof}
    We prove this lemma by applying the inequality of arithmetic and geometric means twice, that is,
    \begin{equation*}
        \prod_{i = 0}^{k - 1} (1 - a u_{i}) \leq \left( \frac{\sum_{i = 0}^{k - 1} (1 - a u_{i})}{k} \right)^{k} = \left( 1 - \frac{a\sum_{i = 0}^{k - 1} u_{i}}{k} \right)^{k} \leq \left( 1 - a \left( \prod_{i = 0}^{k - 1} u_{i} \right)^{\frac{1}{k}} \right)^{k}.
    \end{equation*}
\end{proof}

\subsection{Proof of Lemma~\ref{lem:omega_bounds}}\label{proof:omega_bounds}

\begin{proof}[Proof of Lemma~\ref{lem:omega_bounds}]
From Lemma~5.1.5 in~\cite{nesterovLecturesConvexOptimization2018}, it holds that
\begin{equation*}
    \omega(z) \geq \frac{z^{2}}{2 (1 + z)}
\end{equation*}
for all $z \in [0, +\infty)$. 
For $z \in [0, 1)$, above inequality implies
\begin{equation*}
    \omega(z) \geq \eval{\left(\frac{1}{2 (1 + z)}\right)}{z = 1} z^{2} = \frac{z^{2}}{4}.
\end{equation*}
For $z \in [1, +\infty)$, we have 
\begin{equation*}
    \omega(z) \geq \frac{z}{2 (1 + z)} z \geq \eval{\left(\frac{z}{2 (1 + z)}\right)}{z = 1} z = \frac{z}{4},
\end{equation*}
where we use the fact that $z/ (1 + z)$ is increasing in $z$. Combining the above two inequalities yields~\eqref{eq:omega_lower}. 

The proof of~\eqref{eq:omega_inv_upper} can be found in Lemma~A.1 of~\citet{rodomanovGlobalComplexityAnalysis2024}.
\end{proof}

\subsection{Proof of Lemma~\ref{lem:dec_sqrt}}\label{proof:dec_sqrt}

\begin{proof}
    Since the sequence $\{y_{k}\}_{k \geq 0}$ is strictly decreasing, we immediately get
    \begin{equation*}
        \sqrt{y_{k + 1}} - \sqrt{y_{k}} = \frac{y_{k + 1} - y_{k}}{\sqrt{y_{k + 1}} + \sqrt{y_{k}}} \leq \frac{- \gamma_{k} \sqrt{y_{k}}}{2 \sqrt{y_{k}}} = - \frac{\gamma_{k}}{2}.
    \end{equation*}
    This completes the proof.
\end{proof}

\subsection{Proof of Lemma~\ref{lem:dec_aux_func}}\label{proof:dec_aux_func}

\begin{proof}
    we take the derivative on $\ln \xi(u)$ to obtain
    \begin{equation*}
        \diff{\ln \xi(u)}{u} = \ln (1 - a b^{\frac{1}{u}}) - \frac{a b^{\frac{1}{u}}}{u (1 - a b^{\frac{1}{u}})} \ln \frac{1}{b} < 0
    \end{equation*}
    for all $u \in (0, +\infty)$, which means $\xi(u)$ is a decreasing function.
\end{proof}

\subsection{Proof of Lemma~\ref{lem:aux_func_2}}\label{proof:aux_func_2}

\begin{proof}
    Our proof consists of two parts: the first shows that the function $\zeta(u)$ is well-defined, and the second provides the upper bounds for $\zeta(u)$.
    
    \paragraph{The Well-Definedness of $\zeta(u)$}
    We introduce the auxiliary function 
    \begin{equation}\label{eq:aux_func_2_aux1_def}
        \phi_{1}(u) \defeq a \Biggl(b \Bigl(\max\Bigl\{\frac{u}{s}, 1\Bigr\}\Bigr)^{2 u} \biggl(\frac{1}{\varkappa}\biggr)^{u}\Biggr)^{\frac{1}{k - u}},
    \end{equation}
    where $u \in [0, \bar{u})$. It suffices to prove that $\phi_{1}(u) < 1$ for all $u \in [0, \bar{u})$.

    In the case of $\bar{u} \leq s$, we have $0 \leq u < s$ for all $u \in [0, \bar{u})$.
    Consequently, we have 
    \begin{equation}\label{eq:phi1-1}
        \phi_{1}(u) = a \left( \frac{b}{\varkappa^{u}} \right)^{\frac{1}{k - u}} < 1.
    \end{equation}
    In the case of $\bar{u} > s$, we only need to consider the case $u \in (s, \bar{u})$. 
    Otherwise, when $u \in [0, s]$, we can also derive \eqref{eq:phi1-1} by the above derivation. 
    Therefore, we focus on analyzing the behavior of $\phi_1(u)$ on $u \in (s, \bar{u})$ in the remainder. 
    We take the derivation of $\ln\phi_{1}(u)$ with respect to $u$ as
    \begin{equation}\label{eq:aux_func_2_aux1_diff}
    \begin{aligned}
        \diff{\ln \phi_{1}(u)}{u} & = \diff{\Biggl(\ln a + \frac{1}{k - u} \biggl(\ln b + 2 u \ln \frac{u}{\sqrt{\varkappa} s}\biggr)\Biggr)}{u} \\
        & = \frac{1}{(k - u)^2} \underbrace{\biggl(\ln b + 2 k \ln \frac{u}{\sqrt{\varkappa} s} + 2 (k - u) \biggr) }_{\phi_{1}^{1}(u)}.
    \end{aligned}
    \end{equation}
    Observe that
    \begin{equation*}
        \diff{\phi_{1}^{1}(u)}{u} = 
        \diff{\left( \ln b + 2 k \ln \frac{u}{\sqrt{\varkappa} s} + 2 (k - u) \right)}{u} = 2 \left( \frac{k}{u} - 1\right) > 0,
    \end{equation*}
    where we use the definition of $\bar{u}$ in \eqref{def:bar_u} to directly conclude that $k \geq \bar{u} > u$. 
    Therefore, the function $\phi_{1}^{1}(u)$ is strictly increasing on $(s, \bar{u})$. 
    Moreover, we know that $\diff{\ln \phi_{1}(u)}{u}$ has the same sign as $\phi_{1}^{1}(u)$ by \eqref{eq:aux_func_2_aux1_diff}. 
    Therefore, the function $\diff{\ln \phi_{1}(u)}{u}$ can change the sign at most once on $(s,\bar u)$.
    This leads to the following three possibilities:
    \begin{enumerate}
        \item $\diff{\ln \phi_{1}(u)}{u} > 0$ for all $u \in (s, \bar{u})$.
        \item $\diff{\ln \phi_{1}(u)}{u} < 0$ for all $u \in (s, \bar{u})$.
        \item $\diff{\ln \phi_{1}(u)}{u}$ changes sign on $(s, \bar{u})$, that is, there exists $u_{1, *} \in (s, \bar{u})$ such that $\diff{\ln \phi_{1}(u)}{u} < 0$ for $u \in (s, u_{1, *})$ and $\diff{\ln \phi_{1}(u)}{u} > 0$ for $u \in (u_{1, *}, \bar{u})$.
    \end{enumerate}
    Considering the above three possibilities, we have
    \begin{equation*}
        \phi_{1}(u) < \max\Bigl\{ \phi_{1}(s), \lim_{v \to \bar{u}- 0} \phi_{1}(v)\Bigr\}
    \end{equation*}
    for all $u \in (s, \bar{u})$.
    Since we have $\phi_{1}(s) < 1$ and the definition of $\bar{u}$ in \eqref{def:bar_u} implies $\lim_{v \to \bar{u} - 0} \phi_{1}(v) \leq 1$ , it follows that~$\phi_{1}(u) < 1$.

    Thus, we conclude that $\phi_{1}(u) < 1$ for all $u \in [0, \bar{u})$.

    \paragraph{The upper bounds for $\zeta(u)$}
    We now prove \eqref{eq:aux_func_2_upper_1} for $k \geq 2 \ln (1 / b) + 2 \sqrt{\varkappa} s$.
    Based on the definition of $\phi_1(u)$ in \eqref{eq:aux_func_2_aux1_def}, we have 
    \begin{equation}\label{eq:phi_1_u_s}
        \phi_{1}(\sqrt{\varkappa} s) = a b^{\frac{1}{k - \sqrt{\varkappa} s}} < 1.
    \end{equation}
    Furthermore, the definition of $\bar{u}$ in \eqref{def:bar_u} implies that $\bar u$ is the supremum of $\{u\in(0, k):\phi_1(\sqrt{\varkappa} s) < 1\}$.
    Therefore, \eqref{eq:phi_1_u_s} implies $\bar{u} > \sqrt{\varkappa} s$.

    Next, for any $u_{1}, u_{2}$ satisfying $0 \leq u_{1} < u_{2} \leq s$, we have
    \begin{equation*}
    \begin{aligned}
        \zeta(u_{1}) & = \left( 1 - a \left( \frac{b}{\varkappa^{u_{1}}}\right)^{\frac{1}{k - u_{1}}} \right)^{k - u_{1}} \\
        & \leq \left( 1 - a \left( \frac{b}{\varkappa^{u_{2}}}\right)^{\frac{1}{k - u_{1}}} \right)^{k - u_{1}} \\
        & \leq \left( 1 - a \left( \frac{b}{\varkappa^{u_{2}}}\right)^{\frac{1}{k - u_{2}}} \right)^{k - u_{2}} \\
        &= \zeta(u_{2}),
    \end{aligned}
    \end{equation*}
    where the second inequality follows from Lemma~\ref{lem:dec_aux_func}. 
    This implies that $\zeta(u)$ is monotonically increasing on the interval $[0, s]$. 
    Therefore, the upper bound of $\zeta(u)$ cannot be attained when $u\in[0,s)$, that is,
    \begin{equation}\label{eq:aux_func_2_restrict}
        \sup_{u \in [0, \bar{u})} \zeta(u) = \sup_{u \in [s, \bar{u})} \zeta(u).
    \end{equation}
    This means we only need to consider the case of $u\in [s, \bar{u})$. 
    
    For $u \in [s, \bar{u})$, taking the logarithm of $\zeta(u)$ achieves
    \begin{equation}\label{eq:aux_func_2_ineq_1}
    \begin{aligned}
        \ln \zeta(u) \overset{\eqref{eq:aux_func_2_aux1_def}}{=} (k - u) \ln \left( 1 - \phi_{1}(u) \right) 
         \leq - (k - u) \phi_{1}(u),
    \end{aligned}
    \end{equation}
    where the last step is based on the inequality $\ln (1 - z) \leq - z$ for $z \in [0, 1)$.

    We define the auxiliary function 
    \begin{equation}\label{eq:aux_func_2_aux2_def}
    \begin{aligned}
        \phi_{2}(u) & \defeq \ln (k - u) + \ln \phi_{1}(u) \\
        & = \ln (k - u) + \frac{1}{k - u} \left( \ln b + 2 u \ln \frac{u}{\sqrt{\varkappa} s} \right) + \ln a 
    \end{aligned}
    \end{equation}
    for $u\in[s, \bar{u})$.
    Then \eqref{eq:aux_func_2_ineq_1} can be written as
    \begin{equation}\label{eq:aux_func_2_ineq_2}
        \ln \zeta(u) \leq - \exp \left( \phi_{2}(u) \right).
    \end{equation}
    Taking the derivative on $\phi_{2}(u)$, we have
    \begin{equation}\label{eq:aux_func_2_aux2_diff}
        \diff{\phi_{2}(u)}{u} = \frac{1}{(k - u)^{2}} \underbrace{\left( k - u + \ln b + 2 k \ln \frac{u}{\sqrt{\varkappa} s} \right)}_{\phi_{2}^{1}(u)}.
    \end{equation}
    The condition $k > 2 \ln (1 / b) + 2 \sqrt{\varkappa}s$ implies
    \begin{equation*}
        \phi_{2}^{1}(\sqrt{\varkappa} s) = k - \sqrt{\varkappa} s + \ln b > 0.
    \end{equation*}
    The definition of $\bar{u}$ in \eqref{def:bar_u} indicates $k \geq \bar{u} > u$, which leads to 
    \begin{equation*}
        \diff{\phi_{2}^{1}(u)}{u} = \frac{2k}{u} - 1 > 0.
    \end{equation*}
    Hence, we have $\phi_{2}^{1}(u) > 0$ holds on $u\in[\sqrt{\varkappa} s, \bar{u})$. 
    Consequently, we have
    \begin{equation*}
        \diff{\phi_{2}(u)}{u} = \frac{\phi_{2}^{1}(u)}{(k - u)^{2}} > 0 
    \end{equation*}
    for all $u \in [\sqrt{\varkappa} s, \bar{u})$.
    It follows that the minimum value of $\phi_{2}(u)$ over $[s, \bar{u})$ is attained in $[s, \sqrt{\varkappa} s)$. We let~$u_{2, *} \in [s, \sqrt{\varkappa} s)$ be the minimizer of $\phi_2(u)$.
    Then we have
    \begin{equation}\label{eq:aux_func_2_aux2_restrict}
        \inf_{u \in [s, \bar{u})} \phi_{2}(u) = \inf_{u \in [s, \sqrt{\varkappa} s)} \phi_{2}(u) = \phi_{2}(u_{2, *}),
    \end{equation}
    which implies 
    \begin{equation}\label{eq:aux_func_2_aux2_pos}
        \evaldiff{\phi_{2}(u)}{u}{u_{2, *}} \geq 0.
    \end{equation}
    If \eqref{eq:aux_func_2_aux2_pos} does not hold, i.e., $\evaldiff{\phi_{2}(u)}{u}{u_{2, *}} < 0$, then there would exist a sufficiently small $\delta > 0$ such that $u_{2, *} + \delta \in [s, \sqrt{\varkappa}s)$ and $\phi_{2} (u_{2, *} + \delta) < \phi_{2}(u_{2, *})$, contradicting the assumption that $u_{2, *}$ is a minimizer.

    Combining \cref{eq:aux_func_2_aux2_diff,eq:aux_func_2_aux2_pos}, we have
     \begin{equation*}
        \frac{1}{(k - u_{2,*})^{2}} \left( k - u_{2,*} + \ln b + 2 k \ln \frac{u_{2, *}}{\sqrt{\varkappa} s} \right) \geq 0.
    \end{equation*}
    Rearranging above result leads to
    \begin{equation}\label{eq:aux_func_2_aux2_pos_res}
        \ln \frac{u_{2, *}}{\sqrt{\varkappa} s} \geq - \frac{k - u_{2, *}}{2 k} + \frac{1}{2 k} \ln \frac{1}{b}.
    \end{equation}
    Based on the definition of $\phi_{2}(u)$ in \eqref{eq:aux_func_2_aux2_def}, we have
    \begin{equation}\label{eq:aux_func_2_aux2_lower_bound}
    \begin{aligned}
        \phi_{2}(u_{2, *}) & ~=~ \ln (k - u_{2, *}) + \frac{1}{k - u_{2, *}} \left( \ln b + 2 u_{2, *} \ln \frac{u_{2, *}}{\sqrt{\varkappa} s} \right) + \ln a \\
        & ~\overset{\mathclap{\eqref{eq:aux_func_2_aux2_pos_res}}}{\geq}~ \ln (k - u_{2, *}) + \frac{1}{k - u_{2, *}} \left( \ln b - \frac{u_{2, *}(k - u_{2, *})}{k} + \frac{u_{2, *}}{k} \ln \frac{1}{b} \right) + \ln a \\
        & ~=~ \ln (k - u_{2, *}) - \frac{1}{k - u_{2, *}} \left(   \frac{u_{2, *}(k - u_{2, *}) }{k} + \frac{k - u_{2, *}}{k} \ln \frac{1}{b}\right) + \ln a\\
        & ~=~ \ln (k - u_{2, *}) - \frac{u_{2, *} + \ln (1 / b)}{k} + \ln a \\
        & ~>~ \ln (k - \sqrt{\varkappa} s) - \frac{\sqrt{\varkappa} s + \ln (1 / b)}{k} + \ln a \\
        & ~\geq~ \ln (k - \sqrt{\varkappa} s) - \frac{1}{2} + \ln a,
    \end{aligned}
    \end{equation}
    where the second inequality uses the fact that $u_{2, *} \in [s, \sqrt{\varkappa} s)$, and the last inequality is based on the condition $k \geq 2 \ln (1 / b) + 2 \sqrt{\varkappa} s$.
    
    Combining \cref{eq:aux_func_2_restrict,eq:aux_func_2_ineq_2,eq:aux_func_2_aux2_restrict,eq:aux_func_2_aux2_lower_bound} leads to
    \begin{equation}\label{eq:aux_func_2_upper_0}
    \begin{aligned}
        \sup_{u \in [0, \bar{u})} \zeta(u) & ~\overset{\mathclap{\eqref{eq:aux_func_2_restrict}}}{=}~ \sup_{u \in [s, \bar{u})} \zeta(u) \\
        & ~\overset{\mathclap{\eqref{eq:aux_func_2_ineq_2}}}{\leq}~ \sup_{u \in [s, \bar{u})} \exp \left( - \exp \left( \phi_{2}(u) \right) \right) \\
        & ~=~ \exp \biggl(- \exp \biggl(\inf_{u \in [s, \bar{u})} \phi_{2}(u) \biggr)\biggr) \\
        & ~\overset{\mathclap{\eqref{eq:aux_func_2_aux2_restrict}}}{=}~ \exp \left( - \exp \left( \phi_{2}(u_{2, *}) \right) \right) \\
        & ~\overset{\mathclap{\eqref{eq:aux_func_2_aux2_lower_bound}}}{<}~ \exp \biggl(- a \exp \biggl(- \frac{1}{2}\biggr) \bigl(k - \sqrt{\varkappa} s\bigr)\biggr).
    \end{aligned}
    \end{equation}
    Based on the fact that $\exp \left( - 1 / 2 \right) > 1 / 2$, we immediately obtain \eqref{eq:aux_func_2_upper_1}.

    Moreover, when $a \leq 1 / 3$, \eqref{eq:aux_func_2_upper_0} indicates that proving \eqref{eq:aux_func_2_upper_2} only requires showing that 
    \begin{equation*}
        \exp \left( - a \exp \left( - \frac{1}{2} \right) \right) \leq 1 - \frac{a}{2},
    \end{equation*}
    which is equivalent to
    \begin{equation*}
        -a \exp \left( - \frac{1}{2} \right)  \leq \ln \left( 1 - \frac{a}{2} \right).
    \end{equation*}
    It suffices to prove that the inequality 
    \begin{equation}\label{eq:zexpln}
        z \exp \left( - \frac{1}{2} \right) + \ln \left(1 - \frac{z}{2} \right) \geq 0
    \end{equation}
    holds for all $z \in [0, 1 / 3]$, which is easy to verify since the left-hand side of \eqref{eq:zexpln} equals to zero at $z = 0$, and its derivative satisfies
    \begin{align*}
         & \diff{\left( z \exp \left( - \frac{1}{2} \right) + \ln \left(1 - \frac{z}{2}\right) \right)}{z} \\
         = & \exp \left( - \frac{1}{2} \right) - \frac{1}{2 - z}  \\
         \geq & \exp \left( - \frac{1}{2} \right) - \frac{3}{5} > 0
    \end{align*}
    for all $z \in [0, 1 / 3]$. 
    Thus, we have proved \eqref{eq:aux_func_2_upper_2}.
\end{proof}

\section{Proofs in Section~\ref{sec:improved}}

This section provides the missing proofs for the local superlinear convergence of adaptive BFGS~(Algorithm~\ref{alg:ada_bfgs}).

\subsection{Proof of Lemma~\ref{lem:self_con_bound_improved}}\label{proof:self_con_bound_improved}

\begin{proof}
    We define 
    \begin{align*}
        \varphi(t) = \vd^{\T} \nabla^{2} f(\vx + t\vd) \vd.    
    \end{align*}
    According to Assumptions~\ref{asm:strongly_convex} and~\ref{asm:grad_Lip}, we have
    \begin{equation}\label{eq:varphi_mu_L}
        \mu \Norm{\vd}^{2} \leq \varphi(t) \leq L \Norm{\vd}^{2}.
    \end{equation}
    Moreover, Assumption~\ref{asm:self_con} implies the derivative of $\varphi(t)$ satisfies
    \begin{equation*}
        \abs*{\diff{\varphi(t)}{t}} = \abs*{\nabla^{3} f(\vx + t\vd) \left[\vd, \vd, \vd \right]} \leq 2 M \left( \nabla^{2} f(\vx + t \vd) \left( \vd, \vd \right) \right)^{\frac{3}{2}} = 2 M  \varphi(t)^{\frac{3}{2}},
    \end{equation*}
    which implies 
    \begin{equation*}
        \abs*{\diff{\left( (\varphi(t))^{- \frac{1}{2}} \right)}{t}} = \frac{1}{2} \abs*{(\varphi(t))^{- \frac{3}{2}} \diff{\varphi(t)}{t}} \leq M.
    \end{equation*}
    It follows that
    \begin{equation*}
         -M \leq \diff{\left( \left(\varphi(t)\right)^{- \frac{1}{2}} \right)}{t} \leq M.
    \end{equation*} 
    Taking the integral on above inequality over $[0,t]$, we obtain 
    \begin{equation}\label{eq:varphi_inv_bound}
        (\varphi(0))^{- \frac{1}{2}} - M t \leq \varphi(t)^{- \frac{1}{2}} \leq \varphi(0)^{- \frac{1}{2}} + M t.
    \end{equation}
    Since we have $\varphi(0) = \lNorm{\vd}{\vx}^{2}$, rearranging the upper bound in \eqref{eq:varphi_inv_bound} leads to 
    \begin{equation}\label{eq:varphi_self_con_below}
        \varphi(t) \geq \frac{\lNorm{\vd}{\vx}^{2}}{(1 + M \lNorm{\vd}{\vx} t)^{2}}
    \end{equation}
    for all $t \geq 0$.
    Combining results of \cref{eq:varphi_mu_L,eq:varphi_self_con_below}, we have
    \begin{equation}\label{eq:_varphi_below}
        \varphi(t) \geq \max\left\{ \frac{\lNorm{\vd}{\vx}^{2}}{(1 + M \lNorm{\vd}{\vx} t)^{2}}, \mu \Norm{\vd}^{2} \right\}. 
    \end{equation}
    Note that in the case of $t\in[0,t^l]$, the lower bound of $\phi(t)$ shown in \eqref{eq:varphi_self_con_below} is tighter than the one in  \eqref{eq:varphi_mu_L}, which is because $\lNorm{\vd}{\vx}^{2} / (1 + M \lNorm{\vd}{\vx} t)^{2}$ is decreasing with respect to $t$ and $\lNorm{\vd}{\vx}^{2}/(1 + M \lNorm{\vd}{\vx} t^l)^{2} = \mu \Norm{\vd}^{2}$. 
    Therefore, \eqref{eq:_varphi_below} directly implies
    \begin{equation}\label{eq:varphi_below}
        \varphi(t) \geq \begin{cases}
            \dfrac{\lNorm{\vd}{\vx}^{2}}{(1 + M \lNorm{\vd}{\vx} t)^{2}}, & 0 \leq t \leq t^{l}, \\[0.2cm]
            \mu \Norm{\vd}^{2}, & t > t^{l}.
        \end{cases}
    \end{equation}
    Similarly, the lower bound in \eqref{eq:varphi_inv_bound} yields
    \begin{equation}\label{eq:varphi_self_con_upper}
        \varphi(t) \leq \frac{\lNorm{\vd}{\vx}^{2}}{(1 - M \lNorm{\vd}{\vx} t)^{2}}
    \end{equation}
    for all $t\in(0,1 / (M \lNorm{\vd}{\vx})$. 
    We then apply \eqref{eq:varphi_mu_L}, \eqref{eq:varphi_self_con_upper}, and the definition of $t^{u}$ to achieve
    \begin{equation}\label{eq:varphi_upper}
        \varphi(t) \leq \begin{cases}
            \dfrac{\lNorm{\vd}{\vx}^{2}}{(1 - M \lNorm{\vd}{\vx} t)^{2}}, & 0 \leq t \leq t^{u}, \\[0.2cm]
            L \Norm{\vd}^{2}, & t > t^{u}.
        \end{cases}
    \end{equation}
    Integrating \cref{eq:varphi_below,eq:varphi_upper}, we obtain \cref{eq:self_con_grad_lower_improved,eq:self_con_grad_upper_improved}, respectively. 
    Furthermore, we integrate \eqref{eq:self_con_grad_upper_improved} to yields \eqref{eq:self_con_f_upper_improved}.
\end{proof}

\subsection{Derivation of the Smoothness Aided Adaptive Size}\label{proof:ada_stepsize_improved}

In Section~\ref{sec:dsass}, we have claimed that our smoothness aided adaptive step size 
\begin{equation}\label{eq:ttk-appendix}
    \ttt_{k} = \begin{cases}
        \dfrac{\eta_{k}}{(1 + M \eta_{k}) \lNorm{\vd_{k}}{\vx_{k}}}, & (1 + M \eta_{k}) \alpha_{k} \leq 1, \\[0.3cm]
        \dfrac{M \eta_{k} \alpha_{k}^{2} + (1 - \alpha_{k})^{2}}{M \lNorm{\vd_{k}}{\vx_{k}}}, & (1 + M \eta_{k}) \alpha_{k} > 1,
    \end{cases}
\end{equation}
minimizes the right-hand side of \eqref{eq:self_con_f_upper_improved}, where $\eta_{k}$ and $\alpha_{k}$ follow the definitions in \cref{def:eta,def:alpha}, i.e.,
\begin{equation}\label{def:eta-appendix}
    \eta_{k} \defeq - \frac{\vg_{k}^{\T} \vd_{k}}{\lNorm{\vd_{k}}{\vx_{k}}},
\end{equation}
and 
\begin{equation}\label{def:alpha-appendix}
    \alpha_{k} \defeq \frac{\lNorm{\vd_{k}}{\vx_{k}}}{\sqrt{L} \Norm{\vd_{k}}}.
\end{equation}

Now, we formally present and prove this result.

\begin{prop}
We define
\begin{equation}\label{eq:tHkt}
\tH_k(t) \defeq \begin{cases}
    f(\vx_{k}) + t \vg_{k}^{\T} \vd_{k} + \dfrac{\omega_{*}(M t \lNorm{\vd_{k}}{\vx_{k}})}{M^{2}} , & 0 \leq t \leq t_{k}^{u}, \\[0.3cm]
    f(\vx_{k}) + t \vg_{k}^{\T} \vd_{k} + \dfrac{\omega_{*}(M t_{k}^{u} \lNorm{\vd_{k}}{\vx_{k}})}{M^{2}}  + \dfrac{t_{k}^{u} (t - t_{k}^{u}) \lNorm{\vd_{k}}{\vx_{k}}^{2} }{1 - M t_{k}^{u} \lNorm{\vd_{k}}{\vx_{k}}} + \dfrac{1}{2} L(t - t_{k}^{u})^{2} \Norm{\vd_{k}}^{2}, & t > t_{k}^{u},
\end{cases}
\end{equation}
where $t_{k}^{u}$ follows the definition in \eqref{eq:def_t_l_u}, i.e.,
\begin{equation}\label{def:_t_u-appendix}
    t_{k}^{u} \defeq \frac{1}{M \lNorm{\vd_{k}}{\vx_{k}}} \left( 1 - \frac{\lNorm{\vd_{k}}{\vx_{k}}}{\sqrt{L}\Norm{\vd_{k}}} \right).
\end{equation}
Then it holds
\begin{align}\label{eq:optimal-ttk}
    \ttt_k = \argmin_{t\geq 0} \tH_k(t),
\end{align}
where $\ttt_k$ follows the definition in \eqref{eq:ttk-appendix}.
\end{prop}
\begin{proof}   
Based on the definition of $\tH_k(t)$ in \eqref{eq:tHkt}, we can verify its derivative holds 
\begin{equation}\label{eq:tHkt_first_order}
     \tH'_k(t) = \begin{cases}
        \vg_{k}^{\T} \vd_{k} + \dfrac{\lNorm{\vd_{k}}{\vx_{k}}^{2} t}{1 - M \lNorm{\vd_{k}}{\vx_{k}} t}, & 0 \leq t \leq t_{k}^{u}, \\[0.25cm]
        \vg_{k}^{\T} \vd_{k} + \dfrac{\lNorm{\vd_{k}}{\vx_{k}}^{2} t_{k}^{u}}{1 - M \lNorm{\vd_{k}}{\vx_{k}} t_{k}^{u}} + (t - t_{k}^{u}) L \Norm{\vd_{k}}^{2}, & t > t_{k}^{u}.
    \end{cases}
\end{equation}
Recall that the step size 
\begin{equation}\label{eq:tk-appendix}
    t_{k} = \frac{\eta_{k}}{\lNorm{\vd_{k}}{\vx_{k}} (1 + M \eta_{k})}
\end{equation}
for adaptive BFGS~(Algorithm~\ref{alg:ada_bfgs}) corresponds to the minimizer of
\begin{align*}
    H_k(t)= f(\vx_{k}) + t \vg_{k}^{\T} \vd_{k} + \dfrac{\omega_{*}(M t \lNorm{\vd_{k}}{\vx_{k}})}{M^{2}}, 
\end{align*}
which holds that
\begin{equation}\label{eq:Hkt_first_order} 
    H'_k(t) = \frac{\lNorm{\vd_{k}}{\vx_{k}}^{2} t}{1 - M \lNorm{\vd_{k}}{\vx_{k}} t} 
\end{equation}
for $t\in [0, 1/ (M \lNorm{\vd_{k}}{\vx_{k}}))$.

Now, we prove that the point $\ttt_k$ is the minimizer of $\tH_k(t)$ by showing $\tH_k(t)$ is strictly convex and $\tH'_k(\ttt_k)=0$.

\paragraph{The convexity of $\tH_k$}
We can verify that the second derivative of $\tH_{k}(t)$ is given by
\begin{equation}\label{eq:tHkt_second_order}
    \tH''_{k}(t) = \begin{cases}
        \dfrac{\lNorm{\vd_{k}}{\vx_{k}}^{2}}{(1 - M \lNorm{\vd_{k}}{\vx_{k}} t)^{2}}, & 0 \leq t \leq t_{k}^{u}, \\[0.2cm]
        L \Norm{\vd}_{k}^{2}, & t > t_{k}^{u},
    \end{cases}
\end{equation}
which implies $\tH''_{k}(t) > 0$ for all $u \in [0, +\infty)$.
Therefore, the function $\tH_{k}(t)$ is convex on $[0, +\infty)$.

\paragraph{The derivative of $\tH_k(\cdot)$ at $\ttt_k$}

According to \cref{def:alpha-appendix,def:_t_u-appendix}, we can simplify $t_{k}^{u}$ as 
\begin{equation}\label{def:t_u-appendix}
    t_{k}^{u} = \frac{1 - \alpha_{k}}{M \lNorm{\vd_{k}}{\vx_{k}}}.
\end{equation}
We now analyze two cases in the definition of $\ttt_k$ in \eqref{eq:ttk-appendix} as follows.
\begin{enumerate}
    \item In the case of $(1 + M \eta_{k}) \alpha_{k} \leq 1$, we have
        \begin{equation}\label{eq:tk_l_ku}
        \begin{aligned}
            \ttt_{k} & ~\overset{\mathclap{\eqref{eq:ttk-appendix}}}{=}~ \frac{1}{M \lNorm{\vd_{k}}{\vx_{k}}} \frac{M \eta_{k}}{1 + M \eta_{k}} \\
            & ~=~ \frac{1}{M \lNorm{\vd_{k}}{\vx_{k}}} \left( 1 - \frac{1}{1 + M \eta_{k}} \right) \\
            & ~\leq~ \frac{1 - \alpha_{k}}{M \lNorm{\vd_{k}}{\vx_{k}}} \\
            & ~\overset{\mathclap{\eqref{def:t_u-appendix}}}{=}~ t_{k}^{u},
        \end{aligned}
        \end{equation}
        where the the inequality use the condition that $(1 + M \eta_{k}) \alpha_{k} \leq 1$.
        Therefore, substituting the above expression of $\ttt_{k}$ into \eqref{eq:tHkt_first_order} yields
        \begin{equation*}
        \begin{aligned}
            \tH'_{k}(\ttt_{k}) & ~\overset{\mathclap{\eqref{eq:tHkt_first_order}}}{=}~ \vg_{k}^{\T} \vd_{k} + \frac{\lNorm{\vd_{k}}{\vx_{k}}^{2} \ttt_{k}}{1 - M \lNorm{\vd_{k}}{\vx_{k}} \ttt_{k}} \\
            & ~\overset{\mathclap{\eqref{eq:ttk-appendix}}}{=}~ \vg_{k}^{\T} \vd_{k} + \frac{\lNorm{\vd_{k}}{\vx_{k}} \eta_{k} / (1 + M \eta_{k})}{1 - M \eta_{k} / (1 + M \eta_{k})} \\
            & ~=~ \vg_{k}^{\T} \vd_{k} + \lNorm{\vd_{k}}{\vx_{k}} \eta_{k} \\
            & ~\overset{\mathclap{\eqref{def:eta-appendix}}}{=}~ 0,
        \end{aligned}
        \end{equation*}
        where the first equality relies on the \eqref{eq:tk_l_ku}.

    \item In the case of $(1 + M \eta_{k}) \alpha_{k} > 1$, we have
        \begin{equation*}
        \begin{aligned}
            \ttt_{k} & ~\overset{\mathclap{\eqref{eq:ttk-appendix}}}{=}~ \frac{M \eta_{k} \alpha_{k}^{2} + (1 - \alpha_{k})^{2}}{M \lNorm{\vd_{k}}{\vx_{k}}}
            = \frac{(1 + M \eta_{k}) \alpha_{k}^{2} + 1 - 2 \alpha_{k}}{M \lNorm{\vd_{k}}{\vx_{k}}} \\
            & ~>~ \frac{\alpha_{k} + 1 - 2 \alpha_{k}}{M \lNorm{\vd_{k}}{\vx_{k}}}
            = \frac{1 - \alpha_{k}}{M \lNorm{\vd_{k}}{\vx_{k}}} 
            \overset{\eqref{def:t_u-appendix}}{=} t_{k}^{u},
        \end{aligned}
        \end{equation*}
        where the inequality use the condition that $(1 + M \eta_{k}) \alpha_{k} >  1$.
        Substituting the above expression of $\ttt_{k}$ into \eqref{eq:tHkt_first_order}, we obtain
        \begin{equation*}
        \begin{aligned}
            \tH'_{k}(\ttt_{k}) & ~\overset{\mathclap{\eqref{eq:tHkt_first_order}}}{=}~ \vg_{k}^{\T} \vd_{k} + \dfrac{\lNorm{\vd_{k}}{\vx_{k}}^{2} t_{k}^{u}}{1 - M \lNorm{\vd_{k}}{\vx_{k}} t_{k}^{u}} + (\ttt_{k} - t_{k}^{u}) L \Norm{\vd_{k}}^{2} \\
            & ~\overset{\mathclap{\eqref{eq:ttk-appendix}}}{=}~ \vg_{k}^{\T} \vd_{k} + \frac{(1 - \alpha_{k}) \lNorm{\vd_{k}}{\vx_{k}} / M}{1 - (1 - \alpha_{k})} + \frac{M \eta_{k} \alpha_{k}^{2} + (1 - \alpha_{k})^{2} - (1 - \alpha_{k})}{M \lNorm{\vd_{k}}{\vx_{k}}} L \Norm{\vd_{k}}^{2} \\
            & ~=~ \vg_{k}^{\T} \vd_{k} + \frac{\lNorm{\vd_{k}}{\vx_{k}}}{M} \left( \frac{1}{\alpha_{k}} - 1 \right) + \frac{\lNorm{\vd_{k}}{\vx_{k}}}{M} \frac{L \Norm{\vd_{k}}^{2}}{ \lNorm{\vd_{k}}{\vx_{k}}} \left( (1 + M \eta_{k}) \alpha_{k}^{2} - \alpha_{k} \right) \\
            & ~\overset{\mathclap{\eqref{def:alpha-appendix}}}{=}~ \vg_{k}^{\T} \vd_{k} + \frac{\lNorm{\vd_{k}}{\vx_{k}}}{M} \left( \frac{1}{\alpha_{k}} - 1 \right) + \frac{\lNorm{\vd_{k}}{\vx_{k}}}{M} \frac{(1 + M \eta_{k}) \alpha_{k}^{2} - \alpha_{k}}{\alpha_{k}^{2}} \\
            & ~=~ \vg_{k}^{\T} \vd_{k} + \frac{\lNorm{\vd_{k}}{\vx_{k}}}{M} \left( \frac{1}{\alpha_{k}} - 1 + 1 + M \eta_{k} - \frac{1}{\alpha_{k}} \right) \\
            & ~=~ \vg_{k}^{\T} \vd_{k} + \lNorm{\vd_{k}}{\vx_{k}} \eta_{k} \\
            & ~\overset{\mathclap{\eqref{def:eta-appendix}}}{=}~ 0,
        \end{aligned}
        \end{equation*}
        where the first equality use the condition that $\ttt_{k} > t_{k}^{u}$.
\end{enumerate}
Combining above two cases, we conclude that $\tH'_{k}(\ttt_{k}) = 0$. Thus we complete the proof.
\end{proof}

\subsection{Proof of Lemma~\ref{lem:ada_bfgs_improved_properties}}\label{proof:ada_bfgs_improved_properties}

For the proof of Lemma~\ref{lem:ada_bfgs_improved_properties}, we first introduce two auxiliary functions 
\begin{equation*}
    h_{k}(t) \defeq f(\vx_{k}) + t \vg_{k}^{\T} \vd_{k} + \frac{\omega( M t \lNorm{\vd_{k}}{\vx_{k}})}{M^{2}}, 
\end{equation*}
and
\begin{equation*}
    \thh_{k}(t) \defeq \begin{cases}
        f(\vx_{k}) + t \vg_{k}^{\T} \vd_{k} + \dfrac{\omega( M t \lNorm{\vd_{k}}{\vx_{k}})}{M^{2}} , & 0 \leq t \leq t_{k}^{l}, \\[0.3cm]
        f(\vx_{k}) + t \vg_{k}^{\T} \vd_{k} + \dfrac{\omega( M t \lNorm{\vd_{k}}{\vx_{k}})}{M^{2}} + \dfrac{t_{k}^{l} (t - t_{k}^{l}) \lNorm{\vd}{\vx}^{2} }{1 + M t_{k}^{l} \lNorm{\vd_{k}}{\vx_{k}}} + \dfrac{1}{2} \mu(t - t_{k}^{l})^{2} \Norm{\vd}^{2}, & t > t_{k}^{l},
    \end{cases}
\end{equation*}
where $t_{k}^{l}$ follows the definition in \eqref{eq:def_t_l_u}.
We can verify that the derivatives of $h_{k}(t)$ and $\thh_{k}(t)$ are
\begin{equation}\label{eq:hkt_first_order}
    h'_{k}(t) = \vg_{k}^{\T}\vd_{k} + \frac{t \lNorm{\vd_{k}}{\vx_{k}}^{2}}{1 + M t\lNorm{\vd_{k}}{\vx_{k}}}, 
\end{equation}
and
\begin{equation}\label{eq:thkt_first_order}
    \thh'_{k}(t) = \begin{cases}
        \vg_{k}^{\T}\vd_{k} + \dfrac{t \lNorm{\vd_{k}}{\vx_{k}}^{2}}{1 + M t\lNorm{\vd_{k}}{\vx_{k}}}, & 0 \leq t \leq t_{k}^{l}, \\[0.3cm]
        \vg_{k}^{\T}\vd_{k}  + \dfrac{t^{l} \lNorm{\vd_{k}}{\vx_{k}}^{2} }{1 + M t^{l} \lNorm{\vd_{k}}{\vx_{k}}} + \mu (t - t_{k}^{l}) \Norm{\vd_{k}}^{2}, & t > t_{k}^{l}, 
    \end{cases}
\end{equation}
respectively.

We then provide properties for functions $h'_k(t)$, $\thh'_k(t)$, $H'_k(t)$, and $\tH'_k(t)$ as follows.

\begin{lem}\label{lem:ada_bfgs_improved_properties_inter}
    The functions $h'_{k}(t)$ and $\thh'_{k}(t)$ defined in \cref{eq:hkt_first_order,eq:thkt_first_order} hold
    \begin{align}
        & \thh'_{k}(t) \geq \max\{h'_{k}(t), \vg_{k}^{\T} \vd_{k} + t \mu \Norm{\vd_{k}}^{2}\} \label{eq:aux_h_ineq_1} 
    \end{align}
    for all $t \in [0, +\infty)$.
    Additionally, the functions $H'_{k}(t)$ and defined in \cref{eq:Hkt_first_order,eq:tHkt_first_order} hold
    \begin{equation}\label{eq:aux_H_ineq_1}
        \tH'_{k}(t) \leq H'_{k}(t) \quad \text{for all } t \in \left[0, \frac{1}{M \lNorm{\vd_{k}}{\vx_{k}}}\right), 
    \end{equation}
    and
    \begin{equation}\label{eq:aux_H_ineq_2} 
        \tH'_{k}(t) \leq \vg_{k}^{\T} \vd_{k} + t L \Norm{\vd_{k}}^{2} \quad \text{for all } t \in [0, +\infty).
    \end{equation}   
    Finally, the step sizes $t_{k}$ and $\ttt_{k}$ defined in \cref{eq:tk-appendix,eq:ttk-appendix} hold
    \begin{align}
        \tH'_{k}(\ttt_{k}) = H'_{k}(t_{k}) = 0, \label{eq:def_t}
    \end{align}
    and
    \begin{align}
        \ttt_{k} \geq t_{k}. \label{eq:t_ineq}
    \end{align}
\end{lem}
\begin{proof}
    We can verify that the second derivatives of $h_{k}(t)$ and $\thh_{k}(t)$ have the form of
    \begin{equation}\label{eq:hkt_second_order}
        h''_{k}(t) = \frac{\lNorm{\vd_{k}}{\vx_{k}}^{2}}{(1 + M t \lNorm{\vd_{k}}{\vx_{k}})^{2}},
    \end{equation}
    and
    \begin{equation}\label{eq:thkt_second_order}
        \thh''_{k}(t) = \begin{cases}
            \dfrac{\lNorm{\vd_{k}}{\vx_{k}}^{2}}{(1 + M \lNorm{\vd_{k}}{\vx_{k}} t)^{2}}, & 0 \leq t \leq t_{k}^{l}, \\[0.25cm]
            \mu \Norm{\vd_{k}}^{2}, & t > t_{k}^{l}.
        \end{cases}
    \end{equation}
    Since $h''_{k}(t)$ is strictly decreasing on $[0, +\infty)$ and $\lNorm{\vd}{\vx}^{2} / (1 + M \lNorm{\vd}{\vx} t_{k}^{l})^{2} = \mu \Norm{\vd}^{2}$, we can write $\thh''_{k}(t)$ as
    \begin{equation}\label{eq:_thkt_second_order}
        \thh''_{k}(t) = \max\left\{ h''_{k}(t), \mu \Norm{\vd_{k}}^{2} \right\},
    \end{equation}
    which implies
    \begin{align*}
        \thh'_{k}(t) 
        = \thh'_{k}(0) + \int_{0}^{t} \thh''_{k}(s) \, \mathrm{d} s   
        = \vg_{k}^{\T} \vd_{k} + \int_{0}^{t} \thh''_{k}(s) \, \mathrm{d} s 
        \overset{\eqref{eq:_thkt_second_order}}{\geq} & \vg_{k}^{\T} \vd_{k} + \int_{0}^{t} h''_{k}(s) \, \mathrm{d} s 
        = \thh'_{k}(0) + \int_{0}^{t} h''_{k}(s) \, \mathrm{d} s 
        = h'_{k}(t),
    \end{align*}
    and
    \begin{align*}
        \thh'_{k}(t) 
        = \thh'_{k}(0) + \int_{0}^{t} \thh''_{k}(s) \, \mathrm{d} s 
        = \vg_{k}^{\T} \vd_{k} + \int_{0}^{t} \thh''_{k}(s) \, \mathrm{d} s 
        \overset{\eqref{eq:_thkt_second_order}}{\geq}  \vg_{k}^{\T} \vd_{k} + \int_{0}^{t} \mu \Norm{\vd_{k}}^{2} \, \mathrm{d} s 
        = \vg_{k}^{\T} \vd_{k} + t \mu \Norm{\vd_{k}}^{2}.
    \end{align*}
    Combining above two inequalities leads to \eqref{eq:aux_h_ineq_1}.
    The results of \cref{eq:aux_H_ineq_1,eq:aux_H_ineq_2} can be achieved by a similar derivation to \eqref{eq:aux_h_ineq_1}, so that we omit the details.

    Recall that Section~\ref{proof:ada_stepsize_improved} shows that $\ttt_{k}$ is the minimizer of $\tH_{k}(t)$ and \citet{gaoQuasiNewtonMethodsSuperlinear2019} shows that $t_{k}$ is the minimizer of $H_{k}(t)$. 
    Therefore, we have $H'_{k}(t_{k}) = 0$ and $\tH'_{k}(\ttt_0) = 0$.

    Furthermore, combining \cref{eq:aux_H_ineq_1,eq:def_t}, we have 
    \begin{equation*}
        \tH'_{k}(\ttt_{k}) \overset{\eqref{eq:def_t}}{=} H'_{k}(t_{k}) \overset{\eqref{eq:aux_H_ineq_1}}{\geq} \tH'_{k}(t_{k}),
    \end{equation*}
    where the second inequality is due to $t_{k} \in [0, 1 / (M \lNorm{\vd_{k}}{\vx_{k}}))$. 
    Thus, the strictly increasing property of $\tH'_{k}(t)$ implies $\ttt_{k} \geq t_{k}$.
\end{proof}

We now prove Lemma~\ref{lem:ada_bfgs_improved_properties} as follows.

\begin{proof}
    According to Lemma~\ref{lem:self_con_bound_improved} with $\vx = \vx_{k}$ and $\vd = \vd_{k}$ and Lemma~\ref{lem:ada_bfgs_improved_properties_inter}, we have
    \begin{equation}\label{eq:curvature-1-appendix}
        \vg_{k + 1}^{\T} \vd_{k} = \nabla f(\vx_{k} + \ttt_{k} \vd_{k}) ^{\T} \vd_{k} \overset{\eqref{eq:self_con_grad_lower_improved}, \, \eqref{eq:thkt_first_order}}{\geq} \thh'_{k}(\ttt_{k}) \overset{\eqref{eq:aux_h_ineq_1}}{\geq} h'_{k}(\ttt_{k}) \overset{\eqref{eq:t_ineq}}{\geq} h'_{k}(t_{k}) = \frac{2 M \eta_{k}}{1 + 2 M \eta_{k}} \vg_{k}^{\T} \vd_{k},
    \end{equation}
    where the last inequality is based on the strictly increasing property of $\thh'_k(t)$.
    
    Combining \eqref{eq:def_t} and \eqref{eq:aux_H_ineq_2}, we have
    \begin{equation*}
        \vg_{k}^{\T} \vd_{k} + \ttt_{k} L \Norm{\vd_{k}}^{2} \geq \tH'_{k}(\ttt_{k}) = 0,
    \end{equation*}
    which implies
    \begin{equation}\label{eq:t_ineq_2}
        \ttt_{k} \geq - \vg_{k}^{\T} \vd_{k} / (L \Norm{\vd_{k}}^{2}).
    \end{equation}
    Again using Lemma~\ref{lem:self_con_bound_improved} with $\vx = \vx_{k}$ and $\vd = \vd_{k}$ and Lemma~\ref{lem:ada_bfgs_improved_properties_inter}, we have
    \begin{equation}\label{eq:curvature-2-appendix}
    \begin{aligned}
        \vg_{k + 1}^{\T} \vd_{k} & ~=~ \nabla f(\vx_{k} + \ttt_{k} \vd_{k}) ^{\T} \vd_{k} \overset{\eqref{eq:self_con_grad_lower_improved}, \, \eqref{eq:thkt_first_order}}{\geq} \thh'_{k}(\ttt_{k}) \\
        & ~\overset{\mathclap{\eqref{eq:aux_h_ineq_1}}}{\geq}~  \vg_{k}^{\T} \vd_{k} + \ttt_{k} \mu \Norm{\vd_{k}}^{2} \overset{\eqref{eq:t_ineq_2}}{\geq} \vg_{k}^{\T} \vd_{k} + \frac{- \vg_{k}^{\T} \vd_{k}}{L \Norm{\vd_{k}}^{2}} \mu \Norm{\vd_{k}}^{2} = \left(1 - \frac{1}{\varkappa}\right) \vg_{k}^{\T} \vd_{k}.
    \end{aligned}
    \end{equation}
    Combining \cref{eq:curvature-1-appendix,eq:curvature-2-appendix}, we achieve \eqref{eq:curvature_condition_improved}.

    According to Lemma~\ref{lem:self_con_bound_improved} with $\vx = \vx_{k}$ and $\vd = \vd_{k}$, Lemma~\ref{lem:ada_bfgs_improved_properties_inter}, and the definition of $\tH_{k}$ in \eqref{eq:tHkt}, we have
    \begin{equation}
    \begin{aligned}
        f(\vx_{k + 1}) & ~=~  f(\vx_{k} + \ttt_{k} \vd_{k})
        \overset{\eqref{eq:self_con_f_upper_improved}, \, \eqref{eq:tHkt}}{\leq} \tH_{k}(\ttt_{k}) 
        = \tH_{k}(0) + \int_{0}^{\ttt_{k}} \tH_{k}'(s) \, \mathrm{d} s \\
        & ~=~  f(\vx_{k}) + \int_{0}^{\ttt_{k}} \tH_{k}'(s) \, \mathrm{d} s 
        \overset{\eqref{eq:t_ineq}}{\leq}  f(\vx_{k}) + \int_{0}^{t_{k}} \tH_{k}'(s) \, \mathrm{d} s \\
        & ~\overset{\mathclap{\eqref{eq:aux_H_ineq_1}}}{\leq}~ f(\vx_{k}) + \int_{0}^{t_{k}} H_{k}'(s) \, \mathrm{d} s = H_{k}(0) +  \int_{0}^{t_{k}} H_{k}'(s) \, \mathrm{d} s = H_{k}(t_{k}) \\
        & ~=~ f(\vx_{k}) - \frac{\omega(M \eta_{k})}{M^{2}},
    \end{aligned}
    \end{equation}
    where the second inequality holds because $\tH_{k}'(t)$ is negative for all $t \in [0, \ttt_{k})$ and $\ttt_{k} \geq t_{k}$.
    Hence, we have proved \eqref{eq:self_con_f_dec_improved}.

    To prove \eqref{eq:armijo_condition_improved}, we first demonstrate that $\tH'_{k}(t)$ is convex on the interval $[0, +\infty)$. It suffices to show that for all $s_{1}, s_{2} \in [0, +\infty)$ with $s_{1} \neq s_{2}$~\cite[Theorem~2.1.3]{nesterovLecturesConvexOptimization2018}, we have
    \begin{equation*}
        \bigl(\tH''_{k}(s_{2}) - \tH''_{k}(s_{1})\bigr) (s_{2} - s_{1}) \geq 0.
    \end{equation*}
    Therefore, we only need to prove that $\tH''_{k}(t)$ is increasing on $[0, +\infty)$, which can be directly verified by \eqref{eq:tHkt_second_order}.
    Therefore, we have
    \begin{equation*}
    \begin{aligned}
        f(\vx_{k + 1}) & = f(\vx_{k} + \ttt_{k} \vd_{k})
        \overset{\eqref{eq:self_con_f_upper_improved}, \, \eqref{eq:tHkt}}{\leq} \tH_{k}(\ttt_{k})
        = f(\vx_{k}) + \int_{0}^{\ttt_{k}} \tH'(s) \, \mathrm{d} s \\
        & = f(\vx_{k}) + \int_{0}^{\ttt_{k}} \tH'(s) \, \mathrm{d} s 
        = f(\vx_{k}) + \ttt_{k} \int_{0}^{1} \tH'(\tau \ttt_{k}) \, \mathrm{d} \tau \\
        & \leq f(\vx_{k}) + \ttt_{k} \int_{0}^{1} \bigl((1 - \tau) \tH'(0) + \tau \tH'(\ttt_{k})\bigr) \, \mathrm{d} \tau \\
        & = f(\vx_{k}) + \frac{1}{2} \ttt_{k} \bigl(\tH'(0) + \tH'(\ttt_{k})\bigr) \\
        & = f(\vx_{k}) + \frac{1}{2} \ttt_{k}\vg_{k}^{\T} \vd_{k} \\
        & = f(\vx_{k}) + \frac{1}{2} \vg_{k}^{\T} (\vx_{k + 1} - \vx_{k}),
    \end{aligned}
    \end{equation*}
    where the fourth equality holds by the variable substitution $s = \tau \ttt_{k}$, the second inequality is based on the convexity of $\tH'_{k}(t)$, 
    and the sixth equality uses the facts that $\tH'_{k}(0) = \vg_{k}^{\T} \vd_{k}$ and $\tH'_{k}(\ttt_{k}) = 0$. 
    This proves \eqref{eq:armijo_condition_improved}. 
\end{proof}

\subsection{The Upper Bound of Constant \texorpdfstring{$E$}{E} in the Proof of Theorem~\ref{thm:linear_phase_linear_rates_2_improved}}\label{proof:constant_E}

Based on the definition in \eqref{def:E}, we have
\begin{equation*}
    E = \int_{0}^{+\infty} \frac{\ln (1 + y)}{y(1 + y)} \, \mathrm{d} y = \underbrace{\int_{0}^{1} \frac{\ln (1 + y)}{y(1 + y)} \, \mathrm{d} y}_{E_1} + \underbrace{\int_{1}^{+\infty} \frac{\ln (1 + y)}{y(1 + y)} \, \mathrm{d} y}_{E_2}.
\end{equation*}
For the term $E_1$, we have
\begin{equation*}
    E_1 = \int_{0}^{1} \frac{\ln (1 + y)}{y(1 + y)} \, \mathrm{d} y \leq \int_{0}^{1} \frac{1}{1 + y} \, \mathrm{d} y = \ln 2,
\end{equation*}
where we use the fact that $\ln(1 + y) < y$ for all $y > 0$.

For the term $E_{2}$, we first show that 
\begin{align}\label{eq:lny}
\ln (1 + y) \leq \sqrt{y}     
\end{align}
for all $y \geq 1$. This can be achieved by
\begin{equation*}
    (\ln (1 + y) - \sqrt{y}) \bigr|_{y = 1} = \ln 2 - 1 < 0,
\end{equation*}
and
\begin{equation*}
    \diff{\left( \ln(1 + y) - \sqrt{y} \right)}{y} = \frac{1}{1 + y} - \frac{1}{2 \sqrt{y}} = - \frac{(\sqrt{y} - 1)^{2}}{2 \sqrt{y} (1 + y)} \leq 0
\end{equation*}
for all $y \geq 1$. 
Therefore, we have
\begin{equation*}
    E_{2} = \int_{1}^{+\infty} \frac{\ln (1 + y)}{y(1 + y)} \, \mathrm{d} y \overset{\eqref{eq:lny}}{\leq} \int_{1}^{+\infty} \frac{1}{\sqrt{y}(1 + y)} \, \mathrm{d} y = 2 \int_{1}^{+\infty} \frac{1}{1 + y} \mathrm{d} (\sqrt{y}) = 2 (\arctan{\sqrt{y}}) \bigr|_{y = 1}^{+\infty} = \frac{\pi}{2}.
\end{equation*}
Combining the above upper bounds for $E_1$ and $E_2$, we obtain $E \leq \ln 2 + \pi / 2 < 3$.

\bibliographystyle{plainnat}
\bibliography{reference}
\end{document}

%% file: math.tex
\usepackage{amsmath}\allowdisplaybreaks
\usepackage{amsfonts,bm}
\usepackage{amssymb}
\usepackage{mathtools}

\newcommand{\defeq}{\coloneqq}

\DeclarePairedDelimiter{\Norm}{\|}{\|}
\DeclarePairedDelimiterX{\lNorm}[2]{\|}{\|_{#2}}{#1}

\DeclarePairedDelimiter{\abs}{\lvert}{\rvert}

\newcommand{\T}{\intercal}

\DeclarePairedDelimiterX{\innerangle}[2]{\langle}{\rangle}{#1, #2}

\DeclarePairedDelimiter{\ceil}{\lceil}{\rceil}

\newcommand{\vzero}{\mathbf{0}}
\newcommand{\vone}{\mathbf{1}}

\newcommand{\va}{\mathbf{a}}

\newcommand{\vd}{\mathbf{d}}

\newcommand{\vg}{\mathbf{g}}

\newcommand{\vs}{\mathbf{s}}

\newcommand{\vx}{\mathbf{x}}
\newcommand{\vy}{\mathbf{y}}

\newcommand{\fI}{\mathcal{I}}

\newcommand{\fO}{\mathcal{O}}

\newcommand{\BN}{\mathbb{N}}

\newcommand{\BR}{\mathbb{R}}
\newcommand{\BS}{\mathbb{S}}

\newcommand{\mA}{\mathbf{A}}
\newcommand{\mB}{\mathbf{B}}

\newcommand{\mI}{\mathbf{I}}

\newcommand{\mP}{\mathbf{P}}

\DeclareMathOperator*{\argmin}{arg\,min}

\DeclareMathOperator{\Tr}{Tr}

\usepackage{amsthm}
\theoremstyle{plain}
\newtheorem{thm}{Theorem}

\newtheorem{lem}{Lemma}
\newtheorem{asm}{Assumption}
\newtheorem{remark}{Remark}

\newtheorem{prop}{Proposition}

\makeatletter
\newcommand{\Ddots}{\mathinner{\mkern1mu\raise\p@
\vbox{\kern7\p@\hbox{.}}\mkern2mu
\raise4\p@\hbox{.}\mkern2mu\raise7\p@\hbox{.}\mkern1mu}}
\makeatother

\makeatletter
\newcommand*{\rom}[1]{\expandafter\@slowromancap\romannumeral #1@}
\makeatother

\makeatletter
\newcommand{\biggg}{\bBigg@{3}}
\newcommand{\bigggl}{\mathopen\biggg}

\newcommand{\bigggr}{\mathclose\biggg}

\newcommand{\Biggg}{\bBigg@{3.5}}

\newcommand{\bigggg}{\bBigg@{4}}

\newcommand{\Bigggg}{\bBigg@{4.5}}

\makeatother